\documentclass[a4paper,fleqn]{cas-sc}

\usepackage[authoryear]{natbib}

\def\tsc#1{\csdef{#1}{\textsc{\lowercase{#1}}\xspace}}
\tsc{WGM}
\tsc{QE}

\usepackage{graphicx}
\usepackage{multirow}
\usepackage{amsmath, amssymb, amsfonts, amsthm}
\usepackage{mathrsfs}
\usepackage{xcolor}
\usepackage{textcomp}
\usepackage{manyfoot}
\usepackage{booktabs}
\usepackage{algorithm, algorithmicx, algpseudocode}
\usepackage{listings}
\usepackage{mathtools}
\usepackage{subcaption}
\usepackage{tikz}
\usetikzlibrary{arrows.meta}
\usepackage{tikz-cd}
\usepackage{wrapfig}
\usepackage{pgfplots}
\usepackage{bm}
\usepackage{todonotes}
\usepackage{hyperref}
\usepackage{bookmark}
\usepackage[most]{tcolorbox}
\usepackage[nameinlink]{cleveref}
\usepackage{float}
\usepackage{stackengine}

\hypersetup{
    colorlinks,
    linkcolor={red!50!black},
    citecolor={blue!50!black},
    urlcolor={blue!80!black}
}
\pgfplotsset{compat=1.18}
\newcommand{\bbA}{\mathbb{A}}
\newcommand{\bbB}{\mathbb{B}}
\newcommand{\bbC}{\mathbb{C}}
\newcommand{\bbD}{\mathbb{D}}
\newcommand{\bbE}{\mathbb{E}}

\newcommand{\bbG}{\mathbb{G}}
\newcommand{\bbH}{\mathbb{H}}

\newcommand{\bbM}{\mathbb{M}}
\newcommand{\bbN}{\mathbb{N}}

\newcommand{\bbP}{\mathbb{P}}

\newcommand{\bbR}{\mathbb{R}}

\newcommand{\bbT}{\mathbb{T}}
\newcommand{\bbU}{\mathbb{U}}

\newcommand{\bbW}{\mathbb{W}}

\newcommand{\bfA}{\mathbf{A}}  
\newcommand{\bfB}{\mathbf{B}}

\newcommand{\bfH}{\mathbf{H}}

\newcommand{\bfL}{\mathbf{L}}

\newcommand{\bfe}{\mathbf{e}}
\newcommand{\bff}{\mathbf{f}}

\newcommand{\bfh}{\mathbf{h}}

\newcommand{\bfn}{\mathbf{n}}

\newcommand{\bft}{\mathbf{t}}
\newcommand{\bfu}{\mathbf{u}}
\newcommand{\bfv}{\mathbf{v}}

\newcommand{\bfx}{\mathbf{x}}

\newcommand{\bfzeta}{\bm{\zeta}}
\newcommand{\bfeta}{\bm{\eta}}

\newcommand{\bflambda}{\bm{\lambda}}
\newcommand{\bfmu}{\bm{\mu}}
\newcommand{\bfnu}{\bm{\nu}}

\newcommand{\bfrho}{\bm{\rho}}

\newcommand{\bfsigma}{\bm{\sigma}}

\newcommand{\bfphi}{\bm{\phi}}

\newcommand{\bfchi}{\bm{\chi}}
\newcommand{\bfpsi}{\bm{\psi}}
\newcommand{\bfomega}{\bm{\omega}}

\newcommand{\bfzero}{\mathbf{0}}  

\newcommand{\calD}{\mathcal{D}}
\newcommand{\calE}{\mathcal{E}}
\newcommand{\calF}{\mathcal{F}}

\newcommand{\calH}{\mathcal{H}}

\newcommand{\calK}{\mathcal{K}}

\newcommand{\calT}{\mathcal{T}}

\newcommand{\rmA}{\mathrm{A}}  

\newcommand{\rmH}{\mathrm{H}}
\newcommand{\rmI}{\mathrm{I}}

\newcommand{\rmd}{\mathrm{d}}

\newcommand{\sfS}{\mathsf{S}}

\newcommand{\frakH}{\mathfrak{H}}

\let\Re\relax
\newcommand{\Re}{\mathrm{Re}}  

\let\div\relax
\DeclareMathOperator{\div}{div}

\let\curl\relax
\DeclareMathOperator{\curl}{curl}

\DeclareMathOperator{\grad}{grad}

\DeclareMathOperator{\sym}{sym}

\newcommand{\leftavg}{\{\!\!\{}
\newcommand{\rightavg}{\}\!\!\}}

\newcommand{\leftjump}{[\![}
\newcommand{\rightjump}{]\!]}

\newtheorem{theorem}{Theorem}[section]

\newtheorem{definition}[theorem]{Definition}

\newtheorem{lemma}[theorem]{Lemma}

\newtheorem{remark}[theorem]{Remark}
\newtheorem{remark*}{Remark}

\usepackage{etoolbox}
\AtBeginEnvironment{assumption} {\crefalias{theorem}{assumption}}
\AtBeginEnvironment{conjecture} {\crefalias{theorem}{conjecture}}
\AtBeginEnvironment{corollary}  {\crefalias{theorem}{corollary}}
\AtBeginEnvironment{definition} {\crefalias{theorem}{definition}}
\AtBeginEnvironment{example}    {\crefalias{theorem}{example}}
\AtBeginEnvironment{framework}  {\crefalias{theorem}{framework}}
\AtBeginEnvironment{lemma}      {\crefalias{theorem}{lemma}}
\AtBeginEnvironment{proposition}{\crefalias{theorem}{proposition}}
\AtBeginEnvironment{remark}     {\crefalias{theorem}{remark}}
\AtBeginEnvironment{result}     {\crefalias{theorem}{result}}

\crefname{assumption}{Assumption}{Assumptions}
\Crefname{assumption}{Assumption}{Assumptions}
\crefname{conjecture}{Conjecture}{Conjectures}
\Crefname{conjecture}{Conjecture}{Conjectures}
\crefname{corollary}{Corollary}{Corollaries}
\Crefname{corollary}{Corollary}{Corollaries}
\crefname{definition}{Definition}{Definitions}
\Crefname{definition}{Definition}{Definitions}
\crefname{example}{Example}{Examples}
\Crefname{example}{Example}{Examples}
\crefname{framework}{Framework}{Frameworks}
\Crefname{framework}{Framework}{Frameworks}
\crefname{lemma}{Lemma}{Lemmas}
\Crefname{lemma}{Lemma}{Lemmas}
\crefname{proposition}{Proposition}{Propositions}
\Crefname{proposition}{Proposition}{Propositions}
\crefname{remark}{Remark}{Remarks}
\Crefname{remark}{Remark}{Remarks}
\crefname{result}{Result}{Results}
\Crefname{result}{Result}{Results}

\makeatletter
\newif\if@eqrefsfirst
\DeclareRobustCommand{\@eqrefsitem}[1]{\if@eqrefsfirst\@eqrefsfirstfalse\else,~\fi\ref{#1}}
\DeclareRobustCommand{\eqrefs}[1]{\textup{\tagform@{\@eqrefsfirsttrue\forcsvlist\@eqrefsitem{#1}}}}
\makeatother

\definecolor{seaborngreen}{rgb}{0.3333333333333333, 0.6588235294117647, 0.40784313725490196}  
\definecolor{seaborncyan}{rgb}{0.39215686274509803, 0.7098039215686275, 0.803921568627451}  
\definecolor{seabornblue}{rgb}{0.2980392156862745, 0.4470588235294118, 0.6901960784313725}  
\definecolor{seabornpurple}{rgb}{0.5058823529411764, 0.4470588235294118, 0.6980392156862745}  
\definecolor{seabornred}{rgb}{0.7686274509803922, 0.3058823529411765, 0.3215686274509804}  
\definecolor{seabornorange}{rgb}{0.958, 0.476, 0.206}  
\definecolor{seabornsand}{rgb}{0.8, 0.7254901960784313, 0.4549019607843137}  
\definecolor{matinbg}{rgb}{0.75, 0.90, 1.0}      
\definecolor{matintext}{rgb}{0.098, 0.098, 0.44}  

\newif\ifdraft
\drafttrue
\ifdraft
  \newcommand{\bda}[1]{\todo[color=seabornpurple!40]{{\bf BDA:} #1}}
  \newcommand{\bdainline}[1]{\todo[color=seabornpurple!40,inline]{{\bf BDA:} #1}}
  \newcommand{\pef}[1]{\todo[color=seabornblue!40,inline]{{\bf PEF:} #1}}
  \newcommand{\pefnotinline}[1]{\todo[color=seabornblue!40]{{\bf PEF:} #1}}
  \newcommand{\matin}[1]{\todo[color=matinbg]{{\bf MS:} #1}}
  \newcommand{\matininline}[1]{\todo[color=matinbg,inline]{{\bf MS:} #1}}
  
\else
  \newcommand{\bda}[1]{}
  \newcommand{\bdainline}[1]{}
  \newcommand{\pef}[1]{}
  \newcommand{\pefnotinline}[1]{}
  \newcommand{\matin}[1]{}
  \newcommand{\matininline}[1]{}
  
\fi

\begin{document}
\let\WriteBookmarks\relax
\def\floatpagepagefraction{1}
\def\textpagefraction{.001}

\shorttitle{Enstrophy preservation in Navier--Stokes}    
\shortauthors{B.~D.~Andrews, M.~Shams, P.~E.~Farrell}  

\title [mode = title]{Strongly enstrophy-stable integrators for the incompressible Navier--Stokes equations}   
\tnotemark[1] 
\tnotetext[1]{This work was funded by
the European Union [ERC, GeoFEM, 101164551],
the Engineering and Physical Sciences Research Council (EPSRC) [grant number EP/W026163/1],
the Science and Technology Facilities Council [grant number UKRI/ST/B000495/1],
the EPSRC Energy Programme [grant number EP/W006839/1],
a CASE award from the UK Atomic Energy Authority,
the Donatio Universitatis Carolinae Chair ``Mathematical modelling of multicomponent systems'',
the UKRI Digital Research Infrastructure Programme through the Science and Technology Facilities Council's Computational Science Centre for Research Communities (CoSeC),
and
the National Science Foundation under grant no.~DMS-1929284 while in residence at the Institute for Computational and Experimental Research in Mathematics in Providence, USA.
The authors gratefully acknowledge the hospitality of the Erwin Schrödinger International Institute for Mathematics and Physics, Vienna, where part of this work was carried out.
Views and opinions expressed are, however, those of the authors only and do not necessarily reflect those of the European Union or the European Research Council. Neither the European Union nor the granting authority can be held responsible for them.
For the purpose of open access, the authors have applied a CC BY public copyright licence to any author accepted manuscript arising from this submission. No new data were generated or analysed during this study. The code used to produce the numerical results is openly available \citep{Andrews_2026_GitHub}.}

\author[1]{Boris D.~Andrews}[orcid=0000-0003-1769-3432]
\cormark[1]
\ead{boris.andrews@maths.ox.ac.uk}
\credit{Conceptualization, Software, Investigation, Visualization, Writing -- original draft, Writing -- review \& editing}

\author[1]{Matin Shams}[orcid=0009-0004-2597-1564]
\ead{matinshams81@gmail.com}
\credit{Software, Investigation, Visualization, Writing -- review \& editing}

\author[1,2]{Patrick E.~Farrell}[orcid=0000-0002-1241-7060]
\ead{patrick.farrell@maths.ox.ac.uk}
\credit{Supervision, Funding acquisition, Investigation, Writing -- review \& editing}


\affiliation[1]{organization={Mathematical Institute, University of Oxford}, 
            city={Oxford},
            country={United Kingdom}}

\affiliation[2]{organization={Mathematical Institute, Charles University}, 
            city={Prague},
            country={Czech Republic}}

\cortext[1]{Corresponding author}

\begin{abstract}
We propose a mixed finite element discretisation for the incompressible Navier--Stokes equations that preserves the evolution laws of both energy and enstrophy, in a stronger sense than previous discretisations.
In two dimensions, the evolution law for enstrophy only permits dissipation for thermodynamically isolated systems, leading to a Reynolds-number-independent bound on the velocity gradient that naturally stabilises the scheme, even on severely under-resolved meshes.
In three dimensions, the scheme preserves both dissipation and the generation of enstrophy through vortex stretching.
We enforce these evolution laws by systematically introducing auxiliary variables into the discretisation.
While conforming implementations of these schemes require discrete Stokes complexes with enhanced regularity, we introduce both (i)~equivalent reparametrisations and (ii)~penalty formulations that require only the typical curl- and div-conforming spaces from the standard discrete de Rham complex.
The scheme handles different types of boundary conditions and curved domains.
The robust stabilisation properties of the proposed scheme are demonstrated through numerical simulations of a shear flow, a spherical vortex, and flow past an obstacle.
We observe numerically that preserving the discrete evolution of enstrophy in this way has a strong stabilising effect on the numerical solution, especially in two dimensions.
\end{abstract}


\maketitle
\color{black}


We consider the incompressible Navier--Stokes equations over a bounded Lipschitz domain $\Omega \subset \bbR^n$ for $n \in \{2,3\}$.
At a given time $t$, the strong form of the equations seeks $(\bfu(t), p(t)) : \Omega \to \bbR^n \times \bbR$ such that
\begin{equation}\label{eq:navier-stokes_1}
    \partial_t {\bfu}  =  \bfu \times \curl \bfu - \nabla \left[ \, \frac{1}{2}|\bfu|^2 + p \, \right] + \frac{2}{\Re}\div\varepsilon\,\bfu,  \qquad
    0  =  \div \bfu.
\end{equation}
We use the rotational or Lamb form $\bfu \cdot \nabla\bfu = - \, \bfu\times\curl\bfu + \nabla[\frac{1}{2} |\bfu|^2]$ for the advective term.
The operator $\varepsilon$ denotes the symmetric gradient $\varepsilon \, \bfu \coloneqq \sym \grad \bfu$, and $\Re$ the Reynolds number.

\paragraph{Energy.}

Under suitable ({thermodynamically isolated}) boundary conditions, testing {\eqref{eq:navier-stokes_1}} against $\bfu$ identifies the energy dissipation law:
\begin{equation}\label{eq:energy}
    \calK(\bfu)  \coloneqq  \frac{1}{2}\|\bfu\|^2,  \qquad
    \partial_t \calK  =  - \, \frac{2}{\Re} \|\varepsilon \, \bfu\|^2  \le  0,
\end{equation}
where $\|\cdot\|$ denotes the $L^2(\Omega)$ norm.
Assuming the initial data have bounded energy $\calK(\bfu(0))$, integration implies the bound on the time integral of enstrophy
\begin{equation}\label{eq:energy_bound}
    2 \int_{t=0}^T \|\varepsilon \, \bfu\|^2  \le  \Re \, \calK(\bfu(0)),
\end{equation}
for any final time $T \in (0, \infty)$ up to which a solution exists.
This in turn implies a bound for $\|\nabla \bfu\|$ by Korn's inequality, under certain boundary conditions.

The behaviour of $\|\nabla \bfu\|$ in time $t$ is central to the analysis of the Navier--Stokes equations.
An $L^4$-in-time bound on $\|\nabla \bfu\|$ is sufficient to guarantee uniqueness for weak solutions \citep[Sec.~4(c)]{Robinson_2020}\footnote{
    The precise Prodi--Serrin uniqueness condition under consideration here would typically be expressed as $\bfu \in L^4(0, T; \bfL^6)$ \citep{Prodi_1959,Serrin_1963}.
    This gradient form is sufficient by the continuous embedding $\bfH^1 \hookrightarrow \bfL^6$ \citep{BeiraoDaVeiga_1995}.
}, while a stronger, uniform ($L^\infty$-in-time) bound quantifies the transition from a weak to strong solution \citep[Sec.~5(b)]{Robinson_2020}.
On the discrete level, a bound on $\|\nabla \bfu\|$ implies that continuous finite elements may sensibly be applied to discretise the velocity.
However, the bound from energy dissipation \eqref{eq:energy_bound} is both (i)~$L^2$-in-time only, and (ii)~not $\Re$-robust, i.e.~it scales with the Reynolds number $\Re$.
The former presents an issue for the use of this bound in the analysis.
The latter implies that for large $\Re$, this bound loses its utility, leading to instability in numerical solutions (see \Cref{sec:simulations} below);
this issue is particularly clear in the Euler case $\Re = \infty$, where the bound on $\|\nabla \bfu\|$ is lost entirely.

\paragraph{Enstrophy.}

By testing \eqref{eq:navier-stokes_1} against $\Delta \bfu$, one identifies a second structure of interest, an evolution law for the enstrophy\footnote{
    The enstrophy is often defined as the $\bfL^2$ norm of the vorticity, in particular when working with a stream function formulation.
    The two definitions are equivalent only under certain boundary conditions (see \Cref{sec:scheme}).
    Defined as the $\bfL^2$ norm of the symmetric gradient, i.e.~the natural Dirichlet energy, the enstrophy continues to satisfy a dissipation structure even under more general boundary conditions (see \Cref{sec:bcs}).
}:
\begin{subequations}\label{eq:enstrophy_both}
\begin{equation}\label{eq:enstrophy}
    \calE(\bfu)  \coloneqq  \|\varepsilon \, \bfu\|^2,  \qquad
    \partial_t \calE  =  \int_\Omega \bfu \cdot (\curl \bfu \times \curl^2 \bfu) - \frac{1}{\Re} \|\Delta\bfu\|^2.
\end{equation}
Again, if Korn's inequality is available, a finite enstrophy implies a bound for $\|\nabla\bfu\|$, and bounds on the enstrophy are useful for proving the existence of unique strong solutions \citep{Robinson_2020}.
Unlike the energy-derived bound \eqref{eq:energy_bound}, which is only $L^2$-in-time and scales with $\Re$, a uniform-in-time bound on $\calE$ would yield a uniform-in-time bound on $\|\nabla\bfu\|$.
As discussed below, in 2D the advective term $\bfu \cdot (\curl \bfu \times \curl^2 \bfu)$ in \eqref{eq:enstrophy} vanishes identically, so $\calE$ is non-increasing for thermodynamically isolated systems and hence bounded by $\calE(\bfu(0))$ independently of $\Re$.

In 3D, the advective term in \eqref{eq:enstrophy} does not cancel in general, and the evolution of enstrophy is more complex than a dissipation law.
With appropriate boundary conditions, this identity can be manipulated into the classical \emph{vortex stretching} form,
\begin{equation}\label{eq:enstrophy_stretching_continuous}
    \partial_t \calE  =  \int_\Omega \curl\bfu \cdot \varepsilon\,\bfu \cdot \curl\bfu - \frac{1}{\Re} \|\Delta\bfu\|^2.
\end{equation}
\end{subequations}
Vortex stretching (or conversely \emph{vortex compression}) is the mechanism by which enstrophy is generated (or dissipated) in three dimensions (\citealp[Sec.~5.4]{Doering_Gibbon_1995}; \citealp[Sec.~7.2]{Majda_Bertozzi_2002}).
A vortex tube aligned with a positive eigendirection of the strain rate tensor $\varepsilon\,\bfu$ causes an increase in the enstrophy $\calE$.
Physically, the vortex tube is elongated by the flow, increasing in magnitude in the process due to conservation of angular momentum (\Cref{fig:vortex_stretching}).
This is quantified by the vortex stretching term $\curl\bfu \cdot \varepsilon\,\bfu \cdot \curl\bfu$ in \eqref{eq:enstrophy_stretching_continuous}, positive precisely when $\curl\bfu$ lies predominantly along stretching directions of $\varepsilon\,\bfu$\footnote{
    Its absence in two dimensions is immediate:
    there $\curl\bfu$ is orthogonal to the plane in which $\varepsilon\,\bfu$ acts, so $\curl\bfu \cdot \varepsilon\,\bfu \cdot \curl\bfu$ vanishes identically, and the enstrophy can only dissipate.
    This distinguishes the behaviour of the two- and three-dimensional cases.
}.

\begin{figure}[pos=!ht]
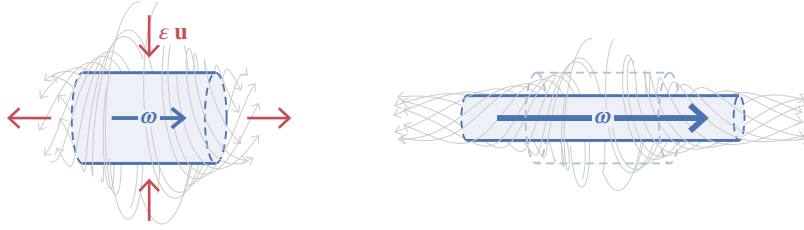

    \centering

    \caption{
        Vortex stretching.
        \emph{Left}: a vortex tube carrying vorticity $\bfomega = \curl\bfu$ along its axis (blue) in a flow field $\bfu$ (grey) that is extensional along that axis and compressive transverse to it (red).
        \emph{Right}: the tube is elongated and narrowed (its original extent dashed).
        Conservation of angular momentum intensifies $\bfomega$, and the enstrophy $\calE$ grows.
        The rate of this production is $\bfomega \cdot \varepsilon\,\bfu \cdot \bfomega$, positive precisely when $\bfomega$ aligns with a positive eigendirection of $\varepsilon\,\bfu$.
        Compare with \citet[Fig.~1]{OpenAI_2026b}.
    }\label{fig:vortex_stretching}
\end{figure}

The advective production in \eqref{eq:enstrophy_stretching_continuous} may then itself be bounded in terms of $\calE$, yielding local-in-time existence of unique strong solutions;
moreover, for sufficiently small initial enstrophy and $\Re$, the advective production cannot overcome the dissipation, and global existence is guaranteed \citep{Robinson_2020}.
Away from these regimes, enstrophy production from vortex stretching is not known to be controlled in general.
Certain 3D geometries can instigate regions of vortex compression ($\curl\bfu \cdot \varepsilon\,\bfu \cdot \curl\bfu < 0$, see \Cref{sec:vortex_street_3d}) stabilising the flow, while recent work indicates that regions of vortex stretching ($\curl\bfu \cdot \varepsilon\,\bfu \cdot \curl\bfu > 0$) can create singular blow-up in finite time \citep{OpenAI_2026a}, in particular in the presence of forcing \citep{Cordoba_MartinezZoroa_2023,Cordoba_MartinezZoroa_Zheng_2026,Alpoge_Buckmaster_2026,OpenAI_2026b}.

\begin{remark*}[Topological interpretation of enstrophy generation]
    More generally, whenever $\curl \bfu \times \curl^2 \bfu = \nabla Q$ for some $Q$, the advective term in \eqref{eq:enstrophy} evaluates to zero (up to boundary conditions)\footnote{
        The same is true when $\bfu \times \curl \bfu = \nabla Q$ for some $Q$.
        However, in such settings the advective term in the momentum equation \eqref{eq:navier-stokes_1} has no contribution to the dynamics ($Q$ is simply the Bernoulli pressure), so the fact that it would not contribute to the evolution of enstrophy is trivial.
    }.
    The topology of such vorticity fields $\curl \bfu$ in 3D is characterised by Arnold's structure theorem \citep[Sec~II.1]{Arnold_Khesin_2008}:
    level sets $\Sigma$ of non-constant $Q$ must form tori, with $\curl \bfu$ and $\curl^2 \bfu$ tangent to $\Sigma$.
    This poses a strong restriction on the topological complexity of the vortex field lines;
    we might loosely say that for flows with less complex topologies, we expect the enstrophy production from advection to be smaller.
\end{remark*}


As a final note on the importance of the enstrophy $\calE$, by writing the dissipation of energy $\calK$ \eqref{eq:energy} in terms of the enstrophy $\partial_t \calK = - \tfrac{2}{\Re}\calE$, we see that the behaviours of the two quantities are closely linked.
A bounded enstrophy implies bounded dissipation in the energy;
blow-ups in the enstrophy imply blow-ups in the dissipation of energy.

\paragraph{Overview.}

In this work, we propose a mixed finite element discretisation of the Navier--Stokes equations \eqref{eq:navier-stokes_1} that preserves discrete versions of the evolution laws of energy \eqref{eq:energy} and enstrophy \eqref{eq:enstrophy_both}.
We refer to these properties as energy and enstrophy stability, respectively.
In 2D, this gives necessary dissipation of the enstrophy $\calE$.
In 3D, we preserve a discrete form of the vortex stretching identity \eqref{eq:enstrophy_stretching_continuous}.

Our intention is that by preserving the behaviour of enstrophy $\calE$, the arguments above may be transferred to the discrete level, including (i)~improved stability at large $\Re$, in particular in 2D where we have a $\Re$-robust bound on $\|\nabla \bfu\|$, (ii)~avoidance of anomalous energy dissipation, and (iii)~a better reproduction of known dynamics and analytic results in 3D.
The novel application of ideas from structure-preservation and finite element exterior calculus to the areas of turbulence and stabilisation represents the main takeaway of this work.

\subsection*{Related literature}

\paragraph{Enstrophy preservation.}

Preserving enstrophy on discretisation has been of pressing interest for decades.
One of the pioneering works in structure-preserving discretisations was the finite difference scheme of \citet{arakawa1966} that conserves both total kinetic energy and total enstrophy for an inviscid streamfunction--vorticity formulation in 2D. This scheme was later shown to be exactly equivalent to a finite element method on Cartesian grids \citep{jespersen1974} and generalised to arbitrary meshes by \citet{salmon1989}.

For thermodynamically isolated systems in 2D, the Euler equations can be viewed as a Hamiltonian system on the vorticity;
all moments of the vorticity $\omega$, including the enstrophy as the second moment, are Casimirs of this system and therefore conserved.
An alternative approach to enstrophy preservation beyond finite elements is then to discretise this Hamiltonian system in such a way as to conserve all Casimirs.
This is possible through a quantisation of the space of divergence-free fields through a discrete Lie algebra (that of skew-Hermitian matrices) preserving the Poisson bracket as a Lie bracket (the commutator);
the continuous Hamiltonian Euler system then reduces to a discrete Lax pair system.
Such a quantisation was done by \citet{Zeitlin_1991,Zeitlin_2004} over periodic domains (i.e.~the flat torus) and the sphere, respectively, in the latter case using a quantisation proposed by \citet{Hoppe_Yau_1998}.
As the truncation level is increased, the discrete Casimirs conserved by the matrix system (i.e.~the moments of the skew-Hermitian matrix) converge to the continuous Casimirs (i.e.~the moments of the vorticity).
To discretise this system in time while preserving all Casimirs (including the discrete enstrophy) \citet{Modin_Viviani_2020a,Modin_Viviani_2020b} and \citet{Cifani_Viviani_Modin_2023} have developed and implemented a class of isospectral symplectic Runge--Kutta methods.

For viscous flows, mixed stream function--vorticity discretisations for the incompressible Navier--Stokes equations were proposed by \citet{Liu_E_2001} in both 2D and 3D.
In the 2D case, this was recently studied by \citet{Lombardi_Pagliantini_2025}, interpreting the Navier--Stokes equations as a Hamiltonian system with dissipation on the vorticity $\omega$.
Using a discrete subcomplex of the de Rham complex, this scheme may be reparametrised in mixed vorticity--velocity--pressure form (\citealp{Dubois_Salaun_Salmon_2003b,Dubois_Salaun_Salmon_2003a,Palha_Gerristma_2017}; \citealp[Sec.~7.2]{Cotter_2023}; \citealp{Hanot_2023}).
We adopt the terminology of \citet{Palha_Gerristma_2017} and refer to the vorticity--velocity--pressure formulation of this scheme as the mass-, energy-, enstrophy-, vorticity-conserving (MEEVC) scheme \eqref{eq:meevc_semidiscrete}.

Both MEEVC and the scheme we propose \eqrefs{eq:semidiscrete_3d,eq:semidiscrete_2d} introduce an auxiliary vorticity variable $\bfomega$. The major distinction between them is in its definition.
Whereas the MEEVC scheme defines $\bfomega$ to be a projection of $\curl\bfu$ in the $L^2$ norm \eqref{eq:meevc_semidiscrete_a}, our definition is similar to (but not exactly\footnote{
    The discrepancy is caused by technicalities of the boundary conditions, and, in 3D, the vorticity only being divergence-free in a weak discrete sense.
    For the full details of this projection, see the scheme \eqrefs{eq:semidiscrete_3d,eq:semidiscrete_2d} presented below in \Cref{sec:scheme}.
}) a projection in the $H^1$ seminorm.
This different definition for the auxiliary vorticity is indicated by the general framework presented by a subset of the authors for the construction of structure-preserving integrators \citep{Andrews_Farrell_2025a,Andrews_Farrell_2025b}, and has substantial consequences for the sense in which enstrophy is preserved.
In 2D, the MEEVC scheme dissipates an auxiliary enstrophy $\tilde{\calE}(\omega) \coloneqq \frac{1}{2}\|\omega\|^2$ defined on the auxiliary vorticity $\omega$;
this differs from the enstrophy dissipation we show for our 2D integrator \eqref{eq:discrete_stability}, which holds for the \emph{primal} enstrophy $\calE(\bfu) \coloneqq \|\varepsilon\,\bfu\|^2$ defined on the primal velocity $\bfu$.
Numerical comparisons of our scheme with MEEVC in \Cref{sec:simulations} indicate that the two schemes differ substantially, with this stronger primal sense of enstrophy stability offering clear advantages in certain cases.

\citet{Palha_Gerristma_2017} proposed a dual-field velocity--vorticity discretisation for the 2D incompressible Navier--Stokes equations.
The authors also observe the stabilisation properties of their scheme on under-resolved meshes, using as a numerical demonstration the roll-up of a shear layer with $\Re = \infty$. One potential downside of this scheme is that the two approximations of the velocity, $\bfu_1$ and $\bfu_2$, satisfy independent evolution equations and can drift apart over long times.

An alternative approach to enstrophy stability in the incompressible Navier--Stokes equations comes in formulating the momentum equation in terms of the vorticity.
Here, one discretises
\begin{equation}\label{eq:vorticity}
    \partial_t \bfomega  =  \bfomega\cdot\nabla\bfu - \bfu\cdot\nabla\bfomega + \frac{1}{\Re}\Delta\bfomega,
\end{equation}
solving for the vorticity $\bfomega$ as the primal variable.
With the enstrophy defined as the squared norm of the vorticity $\tilde{\calE}(\bfomega) \coloneqq \frac{1}{2}\|\bfomega\|^2$, preservation of the evolution of enstrophy is possible without the need for auxiliary variables, assuming an appropriate handling of the nonlinear advective term \citep[Sec.~3.2]{Charnyi_et_al_2017}.
This, however, leaves two choices for how the velocity $\bfu$ is defined:
(i)~by specifying $\bfu$ in advance of solving the discretised vorticity equation \eqref{eq:vorticity}, e.g.~by precomputing a distinct discrete Navier--Stokes solution, in which case it is difficult to guarantee that $\curl\bfu$ and $\bfomega$ remain close over time, or (ii)~by coupling the vorticity equation \eqref{eq:vorticity} with an auxiliary equation solving for discrete $\bfu$ such that $\curl\bfu \approx \bfomega$, in which case the difficulty then lies not in preserving enstrophy stability, but in preserving energy stability.

When considering spectral methods on thermodynamically isolated systems, the spectral basis is typically constructed to be closed under differentiation, a property which does not hold for general finite element methods.
A consequence of this is that the evolution of many quadratic quantities of interest, including the energy, helicity, and in particular enstrophy, is preserved at the semidiscrete level without modification \citep{kreeft2013,palha2014,Palha_Gerristma_2017}.

\paragraph{Stabilisation.}

The typical approach in the literature for stabilisation of $\bfH^1$-conforming Navier--Stokes discretisations is not enstrophy preservation but the introduction of an artificial viscosity/dissipation.
Spectral vanishing viscosity methods, such as those proposed by \citet{tadmor1989} and \citet{Maday_Tadmor_1989}, allow the viscous term to act selectively on the high-frequency modes of the discrete solutions.
The motivation for this approach lies in how typical schemes accumulate energy at finer length scales.
With appropriate tuning, \citet{Tadmor_1990,Tadmor_1993} showed this to converge to the unique entropy solution for scalar conservation laws.
However, the selective damping of high-order modes is known to reduce the effective resolution of the discrete solution;
the approximation power of the highest-order polynomials is degraded.

Continuous interior penalty methods, proposed by \citet{Douglas_Dupont_1976}, address the fine-scale oscillations by penalising the jump in the solution gradient across facets.
These methods have been analysed for the advection--diffusion equation, Oseen's equation, and the Navier--Stokes equations at high $\Re$ \citep{Burman_Hansbo_2004,Burman_Fernandez_Hansbo_2006,Burman_Fernandez_2007}.
At moderate polynomial orders, e.g.~$p \approx 3$ (where there are few high-order terms to penalise), such penalty methods have been observed to outperform vanishing viscosity methods \citep{Moura_et_al_2022}.
However, as with spectral vanishing viscosity methods, the introduction of artificial dissipation renders these schemes irreversible in the Euler case $\Re = \infty$.

Discontinuous Galerkin methods \citep{Cockburn_Karniadakis_Shu_2000} offer intrinsic upwind stabilisation, which is biased towards these finer scales, as discussed by \citet{Moura_Sherwin_Peiro_2015,Moura_et_al_2017}.
The stability properties of discontinuous Galerkin methods are similar to those of continuous interior penalty methods.

\paragraph{Discrete Stokes complexes.}

Conforming discretisations of our proposed schemes \eqrefs{eq:semidiscrete_3d,eq:semidiscrete_2d} require knowledge and an implementation of finite element Stokes complexes, studied in the field of finite element exterior calculus \citep{Arnold_Falk_Winther_2006,Arnold_Falk_Winther_2009,Arnold_2018}.
In 2D, such finite element complexes are well known.
For example, over Alfeld-split meshes \citep{Alfeld_1984}, the Scott--Vogelius complex \citep{Scott_Vogelius_1985a, Scott_Vogelius_1985b} at degree 2 takes the preceding $H^2$-conforming space in the Stokes complex to be the degree-3 Hsieh--Clough--Tocher space \citep{Clough_Tocher_1965}.
This is relatively well supported in finite element software, e.g.~in \texttt{GetFEM} \citep{Renard_Poulios_2020}, \texttt{FreeFEM++} \citep{Hecht_2012}, \texttt{libMesh} \citep{Stogner_Carey_2007}, and \texttt{Firedrake} through \texttt{FIAT} \citep{Brubeck_Kirby_2026}.

Finite element Stokes complexes in 3D are less common.
The first finite element Stokes complexes on arbitrary tetrahedral meshes were proposed by \citet{Neilan_2015}, and further generalised by \citet{Chen_Huang_2024} to de Rham complexes of arbitrary smoothness.
Over split meshes, the minimal polynomial degree can again be reduced.
We refer the reader to the works of \citet{Fu_Guzman_Neilan_2020}, \citet{Hu_Zhang_Zhang_2022} and \citet{Brubeck_Liang_Parker_2026} for such complexes on Alfeld splits, and to \citet{Guzman_Lischke_Neilan_2022} for complexes on Worsey--Farin splits.
As far as we are aware, none of the above proposed $\bfH(\grad\curl)$-conforming spaces are yet supported in any publicly available general-purpose finite element software.

Non-conforming elements offer an alternative approach, particularly relevant on domains with re-entrant corners, which can cause blow-ups in the velocity gradient.
A commonly considered non-conforming 2D Stokes complex is constructed with the Morley element \citep{Morley_1968} and the vector-valued Crouzeix--Raviart element \citep{Crouzeix_Raviart_1973}, both relatively widely implemented, and concludes with the degree-$0$ DG space \citep{Gillette_Hu_Zhang_2020}.
However, this complex is inappropriate for our scheme \eqrefs{eq:semidiscrete_3d,eq:semidiscrete_2d} as the Crouzeix--Raviart space loses coercivity under the norm of the symmetric gradient;
there is no discrete Korn inequality in the vector-valued Crouzeix--Raviart space.
An alternative non-conforming complex that may be used for our discretisation is that constructed from the Mardal--Tai--Winther element \citep{Mardal_Tai_Winther_2002} for the velocity space \citep{Tai_Winther_2006}.
In the context of biharmonic-like equations, \citet{Ainsworth_Parker_2024a, Ainsworth_Parker_2024b} propose reparametrisation systems by which one may effectively use $H^2$-conforming scalar finite elements while only requiring finite elements with at most $H^1$ conformity to be implemented.
The schemes in \Cref{sec:charlie} below are closely related to and heavily inspired by their construction.

\subsection*{Overview}

In \Cref{sec:notation}, we introduce some general notation and preliminary material, including the continuous and discrete Stokes complexes employed throughout this manuscript.
In \Cref{sec:scheme}, we present our discretisation \eqrefs{eq:semidiscrete_3d,eq:semidiscrete_2d}, and show it to preserve discrete forms of the energy and enstrophy stability results \eqrefs{eq:energy,eq:enstrophy}.
In \Cref{sec:implementation}, we consider options for the implementation of the scheme that circumvent the need for an implemented discrete Stokes complex.
In \Cref{sec:bcs}, we discuss how different boundary conditions may be incorporated into the scheme while preserving the stability properties.
In \Cref{sec:simulations}, we test the scheme numerically on various benchmark problems, demonstrating the stabilising properties of our discretisation.
In \Cref{sec:conclusions}, we conclude with discussions of open questions and future work.

\section{Notation \& preliminaries}\label{sec:notation}

Let $\bfn$ denote the outward-pointing unit normal on the boundary $\partial\Omega$ of the domain $\Omega$, and $\bft$ an arbitrary unit tangential vector, i.e.~one satisfying $\bfn\cdot\bft = 0$.
For a general differential operator $\rmd$ over $\Omega$, we define Hilbert spaces $H(\rmd)$ and $H^k(\rmd)$,
\begin{equation}
    H(\rmd)    \coloneqq  \{u \in L^2 : \rmd u \in L^2\},  \qquad
    H^k(\rmd)  \coloneqq  \{u \in H^k : \rmd u \in H^k\}.
\end{equation}
We use boldface to emphasise vector-valued fields and spaces, e.g.~$\bfu \in \bfH^1 = [H^1]^n$.
For $k \in \bbN$, denote $H^k \coloneqq H(\grad^k)$.

\subsection{Stokes complexes in 3D \& 2D}

In 3D, two common notions exist for the Stokes complex\footnote{
    We note the space $\bfH^1(\curl)$ is distinct from $\bfH(\grad \curl)$, with $\bfH^1(\curl) \subsetneq \bfH(\grad \curl)$.
    This distinction is most clear in the finite element setting:
    normal continuity is necessary for $\bfH^1(\curl)$-conformity, but not for $\bfH(\grad \curl)$-conformity.
}:
\begin{equation}
\begin{tikzcd}
    H^2 \arrow{r}{\grad}
        & \bfH^1(\curl) \arrow{r}{\curl}
        & \bfH^1 \arrow{r}{\div}
        & L^2
\end{tikzcd},\qquad
\begin{tikzcd}
    H^1 \arrow{r}{\grad}
        & \bfH(\grad \curl) \arrow{r}{\curl}
        & \bfH^1 \arrow{r}{\div}
        & L^2
\end{tikzcd}.
\end{equation}
We consider the latter.
The construction of our scheme assumes the existence of a discrete Stokes complex,
\begin{subequations}\label{eq:stokes_3d}
\begin{equation}
\begin{tikzcd}
    H^1 \arrow{r}{\grad}
        & \bfH(\grad \curl) \arrow{r}{\curl}
        & \bfH^1 \arrow{r}{\div}
        & L^2  \\
    \bbA \arrow{r}{\grad} \arrow[u]
        & \bbW \arrow{r}{\curl} \arrow[u]
        & \bbU \arrow{r}{\div} \arrow[u]
        & \bbP \arrow[u]
\end{tikzcd},
\end{equation}
typically a finite element subcomplex.
The vertical arrows denote inclusion.
Alongside it, we consider the subcomplex obtained by imposing boundary conditions strongly,
\begin{equation}\label{eq:stokes_3d_bc}
\begin{tikzcd}
    H^1_0 \arrow{r}{\grad}
        & \bfH(\grad \curl) \cap \bfH_0(\curl) \arrow{r}{\curl}
        & \bfH^1 \cap \bfH_0(\div) \arrow{r}{\div}
        & L^2  \\
    \bbA_0 \arrow{r}{\grad} \arrow[u]
        & \bbW_0 \arrow{r}{\curl} \arrow[u]
        & \bbU_0 \arrow{r}{\div} \arrow[u]
        & \bbP \arrow[u]
\end{tikzcd},
\end{equation}
\end{subequations}
where we denote
\begin{subequations}
\begin{align}
    H^1_0  &\coloneqq  \{u \in H^1 : u = 0 \text{ on } \partial\Omega\},  \\
    \bfH_0(\curl)  &\coloneqq  \{\bfu \in \bfH(\curl) : \bfu \times \bfn = \bfzero \text{ on } \partial\Omega\},  \\
    \bfH_0(\div)  &\coloneqq  \{\bfu \in \bfH(\div) : \bfu \cdot \bfn = 0 \text{ on } \partial\Omega\}.
\end{align}
\end{subequations}

In 2D, let $^\perp$ denote a rotation of a vector $(u, v)^\perp \coloneqq (-v, u)$;
let $\nabla^\perp = \grad^\perp$ denote the rotated gradient $\nabla^\perp \phi \coloneqq (- \partial\phi/\partial y, \partial\phi/\partial x)$.
The 3D vector-to-vector operator $\bfu \mapsto \curl\bfu$ has two 2D analogues: (i)~scalar-to-vector $u \mapsto - \nabla^\perp u$, and (ii)~vector-to-scalar $\bfu \mapsto - \div\bfu^\perp$.
We assume the existence of a similar discrete Stokes complex,
\begin{subequations}\label{eq:stokes_2d}
\begin{equation}
\begin{tikzcd}
    H^2 \arrow{r}{\grad^\perp}
        & \bfH^1 \arrow{r}{\div}
        & L^2  \\
    \bbW \arrow{r}{\grad^\perp} \arrow[u]
        & \bbU \arrow{r}{\div} \arrow[u]
        & \bbP \arrow[u]
\end{tikzcd},
\end{equation}
with the corresponding complex with strongly imposed boundary conditions,
\begin{equation}\label{eq:stokes_2d_bc}
\begin{tikzcd}
    H^2 \cap H^1_0 \arrow{r}{\grad^\perp}
        & \bfH^1 \cap \bfH_0(\div) \arrow{r}{\div}
        & L^2  \\
    \bbW_0 \arrow{r}{\grad^\perp} \arrow[u]
        & \bbU_0 \arrow{r}{\div} \arrow[u]
        & \bbP \arrow[u]
\end{tikzcd}.
\end{equation}
\end{subequations}

Neither complex is exact on general domains.
We write $\frakH^k$ for the space of \emph{discrete} harmonic $k$-forms of the unconstrained \emph{discrete} complex \eqrefs{eq:stokes_3d,eq:stokes_2d}, i.e.~those functions in the kernel of the outgoing map that are $L^2$-orthogonal to the image of the incoming one.
For example, $\frakH^2 \subset \bbU$ denotes those divergence-free functions that are $\bfL^2$-orthogonal to $\curl\bbW$.
We write $\frakH^k_0$ for the corresponding harmonic forms in the strongly constrained discrete complex \eqrefs{eq:stokes_3d_bc,eq:stokes_2d_bc}.
Exactness at the $k$-th slot is precisely the statement that $\frakH^k$ or $\frakH^k_0$ is trivial.
We write $b_k$ for the Betti numbers of $\Omega$,
\begin{equation}\label{eq:betti}
    b_k  =  \dim \frakH^k  =  \dim \frakH^{n-k}_0.
\end{equation}
The equality here holds by Poincar\'e--Lefschetz duality ($\frakH^k \cong \frakH^{n-k}_0$) \citep[Thm.~3.43]{Hatcher_2002}:
strongly imposed boundary conditions {reverse the harmonic forms}, moving them from one end of the complex to the other (up to isomorphism).

\subsection{Finite element spaces \& discrete complexes}

Let the domain $\Omega$ be partitioned into a mesh $\calT^h$.
Let $\calF^h$ denote the set of interior facets of $\calT^h$.
We write $\int_{\calT^h} \coloneqq \sum_{K \in \calT^h} \int_K$ and $\int_{\calF^h} \coloneqq \sum_{F \in \calF^h} \int_F$ for integrals broken over the cells and interior facets, respectively.
For a general function $\phi$, continuous in cell interiors, we denote by $\leftavg \phi \rightavg \coloneqq \frac{1}{2}(\phi_+ + \phi_-)$ and $\leftjump \phi \odot \bfn \rightjump \coloneqq \phi_+ \odot \bfn_+ + \phi_- \odot \bfn_-$ the average and jump operators over the facets $\calF^h$, respectively, where $_+$ and $_-$ denote the traces of $\phi$ from the two elements sharing the facet, and $\odot$ is a general bilinear operator.
Let $h$ denote the average of the cell diameters of two cells sharing a facet.

We define here the various finite element spaces that will be used throughout this manuscript, for some degree $k \in \mathbb{N}$:
\begin{subequations}
\begin{align*}
    \text{Discontinuous Galerkin / Lagrange:}\quad  &\bbD\bbG_k,  \\
    \text{Second-type N\'ed\'elec \citep{Nedelec_1986}:}\quad  &\bbN\bbE\bbD^{(2)}_k  \quad(=  [\bbD\bbG_k]^n \cap \bfH(\curl)),  \\
    \text{Brezzi--Douglas--Marini \citep{Brezzi_Douglas_Marini_1985}:}\quad  &\bbB\bbD\bbM_k  \quad(=  [\bbD\bbG_k]^n \cap \bfH(\div)),  \\
    \text{Continuous Galerkin / Lagrange:}\quad  &\bbC\bbG_k  \quad(=  \bbD\bbG_k \cap H^1).
\end{align*}
\end{subequations}
For $k \ge 1$, these constitute finite element de Rham complexes in 3D and 2D:
\begin{equation}\label{eq:fedr_poly}
\begin{tikzcd}
    H^1 \arrow{r}{\grad}
        & \bfH(\curl) \arrow{r}{\curl}
        & \bfH(\div) \arrow{r}{\div}
        & L^2  
        & H^1 \arrow{r}{\grad^\perp}
        & \bfH(\div) \arrow{r}{\div}
        & L^2  \\
    \bbC\bbG_{k+2} \arrow{r}{\grad} \arrow[u]
        & \bbN\bbE\bbD^{(2)}_{k+1} \arrow{r}{\curl} \arrow[u]
        & \bbB\bbD\bbM_k \arrow{r}{\div} \arrow[u]
        & \bbD\bbG_{k-1} \arrow[u]
        & \bbC\bbG_{k+1} \arrow{r}{\grad^\perp} \arrow[u]
        & \bbB\bbD\bbM_k \arrow{r}{\div} \arrow[u]
        & \bbD\bbG_{k-1} \arrow[u]
\end{tikzcd}.
\end{equation}
This 2D complex is used for the simulation in \Cref{sec:vortex_street_2d}.
The vector-valued spaces may be trimmed to give
\begin{subequations}
\begin{align*}
    \text{First-type N\'ed\'elec \citep{Nedelec_1986}:}\quad  &\bbN\bbE\bbD^{(1)}_k \quad(\subsetneq \bbN\bbE\bbD^{(2)}_k),  \\
    \text{Raviart--Thomas \citep{raviart1977}:}\quad  &\bbR\bbT_k \quad(\subsetneq \bbB\bbD\bbM_k),
\end{align*}
\end{subequations}
forming complexes for $k \ge 0$:
\begin{equation}\label{eq:fedr_trimmed}
\begin{tikzcd}
    H^1 \arrow{r}{\grad}
        & \bfH(\curl) \arrow{r}{\curl}
        & \bfH(\div) \arrow{r}{\div}
        & L^2  
        & H^1 \arrow{r}{\grad^\perp}
        & \bfH(\div) \arrow{r}{\div}
        & L^2  \\
    \bbC\bbG_{k+1} \arrow{r}{\grad} \arrow[u]
        & \bbN\bbE\bbD^{(1)}_{k+1} \arrow{r}{\curl} \arrow[u]
        & \bbR\bbT_{k+1} \arrow{r}{\div} \arrow[u]
        & \bbD\bbG_k \arrow[u]
        & \bbC\bbG_{k+1} \arrow{r}{\grad^\perp} \arrow[u]
        & \bbR\bbT_{k+1} \arrow{r}{\div} \arrow[u]
        & \bbD\bbG_k \arrow[u]
\end{tikzcd}.
\end{equation}
In \Cref{sec:penalty} below we propose a penalty method to avoid the need to implement a Stokes complex. In this case the velocity $\bfu$ is chosen to lie in the divergence-free subset of the $\bfH(\div)$-conforming space.
Divergence-free elements of $\bbB\bbD\bbM_k$ coincide precisely with divergence-free elements of $\bbR\bbT_{k+1}$;
consequently, the following forms a complex in 3D, combining the above \eqrefs{eq:fedr_poly,eq:fedr_trimmed} \citep[Sec.~7.5]{Arnold_2018}:
\begin{equation}\label{eq:fedr_hybrid}
\begin{tikzcd}
    H^1 \arrow{r}{\grad}
        & \bfH(\curl) \arrow{r}{\curl}
        & \bfH(\div) \arrow{r}{\div}
        & L^2  \\
    \bbC\bbG_{k+1} \arrow{r}{\grad} \arrow[u]
        & \bbN\bbE\bbD^{(1)}_{k+1} \arrow{r}{\curl} \arrow[u]
        & \bbB\bbD\bbM_{k} \arrow{r}{\div} \arrow[u]
        & \bbD\bbG_{k-1} \arrow[u]
\end{tikzcd}.
\end{equation}
The first two spaces, and the divergence-free subspace of the third, identify exactly with those of \eqref{eq:fedr_trimmed}.
3D discretisations in this complex therefore result in the same discrete velocity $\bfu$ as those on the trimmed polynomial complex \eqref{eq:fedr_trimmed}, but using fewer degrees of freedom.
We use this hybrid finite element de Rham complex for the simulations in \Cref{sec:vortex_street_3d}.

We use the superscript $^\rmA$ to denote macro-elements defined over Alfeld splits \citep{Alfeld_1984}, e.g.~$\bbC\bbG_k^\rmA$.
In 3D, the Scott--Vogelius pair \citep{Scott_Vogelius_1985a, Scott_Vogelius_1985b} discretises the tail of the Stokes complex \eqref{eq:stokes_3d} on Alfeld-split meshes for $k \ge 3$:
\begin{equation}\label{eq:sv_3d}
\begin{tikzcd}
    \cdots \arrow{r}{\curl}
        & \bfH^1 \arrow{r}{\div}
        & L^2  \\
    \cdots \arrow{r}{\curl}
        & \!\,[\bbC\bbG_k^\rmA]^3 \arrow{r}{\div} \arrow[u]
        & \bbD\bbG_{k-1}^\rmA \arrow[u]
\end{tikzcd}.
\end{equation}
We use this (alongside the polynomial de Rham complex above \eqref{eq:fedr_poly}) for the simulations in \Cref{sec:hill_vortex}.
In 2D, we also introduce the (degree-3) Hsieh--Clough--Tocher space, a macro-element over such split meshes:
\begin{equation*}
    \text{Hsieh--Clough--Tocher \citep{Clough_Tocher_1965}:}\quad  \bbH\bbC\bbT_3  \coloneqq  \bbD\bbG_3^\rmA \cap H^2.
\end{equation*}
This gives the following finite element Stokes complex \eqref{eq:stokes_2d}:
\begin{equation}\label{eq:sv_2d}
\begin{tikzcd}
    H^2 \arrow{r}{\grad^\perp}
        & \bfH^1 \arrow{r}{\div}
        & L^2  \\
    \bbH\bbC\bbT_3 \arrow{r}{\grad^\perp} \arrow[u]
        & \!\,[\bbC\bbG_2^\rmA]^2 \arrow{r}{\div} \arrow[u]
        & \bbD\bbG_1^\rmA \arrow[u]
\end{tikzcd}.
\end{equation}
We use this complex for the shear flow simulation in \Cref{sec:kh}.

\section{Energy- \& enstrophy-stable discretisation}\label{sec:scheme}

To introduce the scheme, we restrict our attention for now to the simplest case:
convex polytopal domains with free slip boundary conditions on $\partial\Omega$,
\begin{subequations}\label{eq:slip}
\begin{equation}\label{eq:slip_1}
    \bfu \cdot \bfn = 0,  \qquad
    \bft \cdot \varepsilon\,\bfu \cdot \bfn = 0.
\end{equation}
We will assume for now that there are no periodic boundaries, making the domain necessarily contractible by convexity.
Over such domains, these boundary conditions \eqref{eq:slip_1} may be equivalently formulated as
\begin{equation}\label{eq:slip_2}
    \bfu \cdot \bfn = 0,  \qquad
    \curl\bfu \times \bfn = \bfzero.
\end{equation}
\end{subequations}

\begin{remark}[General boundary conditions]
    A brief discussion of periodic boundary conditions can be found at the end of this section.
    More general boundary conditions, and non-convex and curved domains, are treated in \Cref{sec:bcs} below, with a further discussion of the equivalence between \eqref{eq:slip_1} and \eqref{eq:slip_2} in \Cref{sec:strong_vorticity}.
\end{remark}

To construct our scheme, we begin with the following 3-field (vorticity--velocity--pressure) formulation of the incompressible Navier--Stokes equations in strong form \eqref{eq:navier-stokes_1}:
\begin{subequations}\label{eq:navier-stokes_2}
\begin{align}
    \bfomega  &=  \curl \bfu,  \\
    \partial_t \bfu  &=  \bfu \times \bfomega - \nabla \left[ \, \frac{1}{2}|\bfu|^2 + p \, \right] + \frac{2}{\Re} \div \varepsilon\,\bfu,  \\
    0  &=  \div \bfu.
\end{align}
\end{subequations}
We consider the unforced system for clarity of exposition;
introducing a forcing term into the discretisation proposed below follows in the natural way\footnote{
    A body force $\bff$ enters the momentum equation as $(\bff, \bfv)$.
    Through the test functions used in the proof of \Cref{th:stability}, it contributes the terms $(\bff, \bfu)$ and $(\bff, \curl\bfomega)$ to the right-hand sides of the discrete energy and enstrophy laws \eqref{eq:discrete_stability} (likewise in \eqref{eq:enstrophy_stretching} and \eqref{eq:discrete_stability_bcs}).
}.
The tangential stress condition $\curl\bfu \times \bfn = \bfomega \times \bfn = \bfzero$ may be imposed strongly on $\bfomega$.
Rather than defining $\bfomega$ explicitly, we choose to define it implicitly through an auxiliary Stokes problem:
\begin{subequations}\label{eq:navier-stokes_3}
\begin{align}
    0  &=  \div \bfomega,  \label{eq:navier-stokes_3a}  \\
    \curl^2 \bfomega  &=  \curl^3 \bfu - \nabla \alpha,  \label{eq:navier-stokes_3b}  \\
    \partial_t \bfu  &=  \bfu \times \bfomega - \nabla \left[ \, \frac{1}{2}|\bfu|^2 + p \, \right] + \frac{2}{\Re} \div \varepsilon\,\bfu,  \label{eq:navier-stokes_3c}  \\
    0  &=  \div \bfu.  \label{eq:navier-stokes_3d}
\end{align}
\end{subequations}
Just as the pressure $p$ can be interpreted as a Lagrange multiplier enforcing the divergence-free constraint on $\bfu$, the auxiliary scalar field $\alpha$ can be interpreted as a Lagrange multiplier enforcing the divergence-free constraint on $\bfomega$.
Imposing the boundary condition $\alpha = 0$ ensures that the auxiliary Stokes problem \eqrefs{eq:navier-stokes_3a,eq:navier-stokes_3b} is well posed and has the unique strong-form solution\footnote{
    We show in the proof of \Cref{th:stability} that this boundary condition implies $\alpha = 0$ uniformly over $\Omega$.
    Nevertheless, it is required to ensure well-posedness of the weak formulation.
} $\bfomega = \curl\bfu$ for a given $\bfu$.
Lastly, we note
\begin{equation}
    \curl^3 \bfu
        = \curl \, [\curl^2 \bfu - \nabla \div \bfu]
        = \curl \, [- \Delta \bfu]
        = \curl \, [\nabla \div \bfu - 2 \div \varepsilon \, \bfu]
        = - \, 2 \curl \div \varepsilon \, \bfu
\end{equation}
where the first and final equalities hold by the complex property $\curl \grad = 0$, and the central two equalities hold as alternate formulations of the Laplacian.
We substitute $- \, 2 \curl \div \varepsilon \, \bfu$ for $\curl^3 \bfu$ in \eqref{eq:navier-stokes_3b} accordingly:
\begin{subequations}\label{eq:navier-stokes_4}
\begin{align}
    0  &=  \div \bfomega,  \label{eq:navier-stokes_4a}  \\
    \curl^2 \bfomega  &=  - \, 2 \curl \div \varepsilon \, \bfu - \nabla \alpha,  \label{eq:navier-stokes_4b}  \\
    \partial_t \bfu  &=  \bfu \times \bfomega - \nabla \left[ \, \frac{1}{2}|\bfu|^2 + p \, \right] + \frac{2}{\Re} \div \varepsilon\,\bfu,  \label{eq:navier-stokes_4c}  \\
    0  &=  \div \bfu.  \label{eq:navier-stokes_4d}
\end{align}
\end{subequations}

To derive our variational/weak form of \eqref{eq:navier-stokes_4}, and consequently our finite element semidiscretisation, we test each equation in $L^2$ against test functions in the spaces corresponding to $\alpha$, $\bfomega$, $\bfu$ and $p$, respectively.
Integrating by parts, each equation (except \eqref{eq:navier-stokes_4d}) picks up a boundary term:
\begin{subequations}\label{eq:boundary_terms}
\begin{align}
    0  &=  (\bfomega, \nabla \beta) - \langle \bfomega \cdot \bfn, \beta \rangle,  \label{eq:boundary_terms_a}  \\
    (\curl \bfomega, \curl \bfchi)  &=  \Big[ \, 2 (\varepsilon \, \bfu, \varepsilon \curl \bfchi) - (\nabla \alpha, \bfchi) \, \Big] - 2 \langle \varepsilon \, \bfu \cdot \bfn, \curl \bfchi \rangle,  \label{eq:boundary_terms_b}  \\
    (\partial_t \bfu, \bfv)  &=  \Big[ \, (\bfu \times \bfomega, \bfv) + \left( \frac{1}{2}|\bfu|^2 + p, \div \bfv \right) - \frac{2}{\Re}(\varepsilon\,\bfu, \varepsilon\,\bfv) \, \Big]  \notag  \\
    &\qquad\qquad\qquad\qquad\qquad\qquad\qquad\qquad + \Big[ - \left\langle \frac{1}{2}|\bfu|^2, \bfv \cdot \bfn \right\rangle + \left\langle \frac{2}{\Re} \varepsilon\,\bfu \cdot \bfn - p \bfn, \bfv \right\rangle \Big],  \label{eq:boundary_terms_c}  \\
    0  &=  (\div \bfu, q),  \label{eq:boundary_terms_d}
\end{align}
\end{subequations}
where $\langle \cdot, \cdot \rangle$ denotes the $L^2$ duality pairing on $\partial\Omega$.
In our proposed discretisation, each of the boundary terms vanishes, either by strong enforcement of the boundary condition on the test function, or via the free slip conditions \eqref{eq:slip} under consideration.
We thereby propose the following semidiscretisation.

\begin{definition}[3D semidiscretisation]
    Find $(\alpha, \bfomega, \bfu, p) \in \bbA_0 \times \bbW_0 \times \bbU_0 \times \bbP$ such that
    \begin{subequations}\label{eq:semidiscrete_3d}
    \begin{align}
        0  &=  (\bfomega, \nabla \beta),  \label{eq:semidiscrete_3d_a} \\
        (\curl \bfomega, \curl \bfchi)  &=  2 (\varepsilon \, \bfu, \varepsilon \curl \bfchi) - (\nabla \alpha, \bfchi),  \label{eq:semidiscrete_3d_b} \\
        (\partial_t \bfu, \bfv)  &=  (\bfu \times \bfomega, \bfv) + \left(\frac{1}{2}|\bfu|^2 + p, \div \bfv\right) - \frac{2}{\Re}(\varepsilon\,\bfu, \varepsilon\,\bfv),  \label{eq:semidiscrete_3d_c} \\
        0  &=  (\div \bfu, q),  \label{eq:semidiscrete_3d_d}
    \end{align}
    \end{subequations}
    for all $(\beta, \bfchi, \bfv, q) \in \bbA_0 \times \bbW_0 \times \bbU_0 \times \bbP$.
\end{definition}

We require the spaces $\bbA_0$, $\bbW_0$, $\bbU_0$, $\bbP$ to form a discrete Stokes complex \eqref{eq:stokes_3d_bc}.
This is a sufficient condition for the energy and enstrophy stability properties below \eqref{eq:discrete_stability} to hold.
In 2D the discretisation is largely similar, the two key differences being that (i)~the vorticity $\omega$ is a scalar, and (ii)~the Lagrange multiplier variable $\alpha$ is not required.

\begin{definition}[2D semidiscretisation]
    Find $(\omega, \bfu, p) \in \bbW_0 \times \bbU_0 \times \bbP$ such that
    \begin{subequations}\label{eq:semidiscrete_2d}
    \begin{align}
        (\nabla \omega, \nabla \chi)  &=  - \, 2 (\varepsilon \, \bfu, \varepsilon \, \nabla^\perp \chi),  \label{eq:semidiscrete_2d_a} \\
        (\partial_t \bfu, \bfv)  &=  - \, (\omega \bfu^\perp, \bfv) + \left(\frac{1}{2}|\bfu|^2 + p, \div \bfv\right) - \frac{2}{\Re}(\varepsilon\,\bfu, \varepsilon\,\bfv),  \label{eq:semidiscrete_2d_b} \\
        0  &=  (\div \bfu, q),
    \end{align}
    \end{subequations}
    for all $(\chi, \bfv, q) \in \bbW_0 \times \bbU_0 \times \bbP$.
\end{definition}

Again, the spaces $\bbW_0$, $\bbU_0$, $\bbP$ must form a discrete Stokes complex \eqref{eq:stokes_2d_bc}.

\subsection{Energy \& enstrophy stability}

\begin{theorem}[Energy \& enstrophy stability]\label{th:stability}
    The Navier--Stokes semidiscretisations in 3D \eqref{eq:semidiscrete_3d} and 2D \eqref{eq:semidiscrete_2d} preserve discrete forms of the energy dissipation \eqref{eq:energy} and enstrophy evolution \eqref{eq:enstrophy} laws:
    \begin{subequations}\label{eq:discrete_stability}
    \begin{align}
        \calK(\bfu)  &\coloneqq  \frac{1}{2}\|\bfu\|^2,  &
        \partial_t \calK  &=  - \, \frac{2}{\Re} \calE(\bfu)  \le  0,  \label{eq:energy_discrete}  \\
        \calE(\bfu)  &\coloneqq  \|\varepsilon \, \bfu\|^2,  &
        \partial_t \calE  &=  \int_\Omega \bfu \cdot (\bfomega \times \curl\bfomega) - \frac{1}{\Re} \|\curl\bfomega\|^2.  \label{eq:enstrophy_discrete}
    \end{align}
    \end{subequations}
    In 2D, the advective term $\int_\Omega \bfu \cdot (\bfomega \times \curl\bfomega)$ in the discrete enstrophy law \eqref{eq:enstrophy_discrete} vanishes, implying enstrophy $\calE$ is dissipated.
    In 3D, this can be identified with the classical vortex stretching form \eqref{eq:enstrophy_stretching_continuous} (see \Cref{th:stretching}).
\end{theorem}

\begin{proof}
    We restrict our attention to the 3D case \eqref{eq:semidiscrete_3d}, as the proofs in the 2D case are largely identical.
    For energy, consider $q = \div\bfu$ in \eqref{eq:semidiscrete_3d_d}, made possible by the complex property \eqref{eq:stokes_3d_bc}:
    \begin{equation}\label{eq:stability_proof_1}
        0  =  \|\!\div\bfu\|^2.
    \end{equation}
    Thus $\div\bfu = 0$ pointwise.
    Considering $\bfv = \bfu$ in \eqref{eq:semidiscrete_3d_c}, we use the divergence-free condition on $\bfu$ to observe the discrete energy stability result \eqref{eq:energy_discrete}:
    \begin{equation}
        \partial_t \calK
            =  (\partial_t \bfu, \bfu)
            =  - \, \frac{2}{\Re} \| \varepsilon\,\bfu \|^2
            =  - \, \frac{2}{\Re} \calE(\bfu).
    \end{equation}
    For enstrophy, consider $\bfv = \curl\bfomega$ in \eqref{eq:semidiscrete_3d_c}, similarly possible by the complex property \eqref{eq:stokes_3d_bc}:
    \begin{equation}\label{eq:stability_proof_2}
        (\partial_t \bfu, \curl\bfomega)
            = \int_\Omega \bfu \cdot (\bfomega \times \curl\bfomega) - \frac{2}{\Re} (\varepsilon\,\bfu, \varepsilon\curl\bfomega).
    \end{equation}
    Considering $\bfchi = \nabla\alpha$ in \eqref{eq:semidiscrete_3d_b}, again by the complex property \eqref{eq:stokes_3d_bc}, we observe
    \begin{equation}
        0  =  - \, \|\nabla\alpha\|^2,
    \end{equation}
    implying $\nabla\alpha = \bfzero$ pointwise, and consequently, by the zero boundary condition, $\alpha = 0$\footnote{
        Despite $\alpha$ necessarily being zero on the discrete level, $\alpha$ and $\beta$ cannot be removed from the discretisation \eqref{eq:semidiscrete_3d} without damaging its well-posedness through the introduction of a large nullspace into the auxiliary problem defining $\bfomega$ \eqref{eq:semidiscrete_3d_b}.
    }.
    Lastly, since (i)~$\bfu \in \bbU_0$ lies in the kernel of $\div$ \eqref{eq:stability_proof_1}, and (ii)~the domain $\Omega$ is assumed to be contractible, there exists a streamfunction $\bfpsi \in \bbW_0$ such that $\curl\bfpsi = \bfu$.
    Considering $\bfchi = \partial_t \bfpsi$ and $\bfchi = \bfomega$, respectively, in \eqref{eq:semidiscrete_3d_b}, we observe
    \begin{equation}\label{eq:stability_proof_3}
        (\curl\bfomega, \partial_t \bfu)  =  2 (\varepsilon \, \bfu, \varepsilon \, \partial_t \bfu)  =  \partial_t \calE,
        \qquad
        \| \curl\bfomega \|^2  =  2 (\varepsilon \, \bfu, \varepsilon \curl\bfomega),
    \end{equation}
    where in each equality we use the result $\alpha = 0$.
    Discrete enstrophy stability \eqref{eq:enstrophy_discrete} holds from \eqrefs{eq:stability_proof_2,eq:stability_proof_3}.
\end{proof}

These energy and enstrophy stability results \eqref{eq:discrete_stability} hold for semidiscretisations in space only;
their fully discrete forms after discretisation in time depend on the timestepping scheme used.
As these are quadratic quantities, however, any B-stable Runge--Kutta method, e.g.~Gau\ss{} (including implicit midpoint) and Radau-IIA (including implicit Euler) methods, will preserve the dissipation of energy $\calK$ and, in 2D, enstrophy $\calE$ \cite[Sec.~IV.12]{Hairer_Wanner_1996}.
Moreover in the inviscid limit $\Re \to \infty$, Gau\ss{} methods will preserve the conservation of energy and, in 2D, enstrophy while preserving only the advective contribution to enstrophy in 3D, due to their symplecticity.

\begin{remark}[Equivalent forms of the dissipation]\label{rem:dissipation_forms}
    In 3D, exactness of the discrete Stokes complex \eqref{eq:stokes_3d_bc} implies that for any divergence-free test function $\bfv \in \bbU_0$ there exists $\bfchi \in \bbW_0$ with $\curl\bfchi = \bfv$.
    Since $\nabla\alpha = \bfzero$ (see the proof of \Cref{th:stability}), \eqref{eq:semidiscrete_3d_b} then gives, for all divergence-free $\bfv \in \bbU_0$,
    \begin{equation}
        (\curl \bfomega, \bfv)  =  2 (\varepsilon \, \bfu, \varepsilon \, \bfv).
    \end{equation}
    The dissipation $\frac{2}{\Re}(\varepsilon\,\bfu, \varepsilon\,\bfv)$ in \eqref{eq:semidiscrete_3d_c} may therefore be equivalently replaced with $\frac{1}{\Re}(\curl\bfomega, \bfv)$, and similarly in 2D.
    We adopt the symmetric-gradient form $\frac{2}{\Re}(\varepsilon\,\bfu, \varepsilon\,\bfv)$ throughout, as we have observed advantages in the robustness of the Newton solver in numerical experiments.
\end{remark}

\subsection{Vortex stretching}\label{sec:stretching}

We show here how, in 3D, the discrete enstrophy evolution equation \eqref{eq:enstrophy_discrete} may be manipulated into the following discrete analogue of the vortex stretching form \eqref{eq:enstrophy_stretching_continuous}.

\begin{theorem}[Enstrophy stability: vortex stretching form]\label{th:stretching}
    In 3D, our proposed semidiscretisation \eqref{eq:semidiscrete_3d} satisfies
    \begin{equation}\label{eq:enstrophy_stretching}
        \partial_t \calE
            =  \left(\int_\Omega \bfomega \cdot \varepsilon\,\bfu \cdot \bfomega
            +  \frac{1}{2}\int_{\calF^h} \leftjump |\bfomega|^2 \, \bfn \rightjump \cdot \bfu
            -  \int_{\calT^h} \bfomega \cdot \nabla\big[\bfu\cdot\bfomega\big]\right)
            -  \frac{1}{\Re}\|\curl\bfomega\|^2.
    \end{equation}
\end{theorem}

\begin{proof}
    The result follows immediately from the discrete enstrophy stability \eqref{eq:enstrophy_discrete} established in \Cref{th:stability} and the technical lemma below (\Cref{lem:eq:stretching_identity}), which allows us to rewrite the advective term $\int_\Omega \bfu \cdot (\bfomega \times \curl\bfomega)$.
\end{proof}

Two terms in \eqref{eq:enstrophy_stretching} are distinct from the continuous form \eqref{eq:enstrophy_stretching_continuous}.
The former $\frac{1}{2}\int_{\calF^h} \leftjump |\bfomega|^2 \, \bfn \rightjump \cdot \bfu$ represents a contribution to vortex stretching on the facets $\calF^h$.
Jumps across facets in the vorticity magnitude $|\bfomega|^2$ that align with the flow $\bfu$ contribute to vortex stretching (and vice versa).

The latter $\int_{\calT^h} \bfomega \cdot \nabla[\bfu\cdot\bfomega]$ is a defect caused by the failure of $\bfomega$ to be exactly divergence-free.
By \eqref{eq:semidiscrete_3d_a}, were the local helicity density $\bfu\cdot\bfomega$ itself a member of $\bbA_0$, the term would vanish identically.
In general, however, this is not the case;
thus the residual is governed by how well $\bfu\cdot\bfomega$ may be approximated in $\bbA_0$, and is small when the auxiliary space is rich enough to resolve it.

\begin{remark}[Facet terms and higher regularity Stokes complexes]
    If instead our discrete spaces form a higher regularity discrete Stokes complex
    \begin{equation}
    \begin{tikzcd}
        H^2 \cap H^1_0 \arrow{r}{\grad}
            & \bfH^1(\curl) \cap \bfH_0(\curl) \arrow{r}{\curl}
            & \bfH^1 \cap \bfH_0(\div) \arrow{r}{\div}
            & L^2  \\
        \bbA_0 \arrow{r}{\grad} \arrow[u]
            & \bbW_0 \arrow{r}{\curl} \arrow[u]
            & \bbU_0 \arrow{r}{\div} \arrow[u]
            & \bbP \arrow[u]
    \end{tikzcd},
    \end{equation}
    then having $\bfomega \in \bfH^1$ implies that both (i)~the facet term $\frac{1}{2}\int_{\calF^h} \leftjump |\bfomega|^2 \, \bfn \rightjump \cdot \bfu$ in \eqref{eq:enstrophy_stretching} vanishes, and (ii)~the integral $\int_{\calT^h} \bfomega \cdot \nabla[\bfu\cdot\bfomega]$ becomes well-defined when taken over the whole domain $\Omega$ (i.e.~the integrand $\bfomega \cdot \nabla[\bfu\cdot\bfomega]$ belongs to $L^1$).
    We consider only the lower regularity complex \eqref{eq:stokes_3d_bc} in this manuscript, as the higher regularity above is not required for conformity of the discretisation \eqref{eq:semidiscrete_3d} nor the energy and enstrophy stability results \eqref{eq:discrete_stability} as initially stated in the form of \eqref{eq:enstrophy}.
\end{remark}

\begin{lemma}[Vortex stretching identity]\label{lem:eq:stretching_identity}
    For $\bfomega \in \bbW_0$ and divergence-free $\bfu \in \bbU_0$, the following identity holds:
    \begin{equation}\label{eq:stretching_identity}
        \int_\Omega \bfu\cdot\big(\bfomega\times\curl\bfomega\big)
            =  \int_\Omega \bfomega \cdot \varepsilon\,\bfu \cdot \bfomega
            +  \frac{1}{2}\int_{\calF^h} \leftjump |\bfomega|^2 \, \bfn \rightjump \cdot \bfu
            -  \int_{\calT^h} \bfomega \cdot \nabla\big[\bfu\cdot\bfomega\big].
    \end{equation}
\end{lemma}

\begin{proof}
    Consider first a pair $\bfA$, $\bfB$ of general smooth vector fields:
    \begin{subequations}
    \begin{align}
        \div\big[|\bfA|^2\bfB\big]
            &=  \nabla\big[|\bfA|^2\big]\cdot\bfB
            +  (\div\bfB)\,|\bfA|^2  \\
            &=  2 \, \big(\bfA\cdot\nabla\bfA + \bfA\times\curl\bfA\big)\cdot\bfB
            +  (\div\bfB)\,|\bfA|^2  \\
            &=  2 \, \bfA\cdot\big(\nabla\big[\bfB\cdot\bfA\big] - \nabla\bfB\cdot\bfA\big)
            +  2 \, \bfB\cdot\big(\bfA\times\curl\bfA\big)
            +  (\div\bfB)\,|\bfA|^2  \\
            &=  2 \, \bfA\cdot\nabla\big[\bfB\cdot\bfA\big]
            -  2 \, \bfA\cdot\varepsilon\,\bfB\cdot\bfA
            +  2 \, \bfB\cdot\big(\bfA\times\curl\bfA\big)
            +  (\div\bfB)\,|\bfA|^2.
    \end{align}
    \end{subequations}
    The second line holds by the classical identity for the gradient $\nabla[|\bfA|^2]$ (cf.~the rotational form for advection in the Navier--Stokes equation \eqref{eq:navier-stokes_1}), the third holds by expanding $\nabla[\bfB\cdot\bfA] = \nabla\bfA\cdot\bfB + \nabla\bfB\cdot\bfA$, and the final line holds by replacing $\nabla \bfB$ in $\bfA\cdot\nabla\bfB\cdot\bfA$ with its symmetric component $\varepsilon\,\bfB$.
    Integrating over a cell $K \in \calT^h$,
    \begin{equation}
        \frac{1}{2}\oint_{\partial K} |\bfA|^2 \bfB\cdot\bfn
            =  \int_K \left(
               \bfA\cdot\nabla\big[\bfB\cdot\bfA\big]
            -  \bfA\cdot\varepsilon\,\bfB\cdot\bfA
            +  \bfB\cdot\big(\bfA\times\curl\bfA\big)
            +  \frac{1}{2}(\div\bfB)\,|\bfA|^2
            \right).
    \end{equation}
    Considering $\bfA = \bfomega$, $\bfB = \bfu$,
    \begin{equation}
        \frac{1}{2}\oint_{\partial K} |\bfomega|^2 \bfu\cdot\bfn
            =  \int_K \left(
               \bfomega\cdot\nabla\big[\bfu\cdot\bfomega\big]
            -  \bfomega\cdot\varepsilon\,\bfu\cdot\bfomega
            +  \bfu\cdot\big(\bfomega\times\curl\bfomega\big)
            \right),
    \end{equation}
    where the final term $\frac{1}{2}(\div\bfu)\,|\bfomega|^2$ vanishes by the assumed incompressibility $\div\bfu = 0$.
    Summing over cells,
    \begin{equation}
        \frac{1}{2}\int_{\calF^h} \leftjump |\bfomega|^2 \bfu\cdot\bfn \rightjump
            =  \int_{\calT^h} \left(
               \bfomega\cdot\nabla\big[\bfu\cdot\bfomega\big]
            -  \bfomega\cdot\varepsilon\,\bfu\cdot\bfomega
            +  \bfu\cdot\big(\bfomega\times\curl\bfomega\big)
            \right)
            =  \frac{1}{2}\int_{\calF^h} \leftjump |\bfomega|^2 \bfn \rightjump \cdot \bfu,
    \end{equation}
    where we use the continuity of $\bfu$ to write $\leftjump |\bfomega|^2 \bfu\cdot\bfn \rightjump = \leftjump |\bfomega|^2 \bfn \rightjump \cdot \bfu$\footnote{
        This relies on the $\bfH^1$-continuity of $\bfu$.
        For the penalty method of \Cref{sec:penalty}, where only the normal component $\bfu\cdot\bfn$ is continuous across facets, the identity \eqref{eq:stretching_identity} still holds with $\bfu$ replaced by its average $\leftavg \bfu \rightavg$ in the facet term.
    }.
    Rearranging gives the desired identity \eqref{eq:stretching_identity}.
\end{proof}

\subsection{Periodic boundary conditions}\label{sec:periodic_1}

Convex domains without periodic boundaries are necessarily contractible.
A common setting for the Navier--Stokes equations, however, is the \emph{periodic box}, e.g.~for modelling flow in a channel or pipe.
For example, in \Cref{sec:hill_vortex} we consider our proposed 3D discretisation \eqref{eq:semidiscrete_3d} over the fully periodic cube.
The non-trivial topologies of such domains affect the exactness of the Stokes complex \eqrefs{eq:stokes_3d_bc,eq:stokes_2d_bc} through the introduction of (non-trivial) harmonic forms in $\frakH^k_0$.
Depending on where these harmonic forms appear in the sequence, the solvability and well-posedness of the discretisation can be affected.

\begin{remark}[Harmonic forms in periodic boxes]\label{rem:betti}
    A box $\Omega \subset \bbR^n$ with $p$ periodic directions is homotopy equivalent to the $p$-torus, giving Betti numbers \eqref{eq:betti}
    \begin{equation}
        \dim \frakH^k_0
            \;(=  b_{n-k})\;
            =  \begin{cases}
                \binom{p}{n-k}, & k \ge n - p, \\
                0,              & k < n - p.
            \end{cases}
    \end{equation}
    Harmonic $k$-forms are present precisely for $k \ge n - p$, each further periodic direction lowering the threshold by one.
\end{remark}

\Cref{fig:harmonic_forms} shows periodic boxes with topologies that exhibit harmonic forms at each stage in the complex.
Pairs of arrows mark pairs of periodic edges (in 2D) or faces (in 3D), with the number of tips and the colour distinguishing the periodic directions;
non-periodic (i.e.~free slip) portions of the boundary are hatched in 2D and shaded in 3D.
We consider the implication of (non-trivial) harmonic forms for each variable in our proposed scheme \eqrefs{eq:semidiscrete_3d,eq:semidiscrete_2d}.

\begin{figure}[pos=!ht]
    \centering
    \begin{subfigure}{0.48\textwidth}
        \centering
        \begin{tikzpicture}[scale=0.5, baseline=(current bounding box.center), >=stealth,
                edge/.style   = {line width=1.1pt, black},
                lite/.style   = {line width=1.1pt, black!40},
                hidden/.style = {line width=0.6pt, black!35, densely dashed},
                glass/.style  = {fill=black!15, fill opacity=0.62},
                wall/.style   = {fill=black!48, fill opacity=0.62},
                hatch/.style = {black!55, line width=0.5pt},
                per/.style   = {seabornblue, line width=1pt, -{Straight Barb[length=1.1mm, width=2.4mm]}},
                perr/.style  = {seaborngreen, line width=1pt, -{Straight Barb[length=1.1mm, width=2.4mm] Straight Barb[length=1.1mm, width=2.4mm]}},
                perrr/.style = {seabornred, line width=1pt, -{Straight Barb[length=1.1mm, width=2.4mm] Straight Barb[length=1.1mm, width=2.4mm] Straight Barb[length=1.1mm, width=2.4mm]}}]
            \draw[hidden] (0.72,0.52) -- (2.72,0.52);
            \draw[hidden] (0.72,0.52) -- (0.72,2.52);
            \draw[hidden] (0.72,0.52) -- (0,0);
            \draw[per] (-0.75,1.26) -- (0.55,1.26);
            \draw[perr] (1.36,-0.75) -- (1.36,0.55);
            \draw[perrr] (1.86,1.62) -- (3.16,2.56);
            \fill[glass] (0,2) -- (2,2) -- (2.72,2.52) -- (0.72,2.52) -- cycle;
            \fill[glass] (2,0) -- (2.72,0.52) -- (2.72,2.52) -- (2,2) -- cycle;
            \fill[glass] (0,0) rectangle (2,2);
            \draw[lite] (0,0) rectangle (2,2);
            \draw[lite] (2,0) -- (2.72,0.52) -- (2.72,2.52) -- (2,2);
            \draw[lite] (0,2) -- (0.72,2.52) -- (2.72,2.52);
            \draw[per] (1.85,1.26) -- (3.35,1.26);
            \draw[perr] (1.36,1.85) -- (1.36,3.3);
            \draw[perrr] (-0.15,0.17) -- (0.78,0.84);
        \end{tikzpicture}
        \caption{$\alpha$ (0-forms in 3D)}\label{fig:harmonic_alpha}
    \end{subfigure}
    \begin{subfigure}{0.48\textwidth}
        \centering
        \begin{tikzpicture}[scale=0.5, baseline=(current bounding box.center), >=stealth,
                edge/.style  = {line width=1.1pt, black},
                lite/.style  = {line width=1.1pt, black!40},
                hatch/.style = {black!55, line width=0.5pt},
                per/.style   = {seabornblue, line width=1pt, -{Straight Barb[length=1.1mm, width=2.4mm]}},
                perr/.style  = {seaborngreen, line width=1pt, -{Straight Barb[length=1.1mm, width=2.4mm] Straight Barb[length=1.1mm, width=2.4mm]}},
                perrr/.style = {seabornred, line width=1pt, -{Straight Barb[length=1.1mm, width=2.4mm] Straight Barb[length=1.1mm, width=2.4mm] Straight Barb[length=1.1mm, width=2.4mm]}}]
            \fill[black!5] (0,0) rectangle (2.4,2.4);
            \draw[lite] (0,0) rectangle (2.4,2.4);
            \draw[per] (-0.75,1.2) -- (0.45,1.2);
            \draw[per] (1.95,1.2) -- (3.15,1.2);
            \draw[perr] (1.2,-0.75) -- (1.2,0.45);
            \draw[perr] (1.2,1.95) -- (1.2,3.15);
        \end{tikzpicture}\hspace{3mm}%
        \begin{tikzpicture}[scale=0.5, baseline=(current bounding box.center), >=stealth,
                edge/.style   = {line width=1.1pt, black},
                lite/.style   = {line width=1.1pt, black!40},
                hidden/.style = {line width=0.6pt, black!35, densely dashed},
                glass/.style  = {fill=black!15, fill opacity=0.62},
                wall/.style   = {fill=black!48, fill opacity=0.62},
                hatch/.style = {black!55, line width=0.5pt},
                per/.style   = {seabornblue, line width=1pt, -{Straight Barb[length=1.1mm, width=2.4mm]}},
                perr/.style  = {seaborngreen, line width=1pt, -{Straight Barb[length=1.1mm, width=2.4mm] Straight Barb[length=1.1mm, width=2.4mm]}},
                perrr/.style = {seabornred, line width=1pt, -{Straight Barb[length=1.1mm, width=2.4mm] Straight Barb[length=1.1mm, width=2.4mm] Straight Barb[length=1.1mm, width=2.4mm]}}]
            \draw[hidden] (0.72,0.52) -- (2.72,0.52);
            \draw[hidden] (0.72,0.52) -- (0.72,2.52);
            \draw[hidden] (0.72,0.52) -- (0,0);
            \draw[per] (-0.75,1.26) -- (0.55,1.26);
            \draw[perr] (1.86,1.62) -- (3.16,2.56);
            \fill[wall]  (0,2) -- (2,2) -- (2.72,2.52) -- (0.72,2.52) -- cycle;
            \fill[glass] (2,0) -- (2.72,0.52) -- (2.72,2.52) -- (2,2) -- cycle;
            \fill[glass] (0,0) rectangle (2,2);
            \draw[lite] (0,0) -- (0,2);
            \draw[lite] (2,0) -- (2,2);
            \draw[lite] (2.72,0.52) -- (2.72,2.52);
            \draw[edge] (0,0) -- (2,0) -- (2.72,0.52);
            \draw[edge] (0,2) -- (2,2) -- (2.72,2.52) -- (0.72,2.52) -- cycle;
            \draw[per] (1.85,1.26) -- (3.35,1.26);
            \draw[perr] (-0.15,0.17) -- (0.78,0.84);
        \end{tikzpicture}
        \caption{$\bfomega$ (1-forms in 3D, 0-forms in 2D)}\label{fig:harmonic_omega}
    \end{subfigure}

    \vspace{4mm}

    \begin{subfigure}{0.48\textwidth}
        \centering
        \begin{tikzpicture}[scale=0.5, baseline=(current bounding box.center), >=stealth,
                edge/.style  = {line width=1.1pt, black},
                lite/.style  = {line width=1.1pt, black!40},
                hatch/.style = {black!55, line width=0.5pt},
                per/.style   = {seabornblue, line width=1pt, -{Straight Barb[length=1.1mm, width=2.4mm]}},
                perr/.style  = {seaborngreen, line width=1pt, -{Straight Barb[length=1.1mm, width=2.4mm] Straight Barb[length=1.1mm, width=2.4mm]}},
                perrr/.style = {seabornred, line width=1pt, -{Straight Barb[length=1.1mm, width=2.4mm] Straight Barb[length=1.1mm, width=2.4mm] Straight Barb[length=1.1mm, width=2.4mm]}}]
            \fill[black!5] (0,0) rectangle (2.4,2.4);
            \draw[lite] (0,0) -- (0,2.4);  \draw[lite] (2.4,0) -- (2.4,2.4);
            \draw[edge] (0,0) -- (2.4,0);  \draw[edge] (0,2.4) -- (2.4,2.4);
            \foreach \t in {0.2,0.5,...,2.4} {\draw[hatch] (\t,0) -- ++(-0.18,-0.18);}
            \foreach \t in {0.05,0.35,...,2.25} {\draw[hatch] (\t,2.4) -- ++(0.18,0.18);}
            \draw[per] (-0.75,1.2) -- (0.45,1.2);
            \draw[per] (1.95,1.2) -- (3.15,1.2);
        \end{tikzpicture}\hspace{3mm}%
        \begin{tikzpicture}[scale=0.5, baseline=(current bounding box.center), >=stealth,
                edge/.style   = {line width=1.1pt, black},
                lite/.style   = {line width=1.1pt, black!40},
                hidden/.style = {line width=0.6pt, black!35, densely dashed},
                glass/.style  = {fill=black!15, fill opacity=0.62},
                wall/.style   = {fill=black!48, fill opacity=0.62},
                hatch/.style = {black!55, line width=0.5pt},
                per/.style   = {seabornblue, line width=1pt, -{Straight Barb[length=1.1mm, width=2.4mm]}},
                perr/.style  = {seaborngreen, line width=1pt, -{Straight Barb[length=1.1mm, width=2.4mm] Straight Barb[length=1.1mm, width=2.4mm]}},
                perrr/.style = {seabornred, line width=1pt, -{Straight Barb[length=1.1mm, width=2.4mm] Straight Barb[length=1.1mm, width=2.4mm] Straight Barb[length=1.1mm, width=2.4mm]}}]
            \draw[hidden] (0.72,0.52) -- (2.72,0.52);
            \draw[hidden] (0.72,0.52) -- (0.72,2.52);
            \draw[hidden] (0.72,0.52) -- (0,0);
            \draw[per] (-0.75,1.26) -- (0.55,1.26);
            \fill[wall]  (0,2) -- (2,2) -- (2.72,2.52) -- (0.72,2.52) -- cycle;
            \fill[glass] (2,0) -- (2.72,0.52) -- (2.72,2.52) -- (2,2) -- cycle;
            \fill[wall]  (0,0) rectangle (2,2);
            \draw[edge] (0,0) rectangle (2,2);
            \draw[edge] (2,0) -- (2.72,0.52) -- (2.72,2.52) -- (2,2);
            \draw[edge] (0,2) -- (0.72,2.52) -- (2.72,2.52);
            \draw[per] (1.85,1.26) -- (3.35,1.26);
        \end{tikzpicture}
        \caption{$\bfu$ (2-forms in 3D, 1-forms in 2D)}\label{fig:harmonic_u}
    \end{subfigure}
    \begin{subfigure}{0.48\textwidth}
        \centering
        \begin{tikzpicture}[scale=0.5, baseline=(current bounding box.center), >=stealth,
                edge/.style  = {line width=1.1pt, black},
                lite/.style  = {line width=1.1pt, black!40},
                hatch/.style = {black!55, line width=0.5pt},
                per/.style   = {seabornblue, line width=1pt, -{Straight Barb[length=1.1mm, width=2.4mm]}},
                perr/.style  = {seaborngreen, line width=1pt, -{Straight Barb[length=1.1mm, width=2.4mm] Straight Barb[length=1.1mm, width=2.4mm]}},
                perrr/.style = {seabornred, line width=1pt, -{Straight Barb[length=1.1mm, width=2.4mm] Straight Barb[length=1.1mm, width=2.4mm] Straight Barb[length=1.1mm, width=2.4mm]}}]
            \fill[black!5] (0,0) rectangle (2.4,2.4);
            \draw[edge] (0,0) rectangle (2.4,2.4);
            \foreach \t in {0.2,0.5,...,2.4} {\draw[hatch] (\t,0) -- ++(-0.18,-0.18);}
            \foreach \t in {0.05,0.35,...,2.25} {\draw[hatch] (\t,2.4) -- ++(0.18,0.18);}
            \foreach \t in {0.2,0.5,...,2.4} {\draw[hatch] (0,\t) -- ++(-0.18,-0.18);}
            \foreach \t in {0.05,0.35,...,2.25} {\draw[hatch] (2.4,\t) -- ++(0.18,0.18);}
        \end{tikzpicture}\hspace{8mm}%
        \begin{tikzpicture}[scale=0.5, baseline=(current bounding box.center), >=stealth,
                edge/.style   = {line width=1.1pt, black},
                lite/.style   = {line width=1.1pt, black!40},
                hidden/.style = {line width=0.6pt, black!35, densely dashed},
                glass/.style  = {fill=black!15, fill opacity=0.62},
                wall/.style   = {fill=black!48, fill opacity=0.62},
                hatch/.style = {black!55, line width=0.5pt},
                per/.style   = {seabornblue, line width=1pt, -{Straight Barb[length=1.1mm, width=2.4mm]}},
                perr/.style  = {seaborngreen, line width=1pt, -{Straight Barb[length=1.1mm, width=2.4mm] Straight Barb[length=1.1mm, width=2.4mm]}},
                perrr/.style = {seabornred, line width=1pt, -{Straight Barb[length=1.1mm, width=2.4mm] Straight Barb[length=1.1mm, width=2.4mm] Straight Barb[length=1.1mm, width=2.4mm]}}]
            \draw[hidden] (0.72,0.52) -- (2.72,0.52);
            \draw[hidden] (0.72,0.52) -- (0.72,2.52);
            \draw[hidden] (0.72,0.52) -- (0,0);
            \fill[wall] (0,2) -- (2,2) -- (2.72,2.52) -- (0.72,2.52) -- cycle;
            \fill[wall] (2,0) -- (2.72,0.52) -- (2.72,2.52) -- (2,2) -- cycle;
            \fill[wall] (0,0) rectangle (2,2);
            \draw[edge] (0,0) rectangle (2,2);
            \draw[edge] (2,0) -- (2.72,0.52) -- (2.72,2.52) -- (2,2);
            \draw[edge] (0,2) -- (0.72,2.52) -- (2.72,2.52);
        \end{tikzpicture}
        \caption{$p$ (3-forms in 3D, 2-forms in 2D)}\label{fig:harmonic_p}
    \end{subfigure}
    \caption{
        Domains with differing periodicity exhibiting non-trivial harmonic forms $\frakH^k_0$, potentially presenting issues for different variables in our proposed discretisation \eqrefs{eq:semidiscrete_3d,eq:semidiscrete_2d}.
        Pairs of arrows identify periodic edges (in 2D) or faces (in 3D), distinguished by tip count and colour;
        the remaining, free slip boundaries are hatched in 2D and shaded in 3D.
        Note that a periodic box with non-trivial harmonic $k$-forms ($\frakH^k_0 \ne \{0\}$) will exhibit non-trivial harmonic $l$-forms for all $l \ge k$ ($\frakH^l_0 \ne \{0\}$).
    }\label{fig:harmonic_forms}
\end{figure}
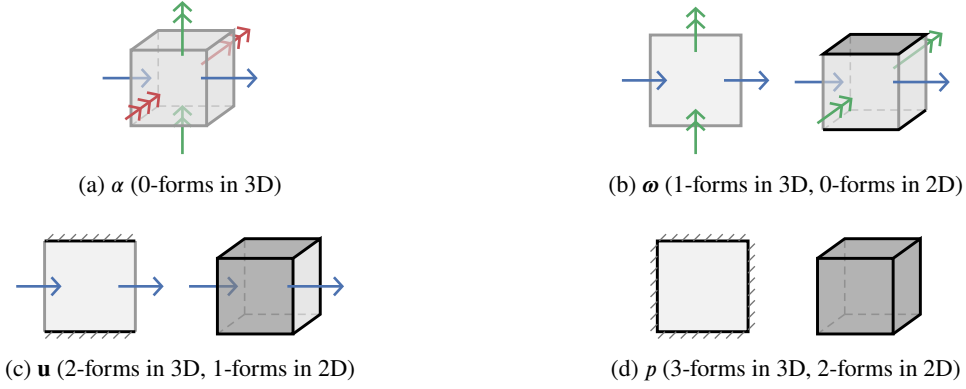

\paragraph{The Lagrange multiplier $\alpha$ (0-forms in 3D, \Cref{fig:harmonic_alpha}).}
Any harmonic 0-form $\alpha_\rmH \in \frakH^0_0 \subset \bbA_0$ must be constant (it must satisfy $\nabla \alpha_\rmH = 0$ by definition) and can therefore only be present when the domain has periodic boundary conditions over the whole boundary $\partial\Omega$.
This causes issues for our 3D semidiscretisation \eqref{eq:semidiscrete_3d} as $\alpha \in \bbA_0$ appears only through its gradient.
With fully periodic boundaries, the scheme becomes ill-posed as $\alpha \in \bbA_0 \, (= \bbA)$ is ill-defined up to a constant.
This freedom is theoretically benign (we are typically interested in only the physical variables $\bfomega$, $\bfu$ and $p$);
the real casualty is the invertibility of the discrete system.

Two solutions present themselves:
(i)~Adding the term $(\alpha, \beta)$ to \eqref{eq:semidiscrete_3d_a} removes the nullspace outright.
As established in the proof of \Cref{th:stability}, $\alpha$ must be constant;
with this term added, taking $\beta = 1$ in \eqref{eq:semidiscrete_3d_a} confirms $\alpha = 0$, eliminating the nullspace.
(ii)~With certain linear solvers, such nullspaces may be reported to the solver as one might typically report the constant pressure nullspace (see the discussion below).

\paragraph{The vorticity $\bfomega$ (1-forms in 3D, 0-forms in 2D, \Cref{fig:harmonic_omega}).}
In 3D, the harmonic 1-forms $\bfomega_\rmH \in \frakH^1_0 \subset \bbW_0$ are the $\curl$-free fields that are $\bfL^2$-orthogonal to $\nabla\bbA_0$\footnote{
    On a periodic box, they are the constant fields orthogonal to two (distinct) periodic directions, i.e.~the constant field in the $z$ direction in \Cref{fig:harmonic_omega}, or constant fields in any of the cardinal directions in the fully periodic case.
}.
Considering $\bfomega = \bfomega_\rmH$ in \eqref{eq:semidiscrete_3d}, it is simple to see that the system defining the vorticity \eqrefs{eq:semidiscrete_3d_a,eq:semidiscrete_3d_b} has a nullspace corresponding precisely to the harmonic component of $\bfomega \in \bbW_0$.
Similarly in 2D, the harmonic 0-forms $\frakH^0_0 \subset \bbW_0$ are the constant functions, present when the boundaries are fully periodic.
Considering constant $\omega$ in \eqref{eq:semidiscrete_2d}, we see there is a similar nullspace.

Unlike the Lagrange multiplier $\alpha$, however, this freedom is \emph{not} benign.
In 3D, the advective term $(\bfu \times \bfomega, \bfv)$ in \eqref{eq:semidiscrete_3d_c} features $\bfomega$ directly, not just $\curl\bfomega$;
similarly in 2D, the advective term $(\omega \bfu^\perp, \bfv)$ features $\omega$, not just $\nabla\omega$.
To make the system well-posed, we must fix the harmonic component of the vorticity.

Consider the continuous case in 3D, where $\bfomega = \curl\bfu$ exactly:
the Hodge decomposition places the image of $\curl$ $\bfL^2$-orthogonal to the harmonic forms.
The ill-posedness on the discrete level can therefore be repaired by enforcing $\bfL^2$-orthogonality of $\bfomega \in \bbW_0$ to the harmonic forms $\frakH^1_0$.
Equivalently, we fix the harmonic component of $\bfomega$ to be zero.

Practically, one may restore well-posedness to \eqref{eq:semidiscrete_3d} by (i)~solving also for $\bflambda \in \frakH^1_0$, modifying \eqref{eq:semidiscrete_3d_b} with the Lagrange multiplier term,
\begin{subequations}\label{eq:harmonic_constraint}
\begin{equation}
    (\curl \bfomega, \curl \bfchi)  =  2 (\varepsilon \, \bfu, \varepsilon \curl \bfchi) - (\nabla \alpha, \bfchi) + (\bflambda, \bfchi),
\end{equation}
and (ii)~imposing the additional constraint
\begin{equation}
    0  =  (\bfomega, \bfmu)
\end{equation}
\end{subequations}
for all $\bfmu \in \frakH^1_0$.
Since $\frakH^1_0$ is of a limited finite dimension ($\dim\frakH^1_0 = b_2$), a basis for $\frakH^1_0$ may be computed \emph{a priori}.
We apply a similar fix in \Cref{sec:vortex_street_3d} (with a slight distinction due to the handling of the more general boundary conditions; see \Cref{sec:bcs}) to handle harmonic forms in $\frakH^1_0$.

A similar orthogonality constraint restores well-posedness to the 2D scheme \eqref{eq:semidiscrete_2d}.
However, in this case the only possible harmonic 0-form $\omega_\rmH \in \frakH^0_0$ is the constant function.
Well-posedness is therefore restored in fully periodic 2D problems by fixing
\begin{equation}
    \int_\Omega \omega = 0,
\end{equation}
again through a Lagrange multiplier.

\paragraph{The velocity $\bfu$ (2-forms in 3D, 1-forms in 2D, \Cref{fig:harmonic_u}).}
In 3D, the harmonic 2-forms $\bfu_\rmH \in \frakH^2_0 \subset \bbU_0$ (1-forms $\frakH^1_0$ in 2D) are the divergence-free fields that are $\bfL^2$-orthogonal to $\curl\bbW_0$ (or $\grad^\perp\bbW_0$ in 2D)\footnote{
    On a periodic box, they are the constant fields in the periodic directions, i.e.~the constant field in the $x$ direction in \Cref{fig:harmonic_u}.
}.
While these do not affect the well-posedness of the velocity equations \eqrefs{eq:semidiscrete_3d_c,eq:semidiscrete_2d}\footnote{
    Were we instead to discretise the \emph{stationary} problem, the harmonic component would have to be prescribed.
    This is analogous to the necessity of prescribing the mean flow in any periodic channel.
}, exactness at $\bbU_0$ was used in the proof of enstrophy stability \eqref{eq:enstrophy_discrete} to write $\bfu = \curl\bfpsi$ for some $\bfpsi \in \bbW_0$.

We can, however, show a slightly modified version of \eqref{eq:enstrophy_discrete}.
Using the divergence-free property of $\bfu$, decompose $\bfu = \curl \bfpsi + \bfh$ for $\bfpsi \in \bbW_0$, with $\bfh \in \frakH^2_0 \subset \bbU_0$ the harmonic component\footnote{
    For the magnetofrictional equations, \cite{He_et_al_2026} recently used similar discrete Hodge decompositions to show modified discrete helicity stability results on non-trivial topologies where exactness fails.
}.
By the $\bfL^2$-orthogonality of the harmonic form $\bfh \in \frakH^2_0$ to $\curl\bfomega \in \curl\bbW_0$, \eqref{eq:stability_proof_2} reduces to
\begin{subequations}
\begin{equation}
    (\partial_t \curl\bfpsi, \curl\bfomega)
        = (\partial_t \bfu, \curl\bfomega)
        = \int_\Omega \bfu \cdot (\bfomega \times \curl\bfomega) - \frac{2}{\Re} (\varepsilon\,\bfu, \varepsilon\curl\bfomega).
\end{equation}
Through the same test functions, \eqref{eq:stability_proof_3} gives
\begin{equation}
    (\curl\bfomega, \partial_t \curl\bfpsi)
        =  2 (\varepsilon \, \bfu, \varepsilon \, \partial_t \curl\bfpsi)
        =  \partial_t \calE - 2 (\varepsilon \, \bfu, \varepsilon \, \partial_t \bfh),
    \qquad
    \| \curl\bfomega \|^2
        =  2 (\varepsilon \, \bfu, \varepsilon \curl\bfomega),
\end{equation}
\end{subequations}
giving the modified enstrophy stability result
\begin{equation}\label{eq:enstrophy_discrete_modified}
    \partial_t \calE  =  \int_\Omega \bfu \cdot (\bfomega \times \curl\bfomega) - \frac{1}{\Re} \|\curl\bfomega\|^2 + 2 (\varepsilon \, \bfu, \varepsilon \, \partial_t \bfh).
\end{equation}
On the periodic box, the harmonic component $\bfh$ is constant (assuming $\bbU_0$ contains the constant function), implying the final $2 (\varepsilon \, \bfu, \varepsilon \, \partial_t \bfh)$ term vanishes.
Thus \eqref{eq:enstrophy_discrete_modified} reduces to \eqref{eq:enstrophy_discrete}, returning the original enstrophy stability result.

\paragraph{The pressure $p$ (3-forms in 3D, 2-forms in 2D, \Cref{fig:harmonic_p}).}
Finally, any connected domain has $\dim\frakH^n_0 = b_0 = 1$, carrying exactly one harmonic form in the last slot:
the constant function.
This is the familiar pressure nullspace of any enclosed flow and requires no handling different from that used in a classical Navier--Stokes discretisation:
either (i)~normalise $\int_\Omega p = 0$ through a real-valued Lagrange multiplier, or (ii)~report it to the solver.

\vspace{2mm}

Only the vorticity then necessitates any substantial modification to the discretisation.
More general topologies (e.g.~holes, tunnels, voids) introduce harmonic forms in a similar manner, and are handled analogously in \Cref{sec:periodic_2} below, once more general boundary conditions have been considered.

\section{Practical implementation}\label{sec:implementation}

Implementation of the discretisation as presented \eqrefs{eq:semidiscrete_3d,eq:semidiscrete_2d} requires access to a discrete Stokes complex.
However, as discussed above, such complexes are rarely available in finite element software.

In this section, we consider practical alternatives that do not require the implementation of a discrete Stokes complex.
\Cref{sec:penalty} considers a penalty method, whereby the discrete Stokes complex \eqrefs{eq:stokes_3d_bc,eq:stokes_2d_bc} is modelled via a discrete de Rham complex, with penalty terms handling the non-conformity.
Sections~\ref{sec:charlie} and~\ref{sec:streamfunction} consider two alternative ways to implement the full conforming scheme without access to the full Stokes complex.
The former removes the requirement of an implementation of the vorticity space $\bbW_0$;
the latter on the other hand reparametrises solely in $\bbA_0$ and $\bbW_0$ in 3D, or solely in $\bbW_0$ in 2D.
Finally, \Cref{sec:meevc_stabilisation} borrows a $\delta$-sized component of the MEEVC scheme \eqref{eq:meevc_semidiscrete} to impose certain constraints on the vorticity implicitly, dispensing with both the Lagrange multiplier $\alpha$ and those enforcing any harmonic constraints, at the cost of a $\delta$-sized perturbation to the enstrophy stability.

\subsection{Penalty method}\label{sec:penalty}

A discrete Stokes complex is often unavailable in finite element software, in particular in 3D or on curved meshes.
We therefore consider a penalty method, in which the discrete Stokes complex \eqrefs{eq:stokes_3d_bc,eq:stokes_2d_bc} is replaced by a far more widely available discrete de Rham complex, and the resulting non-conformity of the velocity is handled by interior penalty terms on the facets \citep{Douglas_Dupont_1976,DiPietro_Ern_2011}.
The enstrophy structure survives this replacement, though in a broken form.
This discretisation is effective in 2D. However, in 3D it suffers a reparable failure mode that we return to in \Cref{sec:penalty_sigma}.


\paragraph{3D.}
Suppose instead that the discrete spaces $\bbA_0$, $\bbW_0$, $\bbU_0$, $\bbP$ form a discrete de Rham complex,
\begin{equation}
\begin{tikzcd}
    H^1_0 \arrow{r}{\grad}
        & \bfH_0(\curl) \arrow{r}{\curl}
        & \bfH_0(\div) \arrow{r}{\div}
        & L^2  \\
    \bbA_0 \arrow{r}{\grad} \arrow[u]
        & \bbW_0 \arrow{r}{\curl} \arrow[u]
        & \bbU_0 \arrow{r}{\div} \arrow[u]
        & \bbP \arrow[u]
\end{tikzcd},
\end{equation}
without the enhanced regularity of the Stokes complex \eqref{eq:stokes_3d_bc}.
To handle the non-conformity in $\bbW_0$ and $\bbU_0$, we approximate the inner product $(\varepsilon\,\cdot, \varepsilon\,\cdot)$ over $\bbU_0$ with a symmetric interior penalty term.
Define the bilinear form $\calD^h$
\begin{multline}
    \calD^h\,[\bfu, \bfv]
        \coloneqq
        \int_{\calT^h} \varepsilon\,\bfu : \varepsilon\,\bfv
      - \int_{\calF^h} \left(
            \leftavg \varepsilon\,\bfu \rightavg : \sym \leftjump \bfv \otimes \bfn \rightjump
          + \sym \leftjump \bfu \otimes \bfn \rightjump : \leftavg \varepsilon\,\bfv \rightavg
        \right)  \\
      + \sigma \int_{\calF^h} \frac{1}{h} \sym \leftjump \bfu \otimes \bfn \rightjump : \sym \leftjump \bfv \otimes \bfn \rightjump.
\end{multline}
We use $\calD^h$ to discretise the non-conforming symmetric-gradient inner products, namely $(\varepsilon\,\bfu, \varepsilon\,\curl\bfchi)$ in the vorticity equation \eqref{eq:semidiscrete_3d_b} and the viscous term $(\varepsilon\,\bfu, \varepsilon\,\bfv)$ in the momentum equation \eqref{eq:semidiscrete_3d_c}.

\begin{definition}[3D semidiscretisation (penalty method)]
    Find $(\alpha, \bfomega, \bfu, p) \in \bbA_0 \times \bbW_0 \times \bbU_0 \times \bbP$ such that
    \begin{subequations}\label{eq:semidiscrete_3d_penalty}
    \begin{align}
        0  &=  (\bfomega, \nabla \beta),  \label{eq:semidiscrete_3d_penalty_a}  \\
        (\curl \bfomega, \curl \bfchi)  &=  2 \calD^h\,[\bfu, \curl\bfchi] - (\nabla \alpha, \bfchi),  \label{eq:semidiscrete_3d_penalty_b}  \\
        (\partial_t \bfu, \bfv)  &=  (\bfu \times \bfomega, \bfv) + \left(\frac{1}{2}|\bfu|^2 + p, \div \bfv\right) - \frac{2}{\Re}\calD^h\,[\bfu, \bfv],  \\
        0  &=  (\div \bfu, q),
    \end{align}
    \end{subequations}
    for all $(\beta, \bfchi, \bfv, q) \in \bbA_0 \times \bbW_0 \times \bbU_0 \times \bbP$.
\end{definition}

\paragraph{2D.}
The penalty method here is similar.

\begin{definition}[2D semidiscretisation (penalty method)]
    Find $(\omega, \bfu, p) \in \bbW_0 \times \bbU_0 \times \bbP$ such that
    \begin{subequations}\label{eq:semidiscrete_2d_penalty}
    \begin{align}
        (\nabla \omega, \nabla \chi)  &=  - \, 2 \calD^h\,[\bfu, \nabla^\perp \chi],  \label{eq:semidiscrete_2d_penalty_a}  \\
        (\partial_t \bfu, \bfv)  &=  - \, (\omega \bfu^\perp, \bfv) + \left(\frac{1}{2}|\bfu|^2 + p, \div \bfv\right) - \frac{2}{\Re}\calD^h\,[\bfu, \bfv],  \\
        0  &=  (\div \bfu, q),
    \end{align}
    \end{subequations}
    for all $(\chi, \bfv, q) \in \bbW_0 \times \bbU_0 \times \bbP$.
\end{definition}

In both cases \eqrefs{eq:semidiscrete_3d_penalty,eq:semidiscrete_2d_penalty}, we preserve the evolution not of the enstrophy $\calE(\bfu) \coloneqq \|\varepsilon\,\bfu\|^2$, but of a broken enstrophy $\calE^h(\bfu) \coloneqq \calD^h\,[\bfu, \bfu]$, the proof being identical to that of \eqref{eq:enstrophy_discrete}.
Numerical tests with this penalty method can be found in \Cref{sec:vortex_street} below.

\begin{remark}[Boundary conditions]
    In addition to the convenience of not requiring a discrete Stokes complex, a further appeal of the penalty method comes in the strong enforcement of boundary conditions.
    For instance, the typical boundary condition for $\bfu \in \bfH^1$ is $\bfu = \bfzero$ on $\partial\Omega$, i.e.~$\bfu \in \bfH^1_0$.
    This carries over to the discrete setting for $\bfu \in \bbU \subset \bfH^1$;
    in spaces such as the continuous Lagrange space, it is simple to enforce $\bfu = \bfzero$ on $\partial\Omega$.
    In the above, however, we consider the boundary condition $\bfu \cdot \bfn = 0$, which is in general non-trivial to enforce with the higher regularity $\bbU \subset \bfH^1$.
    The de Rham complex penalty method avoids this, as the boundary condition $\bfu \cdot \bfn = 0$ is straightforward to enforce using the $H(\div)$-conforming spaces $\bbU$, e.g.~the Brezzi--Douglas--Marini space.
    Similar logic applies to $\bfomega \in \bbW$.
\end{remark}

\subsubsection{Breakdown of the vorticity in 3D}\label{sec:penalty_sigma}

With the penalty method, jumps in the velocity are controlled by the broken enstrophy $\calE^h(\bfu) \coloneqq \calD^h\,[\bfu, \bfu]$.
It carries the penalty term explicitly:
\begin{equation}\label{eq:penalty_jump_term}
    \calE^h(\bfu)  =  \cdots + \sigma \int_{\calF^h} \frac{1}{h} \left| \sym \leftjump \bfu \otimes \bfn \rightjump \right|^2,
\end{equation}
such that a bound on $\calE^h$ offers a bound on the velocity jumps.
Consider then the auxiliary Stokes problem, \eqrefs{eq:semidiscrete_3d_penalty_a,eq:semidiscrete_3d_penalty_b} in 3D and \eqref{eq:semidiscrete_2d_penalty_a} in 2D, reconstructing the vorticity $\bfomega$ from the velocity $\bfu$.
The right-hand side takes the form $2\calD^h\,[\bfu, \curl\bfchi]$;
through the final term of $\calD^h$ this carries a contribution scaling as $\sigma / h \cdot \sym \leftjump \bfu \otimes \bfn \rightjump$.

In 2D, the broken enstrophy $\calE^h$ is dissipated by \eqref{eq:enstrophy_discrete}, and hence bounded uniformly in time by its initial value.
With it, the velocity jumps \eqref{eq:penalty_jump_term} are necessarily controlled, and the right-hand side of the vorticity reconstruction \eqref{eq:semidiscrete_2d_penalty_a} poses little issue.
Moreover, if the vorticity $\omega$ were to diverge, there is no mechanism in 2D by which a large vorticity field may lead to blow-ups in the enstrophy.

In 3D, enstrophy is no longer necessarily dissipated.
There is therefore no \emph{a priori} bound on $\calE^h$, and consequently none on the velocity jumps.
Jumps in $\bfu$ across facets may offer a contribution to the vorticity $\bfomega$ through \eqref{eq:semidiscrete_3d_penalty_b} that scales with the penalty parameter $\sigma$.
In practice, we have observed the vorticity $\bfomega$ in 3D discretisations using the above penalty method \eqref{eq:semidiscrete_3d_penalty} to diverge, thus causing the solver to fail as the large vorticity field feeds back into the enstrophy $\calE^h$ through vortex stretching \eqref{eq:enstrophy_stretching}.
This is exactly the case for the discretisation in \Cref{sec:vortex_street_3d}:
the vorticity blows up in localised regions and the solver fails.
Thus, without modification, this failure mode often renders the 3D penalty method \eqref{eq:semidiscrete_3d_penalty} inadmissible.

A practical remedy to this blow-up is as follows.
Denote by $\bar{\calD}^h$ the form $\calD^h$ with its penalty parameter $\sigma$ replaced by some reduced $\bar{\sigma} \in (0, \sigma)$.
In 3D, where such instabilities occur, we propose using $\bar{\calD}^h$ in place of $\calD^h$ in the vorticity reconstruction alone,
\begin{equation}\label{eq:semidiscrete_3d_penalty_reduced}
    (\curl \bfomega, \curl \bfchi)  =  2 \bar{\calD}^h\,[\bfu, \curl\bfchi] - (\nabla \alpha, \bfchi),
\end{equation}
leaving $\calD^h$ unchanged in the momentum equation.
The amplification of the velocity jumps in the vorticity reconstruction is reduced by the ratio $\bar{\sigma} / \sigma$, weakening the influence of growth in the jumps of $\bfu$ on blow ups in $\bfomega$.
Naturally this modification is not free:
discrete enstrophy stability \eqref{eq:enstrophy_discrete} no longer holds exactly, differing through defect facet terms scaling with $(\sigma - \bar{\sigma}) / h \cdot \sym \leftjump \bfu \otimes \bfn \rightjump$\footnote{
    With this modification, one may define a \emph{reduced} broken enstrophy $\bar{\calE}^h(\bfu) \coloneqq \bar{\calD}^h\,[\bfu, \bfu]$.
    Following the proof of \Cref{th:stability}, we see
    \begin{equation}\label{eq:enstrophy_discrete_reduced}
        \partial_t \bar{\calE}^h
            =  \int_\Omega \bfu \cdot (\bfomega \times \curl\bfomega) - \frac{1}{\Re} \|\curl\bfomega\|^2
            -  \frac{2}{\Re} (\sigma - \bar{\sigma}) \int_{\calF^h} \frac{1}{h} \sym \leftjump \bfu \otimes \bfn \rightjump : \sym \leftjump \curl\bfomega \otimes \bfn \rightjump.
    \end{equation}
    The defect in the final term vanishes in the inviscid limit $\Re \to \infty$, where the discrete enstrophy identity \eqref{eq:enstrophy_discrete} therefore holds exactly for $\bar{\calE}^h$.
    In that limit, however, $\sigma$ enters only through the vorticity reconstruction;
    thus the modified discretisation coincides simply with the unmodified penalty method \eqref{eq:semidiscrete_3d_penalty} with $\bar{\sigma}$ in place of $\sigma$.
}.
This modification should be regarded as a pragmatic measure rather than one retaining the structure of \Cref{th:stability}.
We adopt this modification for the 3D penalty method simulations of \Cref{sec:vortex_street_3d}, where we also report the effect of varying $\bar{\sigma}$.

\begin{remark}[Stabilisation through strong and weak conformity]\label{rem:conformity}
    The true remedy is to discretise with a conforming Stokes complex \eqref{eq:stokes_3d_bc}.
    There $\bfu$ is continuous across facets, the jumps vanish identically, and no penalty terms pollute the vorticity.
    The conforming 3D simulations presented in \Cref{sec:hill_vortex} exhibit no such blow-up.
    The penalty method (and with it the reduction to $\bar{\sigma}$ above in the 3D case) is needed in \Cref{sec:vortex_street}, because a \emph{strongly} conforming complex would force $\bfu = \bfzero$ at the vertices of the affinely approximated curved boundary.
    A more appropriate approach there would be a \emph{weakly} conforming Stokes complex \citep{Mardal_Tai_Winther_2002,Tai_Winther_2006}, able to handle such boundaries without resorting to a penalty method.
    We leave this to future work.
\end{remark}

\subsection{Equivalent implementation without $\bfH(\grad \curl)$-conforming or $H^2$-conforming elements}\label{sec:charlie}

In \eqref{eq:semidiscrete_3d}, with the exception of the advective term in \eqref{eq:semidiscrete_3d_c} and the scalar constraints involving $\alpha$ and $\beta$, the vorticity $\bfomega \in \bbW_0$ and the test function $\bfchi \in \bbW_0$ appear exclusively as their curls, $\curl\bfomega$ and $\curl\bfchi$.
Through the exactness of the discrete Stokes complex \eqref{eq:stokes_3d_bc}, we may parametrise $\curl\bfomega$ through the divergence-free kernel in $\bbU_0$.
In principle, provided we can compute $\bfomega$ from $\curl\bfomega$, we need only have computational access to the final two spaces $\bbU$ and $\bbP$ of the Stokes complex \eqref{eq:stokes_3d}.
An identical argument holds in the 2D case \eqref{eq:semidiscrete_2d}.

This idea aligns with one proposed by \citet{Ainsworth_Parker_2024a, Ainsworth_Parker_2024b} for biharmonic-like equations, wherein the authors propose specially constructed reparametrisations that allow one to compute effectively with $H^2$-conforming scalar finite elements while requiring the implementation only of finite elements with at most $H^1$ conformity.

\paragraph{3D.}
Denoting $\bfeta = \curl\bfomega \in \bbU_0$, we propose the following reparametrisation of \eqref{eq:semidiscrete_3d}:
find $(\bfeta, R, \bfu, p) \in \bbU_0 \times \bbP \times \bbU_0 \times \bbP$ such that
\begin{subequations}\label{eq:charlie_sub1}
\begin{align}
    (\bfeta, \bfzeta)  &=  2 (\varepsilon \, \bfu, \varepsilon \bfzeta) + (R, \div \bfzeta),  \\
    0  &=  (\div \bfeta, S),  \\
    (\partial_t \bfu, \bfv)  &=  (\bfu \times \bfomega, \bfv) + \left(\frac{1}{2}|\bfu|^2 + p, \div \bfv\right) - \frac{2}{\Re}(\varepsilon\,\bfu, \varepsilon\,\bfv),  \label{eq:charlie_sub1_c}  \\
    0  &=  (\div \bfu, q),
\end{align}
\end{subequations}
for all $(\bfzeta, S, \bfv, q) \in \bbU_0 \times \bbP \times \bbU_0 \times \bbP$.

The vorticity $\bfomega$ in \eqref{eq:charlie_sub1_c} is related to $\bfeta$ by $\curl\bfomega = \bfeta$.
To complete this system, it is sufficient to be able to compute, for a given divergence-free $\bfeta$, a discretely divergence-free vorticity field $\bfomega$ such that $\curl\bfomega = \bfeta$.
This is possible in the general case, without access to either $\bbA_0$ or $\bbW_0$.
Assume the existence of a second discrete complex, a discrete de Rham complex,
\begin{equation}\label{eq:auxiliary_de_rham}
\begin{tikzcd}
    H^1_0 \arrow{r}{\grad}
        & \bfH_0(\curl) \arrow{r}{\curl}
        & \bfH_0(\div) \arrow{r}{\div}
        & L^2  \\
    \hat{\bbA}_0 \arrow{r}{\grad} \arrow[u]
        & \hat{\bbW}_0 \arrow{r}{\curl} \arrow[u]
        & \hat{\bbU}_0 \arrow{r}{\div} \arrow[u]
        & \hat{\bbP} \arrow[u]
\end{tikzcd},
\end{equation}
such that $\bbU_0 \subset \hat{\bbU}_0$.
Given that $\bfeta$ lies in the divergence-free kernel in $\bbU_0$, it must naturally lie in the divergence-free kernel in $\hat{\bbU}_0$.
Treating $\bfeta \in \bbU_0 \subset \hat{\bbU}_0$ as such, it is then sufficient that we are able to compute a (left) inverse to the $\curl$ in the auxiliary de Rham complex \eqref{eq:auxiliary_de_rham}.
This may be done as follows:
find $(\alpha, \bfomega) \in \hat{\bbA}_0 \times \hat{\bbW}_0$ such that
\begin{subequations}\label{eq:charlie_sub2}
\begin{align}
    (\bfeta, \curl\bfchi)  &= (\curl\bfomega, \curl\bfchi) + (\nabla\alpha, \bfchi),  \\
    0  &=  (\bfomega, \nabla\beta),
\end{align}
\end{subequations}
for all $(\beta, \bfchi) \in \hat{\bbA}_0 \times \hat{\bbW}_0$.
Combining \eqref{eq:charlie_sub1} and \eqref{eq:charlie_sub2}, we arrive at our equivalent semidiscretisation.

\begin{definition}[3D semidiscretisation ($\bfH(\grad \curl)$-free reparametrisation)]
    Find $(\alpha, \bfomega, \bfeta, R, \bfu, p) \in \hat{\bbA}_0 \times \hat{\bbW}_0 \times \bbU_0 \times \bbP \times \bbU_0 \times \bbP$ such that
    \begin{subequations}\label{eq:charlie_3d}
    \begin{align}
        (\bfeta, \curl\bfchi)  &= (\curl\bfomega, \curl\bfchi) + (\nabla\alpha, \bfchi),  \\
        0  &=  (\bfomega, \nabla\beta),  \\
        (\bfeta, \bfzeta)  &=  2 (\varepsilon \, \bfu, \varepsilon \bfzeta) + (R, \div \bfzeta),  \label{eq:charlie_3d_c}  \\
        0  &=  (\div \bfeta, S),  \\
        (\partial_t \bfu, \bfv)  &=  (\bfu \times \bfomega, \bfv) + \left(\frac{1}{2}|\bfu|^2 + p, \div \bfv\right) - \frac{2}{\Re}(\varepsilon\,\bfu, \varepsilon\,\bfv),  \label{eq:charlie_3d_e}  \\
        0  &=  (\div \bfu, q),
    \end{align}
    \end{subequations}
    for all $(\beta, \bfchi, \bfzeta, S, \bfv, q) \in \hat{\bbA}_0 \times \hat{\bbW}_0 \times \bbU_0 \times \bbP \times \bbU_0 \times \bbP$.
\end{definition}

If $\bbA_0 = \hat{\bbA}_0$ and $\bbW_0 \subset \hat{\bbW}_0$, then this represents an exact reparametrisation of the original semidiscretisation \eqref{eq:semidiscrete_3d}.

\paragraph{2D.}
The construction here is similar, relying on the existence of an auxiliary discrete de Rham complex,
\begin{equation}
\begin{tikzcd}
    H^1_0 \arrow{r}{\grad^\perp}
        & \bfH_0(\div) \arrow{r}{\div}
        & L^2  \\
    \hat{\bbW}_0 \arrow{r}{\grad^\perp} \arrow[u]
        & \hat{\bbU}_0 \arrow{r}{\div} \arrow[u]
        & \hat{\bbP} \arrow[u]
\end{tikzcd},
\end{equation}
such that $\bbU_0 \subset \hat{\bbU}_0$.

\begin{definition}[2D semidiscretisation ($H^2$-free reparametrisation)]
    Find $(\omega, \bfeta, R, \bfu, p) \in \hat{\bbW}_0 \times \bbU_0 \times \bbP \times \bbU_0 \times \bbP$ such that
    \begin{subequations}\label{eq:charlie_2d}
    \begin{align}
        - \, (\bfeta, \nabla^\perp\chi)  &= (\nabla\omega, \nabla\chi),  \\
        (\bfeta, \bfzeta)  &=  2 (\varepsilon \, \bfu, \varepsilon \bfzeta) + (R, \div \bfzeta),  \\
        0  &=  (\div \bfeta, S),  \\
        (\partial_t \bfu, \bfv)  &=  - \, (\omega \bfu^\perp, \bfv) + \left(\frac{1}{2}|\bfu|^2 + p, \div \bfv\right) - \frac{2}{\Re}(\varepsilon\,\bfu, \varepsilon\,\bfv),  \\
        0  &=  (\div \bfu, q),
    \end{align}
    \end{subequations}
    for all $(\chi, \bfzeta, S, \bfv, q) \in \hat{\bbW}_0 \times \bbU_0 \times \bbP \times \bbU_0 \times \bbP$.
\end{definition}

Numerical tests on these reparametrisations can be found below in \Cref{sec:kh} for the 2D case, and \Cref{sec:hill_vortex} for the 3D case.

\begin{remark}[Computational overhead of reparametrisation]\label{rem:charlie_overhead}
    While the reparametrisations proposed here \eqrefs{eq:charlie_3d,eq:charlie_2d} are theoretically effective, we note they impose a significant computational overhead.
    Consider for example the 2D Scott--Vogelius complex \eqref{eq:sv_2d}.
    Without reparametrisation, our proposed scheme \eqref{eq:semidiscrete_2d} introduces approximately 14\% more degrees of freedom (DoFs) over a standard unstabilised scheme\footnote{
        As the mesh size approaches $0$, the vorticity space $\bbH\bbC\bbT_3$ requires 3 DoFs per (macro-)cell.
        The velocity space $[\bbC\bbG_2^\rmA]^2$, however, requires 12, and the pressure space $\bbD\bbG_1^\rmA$ requires 9.
        This comes as a consequence of the higher regularity of the spaces earlier in the complex.
    }.
    However, when we apply the reparametrisation \eqref{eq:charlie_2d} in \Cref{sec:kh} to the same complex, it requires approximately 164\% more DoFs in total.
\end{remark}

\subsection{Streamfunction formulation}\label{sec:streamfunction}

Finally, we consider a streamfunction reformulation of the scheme.

\paragraph{3D.}
To reparametrise the velocity in terms of its streamfunction, we write $\bfu = \curl \bfpsi$, $\bfv = \curl \bfphi$ for some $\bfpsi, \bfphi \in \bbW_0$.
Unlike in 2D below, the 3D streamfunction $\bfpsi$ is not unique, but is defined only up to an arbitrary gauge $\nabla\xi \in \bbW_0$.
Similar to the discrete vorticity $\bfomega$, we must therefore introduce a new Lagrange multiplier to enforce uniqueness.

\begin{definition}[3D semidiscretisation (streamfunction)]
    Find $(\alpha, \bfomega, \gamma, \bfpsi) \in \bbA_0 \times \bbW_0 \times \bbA_0 \times \bbW_0$ such that
    \begin{subequations}\label{eq:semidiscrete_3d_streamfunction}
    \begin{align}
        0  &=  (\bfomega, \nabla\beta),  \\
        (\curl\bfomega, \curl\bfchi)  &=  2 (\varepsilon\curl\bfpsi, \varepsilon\curl\bfchi) - (\nabla\alpha, \bfchi),  \\
        0  &=  (\bfpsi, \nabla\delta),  \\
        (\curl\partial_t \bfpsi, \curl\bfphi)  &=  (\curl\bfpsi \times \bfomega, \curl\bfphi) - \frac{2}{\Re}(\varepsilon\curl\bfpsi, \varepsilon\curl\bfphi) - (\nabla\gamma, \bfphi),
    \end{align}
    \end{subequations}
    for all $(\beta, \bfchi, \delta, \bfphi) \in \bbA_0 \times \bbW_0 \times \bbA_0 \times \bbW_0$.
\end{definition}

It is further possible to combine this streamfunction reparametrisation with the penalty method \eqref{eq:semidiscrete_3d_penalty}.
Whether the combined dimensionality of $\bbA_0$ and $\bbW_0$ is lower than that of $\bbU_0$ and $\bbP$, and consequently whether we will see a computational gain through the streamfunction reparametrisation, depends on the complex being used.
For instance, when combined with the penalty method over discrete de Rham complexes, the polynomial complex \eqref{eq:fedr_poly} increases in dimensionality toward the start of the complex, whereas the trimmed complex \eqref{eq:fedr_trimmed} decreases.
A streamfunction reformulation would therefore be preferable in the latter case, but not the former.

\paragraph{2D.}
Here, we may simply write $\bfu = - \nabla^\perp \psi$, $\bfv = - \nabla^\perp \phi$ for some $\psi, \phi \in \bbW_0$.

\begin{definition}[2D semidiscretisation (streamfunction)]
    Find $(\omega, \psi) \in \bbW_0 \times \bbW_0$ such that
    \begin{subequations}\label{eq:semidiscrete_2d_streamfunction}
    \begin{align}
        (\nabla\omega,\nabla \chi)  &=  2 (\varepsilon \, \nabla^\perp\psi, \varepsilon \, \nabla^\perp\chi),  \\
        (\nabla\partial_t \psi, \nabla\phi)  &=  (\omega \nabla\psi, \nabla^\perp\phi) - \frac{2}{\Re}(\varepsilon \, \nabla^\perp\psi, \varepsilon \, \nabla^\perp\phi),
    \end{align}
    \end{subequations}
    for all $(\chi, \phi) \in \bbW_0 \times \bbW_0$.
\end{definition}

It is again possible to combine this reparametrisation with the penalty method \eqref{eq:semidiscrete_2d_penalty}.
As well as reducing the number of variables by 1, the dimensionality of $\bbW_0$ must be lower than the combined size of $\bbU_0$ and $\bbP$.
This reparametrisation therefore necessarily reduces the size of the discrete problem.

\subsection{MEEVC-inspired stabilisation}\label{sec:meevc_stabilisation}

Our proposed semidiscretisation \eqrefs{eq:semidiscrete_3d,eq:semidiscrete_2d} carries two computational overheads that are of little physical interest.
The first comes in 3D in the form of the Lagrange multiplier $\alpha \in \bbA_0$, the sole purpose of which is to impose the discrete divergence-free condition $(\bfomega, \nabla\beta) = 0$ on the vorticity \eqref{eq:semidiscrete_3d_a}.
Moreover, as shown in the proof of \Cref{th:stability}, $\alpha$ is necessarily zero at a solution to \eqref{eq:semidiscrete_3d}.
The second, present in both 3D and 2D whenever the domain carries non-trivial topology (\Cref{sec:periodic_1}), is the further multiplier $\bflambda \in \frakH^1_0$ (or $\lambda \in \frakH^0_0$ in 2D) required to fix the harmonic component of $\bfomega$ \eqref{eq:harmonic_constraint}.
This additionally requires a basis for $\frakH^1_0$ to be computed \emph{a priori}.
Both enlarge the discrete system, the latter furthermore requiring special handling so as not to break the locality of discretisation.

The MEEVC scheme (\citealp{Liu_E_2001}; \citealp[Sec.~7.2]{Cotter_2023}; \citealp{Hanot_2023}) avoids both, by defining its vorticity as an $\bfL^2$ projection of $\curl\bfu$ rather than through a $\curl\curl$ problem.
In 3D it is defined as follows:

\begin{definition}[MEEVC semidiscretisation]
    Find $(\bfomega, \bfu, P) \in \bbW_0 \times \bbU_0 \times \bbP$ such that
    \begin{subequations}\label{eq:meevc_semidiscrete}
    \begin{align}
        (\bfomega, \bfchi)  &=  (\bfu, \curl \bfchi),  \label{eq:meevc_semidiscrete_a}  \\
        (\partial_t \bfu, \bfv)  &=  (\bfu \times \bfomega, \bfv) + (P, \div \bfv) - \frac{1}{\Re}(\curl \bfomega, \bfv),  \\
        0  &=  (\div \bfu, q),
    \end{align}
    \end{subequations}
    for all $(\bfchi, \bfv, q) \in \bbW_0 \times \bbU_0 \times \bbP$.
\end{definition}

Restricting \eqref{eq:meevc_semidiscrete_a} to those test functions with $\curl\bfchi = \bfzero$, the right-hand side vanishes identically, so that $(\bfomega, \bfchi)  =  0$ for all curl-free $\bfchi \in \bbW_0$.
That is, the MEEVC vorticity is necessarily $\bfL^2$-orthogonal to every $\curl$-free function in $\bbW_0$.
By the definition of the discrete harmonic forms, the curl-free subspace of $\bbW_0$ decomposes ($\bfL^2$-orthogonally) as $\nabla\bbA_0 \oplus \frakH^1_0$.
The orthogonality condition in MEEVC therefore delivers both of the constraints we impose by hand:
taking $\bfchi = \nabla\beta$ recovers the discrete divergence-free condition \eqref{eq:semidiscrete_3d_a}, and taking $\bfchi \in \frakH^1_0$ recovers the harmonic constraint \eqref{eq:harmonic_constraint}.

We would like to borrow this behaviour without surrendering the primal enstrophy stability that distinguishes our scheme from MEEVC (\Cref{th:stability}).
A simple proposal is to add a $\delta$-sized component of the MEEVC projection \eqref{eq:meevc_semidiscrete_a} to the equation defining our vorticity, for some chosen $0 < \delta \ll 1$.

\begin{definition}[3D semidiscretisation ($\delta$-stabilised)]
    Find $(\bfomega, \bfu, p) \in \bbW_0 \times \bbU_0 \times \bbP$ such that
    \begin{subequations}\label{eq:semidiscrete_3d_delta}
    \begin{align}
        (\curl \bfomega, \curl \bfchi) + \delta \, (\bfomega, \bfchi)
            &=  2 (\varepsilon \, \bfu, \varepsilon \curl \bfchi) + \delta \, (\bfu, \curl \bfchi),  \label{eq:semidiscrete_3d_delta_a} \\
        (\partial_t \bfu, \bfv)  &=  (\bfu \times \bfomega, \bfv) + \left(\frac{1}{2}|\bfu|^2 + p, \div \bfv\right) - \frac{2}{\Re}(\varepsilon\,\bfu, \varepsilon\,\bfv),  \label{eq:semidiscrete_3d_delta_b} \\
        0  &=  (\div \bfu, q),  \label{eq:semidiscrete_3d_delta_c}
    \end{align}
    \end{subequations}
    for all $(\bfchi, \bfv, q) \in \bbW_0 \times \bbU_0 \times \bbP$.
\end{definition}

In this scheme, $\alpha$ does not appear, and $\bflambda$ is not required.
Just as in the MEEVC scheme, the discrete divergence-free condition on the vorticity $\bfomega$ \eqref{eq:semidiscrete_3d_a} can be seen by considering $\bfchi = \nabla\beta$ for $\beta \in \bbA_0$;
the harmonic constraint \eqref{eq:harmonic_constraint} can be seen by considering $\bfchi \in \frakH^1_0$.
Note that this argument is independent of the size of $\delta$.
The vorticity remains a consistent discrete approximation to $\curl\bfu$, since \eqref{eq:semidiscrete_3d_delta_a} is a linear combination of two relations \eqrefs{eq:semidiscrete_3d_b,eq:meevc_semidiscrete_a} that the exact vorticity $\bfomega = \curl\bfu$ satisfies individually.
As $\delta \to 0$ the scheme approaches our discretisation \eqref{eq:semidiscrete_3d}, while continuing to enforce the discrete divergence-free and harmonic-free constraints implicitly by the argument above.
We adopt this stabilisation, with $\delta = 10^{-5}$, for the 3D simulations of \Cref{sec:vortex_street_3d}.

There are two costs to this stabilisation.
The first is borne by the enstrophy stability \eqref{eq:enstrophy_discrete}.
Repeating the proof of \Cref{th:stability} for \eqref{eq:semidiscrete_3d_delta}, the enstrophy law \eqref{eq:enstrophy_discrete} is recovered up to a perturbation of size $\delta$.

The second issue is that the operator on the left of \eqref{eq:semidiscrete_3d_delta_a} degenerates towards the singular $\curl\curl$ operator as $\delta \to 0$, whose kernel in $\bbW_0$ is precisely the space that $\delta$ was controlling.
The system becomes increasingly ill-conditioned for small $\delta$.
Note, however, that $\delta$-robust preconditioners for operators of exactly this form are available (see e.g.~\citet{Hiptmair_Xu_2007,Brubeck_et_al_2026}), so the ill-conditioning need not be prohibitive in practice.
The choice of $\delta$ is thus a trade-off between the conditioning of the vorticity block and the size of the perturbation to the enstrophy stability.

\begin{remark}[Inconsistent vorticity reconstruction]
    The stabilised scheme \eqref{eq:semidiscrete_3d_delta} could be made yet more efficient by removing the right-hand side $\delta \, (\bfu, \curl \bfchi)$ term in \eqref{eq:semidiscrete_3d_delta_a}.
    In this case, the remaining left-hand side $\delta \, (\bfomega, \bfchi)$ term could be interpreted as a kind of stabilisation (instead of the interpretation as a $\delta$-sized MEEVC component).
    However, this would render $\bfomega$ only a consistent approximation to $\curl\bfu$ in the limit $\delta \to 0$:
    the vorticity reconstruction \eqref{eq:semidiscrete_3d_delta_a} would not in general be solved by $\bfomega = \curl\bfu$.
\end{remark}

The same construction applies in 2D.
There, the multiplier $\alpha$ is already absent, but a harmonic 0-form $\omega_\rmH \in \frakH^0_0$ is present under fully periodic boundary conditions, requiring the circulation $\int_\Omega \omega$ to be fixed (\Cref{sec:periodic_1}).
Adding $\delta(\omega, \chi)$ to the left-hand side of \eqref{eq:semidiscrete_2d_a}, and subtracting $\delta(\bfu, \nabla^\perp \chi)$ from the right removes this requirement by the same argument.
In 2D, however, the constraint is a single scalar condition and is therefore less frequently the dominant concern.

\section{General boundary conditions}\label{sec:bcs}

So far we have only considered free slip boundary conditions \eqref{eq:slip_1} on convex polytopal domains.
This section concerns other general homogeneous and inhomogeneous boundary conditions commonly considered for the Navier--Stokes equations, and curved boundaries which break the equivalence between the zero tangential stress condition $\bft \cdot \varepsilon \, \bfu \cdot \bfn = 0$ \eqref{eq:slip_1} and the condition $\bfomega \times \bfn = \bfzero$ \eqref{eq:slip_2}.
The exactness of the underlying Stokes complex is essential for the proofs of discrete energy and enstrophy stability \eqref{eq:discrete_stability};
boundary conditions must be carefully enforced in such a way as not to impact this.


The general idea is to enforce natural boundary conditions weakly, and essential boundary conditions through Lagrange multipliers;
this latter approach is necessary so as not to lose the energy and enstrophy structures \eqref{eq:discrete_stability}.
We work from the variational problem identified in \Cref{sec:scheme} \eqref{eq:boundary_terms} featuring the boundary components.
In the vorticity constraint \eqref{eq:boundary_terms_a}, we further substitute $\curl\bfu$ for the vorticity $\bfomega$ in the boundary term $\langle \bfomega \cdot \bfn, \beta \rangle$:
\begin{equation}
    - \, \langle \bfomega \cdot \bfn, \beta \rangle
        \mapsto  - \, \langle \curl\bfu \cdot \bfn, \beta \rangle
        =  \langle \bfu \times \bfn, \nabla\beta \rangle,
\end{equation}
where in the second equality we integrate by parts over the boundary\footnote{
    This keeps the boundary contribution conforming even under the penalty method of \Cref{sec:penalty}.
}.

We partition the boundary then into three disjoint regions $\partial\Omega = \Gamma_D \cup \Gamma_S \cup \Gamma_N$.
We denote by $\langle \cdot, \cdot \rangle_D$ the restriction of the boundary pairing $\langle \cdot, \cdot \rangle$ to $\Gamma_D$, and similarly for $\langle \cdot, \cdot \rangle_S$ and $\langle \cdot, \cdot \rangle_N$.
The $L^2$ norms therein are denoted $\|\cdot\|_D$, $\|\cdot\|_S$ and $\|\cdot\|_N$.
On these three regions we consider Dirichlet, partial/free slip, and traction boundary conditions, respectively, as we now detail.


\vspace{3mm}

On $\Gamma_D$, we consider the Dirichlet boundary condition
\begin{equation}\label{eq:dirichlet}
    \bfu = \bfu_{\mathrm{BC}}
\end{equation}
for some prescribed $\bfu_{\mathrm{BC}}$.

On $\Gamma_S$, let $\bfu_{||} = \bfu - (\bfu \cdot \bfn) \bfn$ (satisfying $\bfu_{||} \cdot \bfn = 0$) denote the tangential component of $\bfu$.
We consider the general partial slip boundary condition
\begin{equation}\label{eq:partial}
    \bfu \cdot \bfn = u_{\perp,\mathrm{BC}},  \qquad
    \bft \cdot \varepsilon\,\bfu \cdot \bfn = \gamma \bft \cdot (\bfu_{||,\mathrm{BC}} - \bfu_{||}),
\end{equation}
where $\gamma \ge 0$ is a friction parameter, for a prescribed tangential velocity $\bfu_{||,\mathrm{BC}}$ satisfying $\bfu_{||,\mathrm{BC}} \cdot \bfn = 0$, and a prescribed normal velocity $u_{\perp,\mathrm{BC}}$ (both often zero).
When $\gamma = 0$ and $u_{\perp,\mathrm{BC}} = 0$, this reduces to the free slip condition \eqref{eq:slip_1}.

Lastly, on $\Gamma_N$ we consider traction boundary conditions,
\begin{equation}\label{eq:traction}
    \frac{2}{\Re}\varepsilon\,\bfu \cdot \bfn - p \bfn = - \, p_{\mathrm{BC}} \bfn,
\end{equation}
the natural Neumann boundary condition, appropriate at outlets where the flow is fully developed.
Here $p_{\mathrm{BC}}$ is a prescribed reference pressure, potentially non-uniform over $\Gamma_N$.

\vspace{3mm}

It is straightforward to enforce the tangential--normal stress condition on $\Gamma_S$ and the traction boundary condition on $\Gamma_N$ weakly, modifying the boundary terms in the momentum equation \eqref{eq:boundary_terms_c} to the following:
\begin{subequations}
\begin{align}
    (\partial_t \bfu, \bfv)  &=  \Big[ \, (\bfu \times \bfomega, \bfv) + \Big(\frac{1}{2}|\bfu|^2 + p, \div \bfv\Big) - \frac{2}{\Re}(\varepsilon\,\bfu, \varepsilon\,\bfv) \, \Big]  \notag \\
        &\qquad\quad + \Big[ - \Big\langle \frac{1}{2}|\bfu|^2, \bfv \cdot \bfn \Big\rangle + \Big\langle \frac{2}{\Re}\varepsilon\,\bfu \cdot \bfn - p \bfn, \bfv \Big\rangle_D  \notag  \\
        &\qquad\qquad\qquad\qquad + \Big\langle \frac{2\gamma}{\Re}(\bfu_{||,\mathrm{BC}} - \bfu_{||}) + \Big(\bfn \cdot \frac{2}{\Re}\varepsilon\,\bfu \cdot \bfn - p\Big)\,\bfn, \bfv \Big\rangle_S \! - \langle p_{\mathrm{BC}} \bfn, \bfv \rangle_N \, \Big].  \label{eq:boundary_terms_weak_c}  \\
\intertext{
With a little manipulation, similar modifications can be made in the vorticity equation \eqref{eq:boundary_terms_b}:
}
    (\curl \bfomega, \curl \bfchi)  &=  \Big[ \, 2 (\varepsilon \, \bfu, \varepsilon \curl \bfchi) - (\nabla \alpha, \bfchi) \, \Big]  \notag  \\
        &\qquad + \Re \, \Big[ - \langle p \bfn, \curl \bfchi \rangle - \Big\langle \frac{2}{\Re}\varepsilon\,\bfu \cdot \bfn - p \bfn, \curl \bfchi \Big\rangle_D  \notag  \\
        &\qquad\qquad\quad - \Big\langle \frac{2\gamma}{\Re}(\bfu_{||,\mathrm{BC}} - \bfu_{||}) + \Big(\bfn \cdot \frac{2}{\Re}\varepsilon\,\bfu \cdot \bfn - p\Big)\,\bfn, \curl \bfchi \Big\rangle_S \! + \langle p_{\mathrm{BC}} \bfn, \curl \bfchi \rangle_N \, \Big].  \label{eq:boundary_terms_weak_b}
\end{align}
\end{subequations}

Were we now to proceed with imposing the remaining velocity data on $\Gamma_D$ and $\Gamma_S$ strongly, this would obstruct the exactness arguments underlying \Cref{th:stability};
strongly constraining $\bfv$ on the boundary would preclude the choice $\bfv = \curl\bfomega$ used in the proof of enstrophy stability.
Instead we impose these boundary conditions in a variational sense.
Let $\bbU_D$ denote the vector-valued trace of $\bbU$ on $\Gamma_D$, and $\bbU_S$ the normal trace of $\bbU$ on $\Gamma_S$.
We require that
\begin{subequations}
\begin{align}
    \langle \bfu_{\mathrm{BC}}, \bfrho \rangle_D  &=  \langle \bfu, \bfrho \rangle_D,  \\
    \langle u_{\perp,\mathrm{BC}}, r \rangle_S  &=  \langle \bfu \cdot \bfn, r \rangle_S,
\end{align}
\end{subequations}
for all $(\bfrho, r) \in \bbU_D \times \bbU_S$.

To keep the number of equations and unknowns the same, we must introduce Lagrange multipliers $(\bfsigma, s) \in \bbU_D \times \bbU_S$ in their respective spaces into the momentum equation \eqref{eq:boundary_terms_weak_c}:
\begin{subequations}
\begin{align}
    (\partial_t \bfu, \bfv)  &=  \Big[ \, (\bfu \times \bfomega, \bfv) + \Big(\frac{1}{2}|\bfu|^2 + p, \div \bfv\Big) - \frac{2}{\Re}(\varepsilon\,\bfu, \varepsilon\,\bfv) \, \Big]  \notag \\
        &\qquad\quad + \Big[ - \Big\langle \frac{1}{2}|\bfu|^2, \bfv \cdot \bfn \Big\rangle + \langle \bfsigma, \bfv \rangle_D + \Big\langle \frac{2\gamma}{\Re}(\bfu_{||,\mathrm{BC}} - \bfu_{||}) + s\bfn, \bfv \Big\rangle_S \! - \langle p_{\mathrm{BC}} \bfn, \bfv \rangle_N \, \Big].  \\
\intertext{
Comparing like terms on the boundary with \eqref{eq:boundary_terms_weak_c}, we see that the Lagrange multipliers $\bfsigma$ and $s$ represent discrete approximations, respectively, to the traction $\frac{2}{\Re}\varepsilon\,\bfu \cdot \bfn - p \bfn$ on $\Gamma_D$, as $\langle \bfsigma, \bfv \rangle_D$ stands in for $\langle \tfrac{2}{\Re}\varepsilon\,\bfu \cdot \bfn - p \bfn, \bfv \rangle_D$, and the normal traction $\bfn \cdot \frac{2}{\Re}\varepsilon\,\bfu \cdot \bfn - p$ on $\Gamma_S$, as $\langle s\bfn, \bfv \rangle_S$ stands in for $\langle (\bfn \cdot \frac{2}{\Re}\varepsilon\,\bfu \cdot \bfn - p)\bfn, \bfv \rangle_S$.
Accordingly, we substitute in $\bfsigma$ and $s$ in the analogous places in the vorticity equation \eqref{eq:boundary_terms_weak_b}:
}
    (\curl \bfomega, \curl \bfchi)  &=  \Big[ \, 2 (\varepsilon \, \bfu, \varepsilon \curl \bfchi) - (\nabla \alpha, \bfchi) \, \Big]  \notag  \\
        &\qquad + \Re \, \Big[ - \langle p \bfn, \curl \bfchi \rangle - \langle \bfsigma, \curl \bfchi \rangle_D  \notag  \\
        &\qquad\qquad\qquad\qquad\qquad\qquad - \Big\langle \frac{2\gamma}{\Re}(\bfu_{||,\mathrm{BC}} - \bfu_{||}) + s\bfn, \curl \bfchi \Big\rangle_S \! + \langle p_{\mathrm{BC}} \bfn, \curl \bfchi \rangle_N \, \Big],
\end{align}
\end{subequations}
Together, this gives the following semidiscretisation, a modification of \eqref{eq:semidiscrete_3d}.

\begin{definition}[3D semidiscretisation (general boundary conditions)]
    Find $(\alpha, \bfomega, \bfu, p, \bfsigma, s) \in \bbA \times \bbW \times \bbU \times \bbP \times \bbU_D \times \bbU_S$ such that
    \begin{subequations}\label{eq:semidiscrete_3d_bcs}
    \begin{align}
        0  &=  (\bfomega, \nabla \beta) + \langle \bfu \times \bfn, \nabla\beta \rangle,  \label{eq:semidiscrete_3d_bcs_a}  \\
        (\curl \bfomega, \curl \bfchi)  &=  \Big[ \, 2 (\varepsilon \, \bfu, \varepsilon \curl \bfchi) - (\nabla \alpha, \bfchi) \, \Big]  \notag  \\
            &\qquad + \Re \, \Big[ - \langle p \bfn, \curl \bfchi \rangle - \langle \bfsigma, \curl \bfchi \rangle_D  \notag  \\
            &\qquad\qquad\qquad\qquad\qquad\qquad - \Big\langle \frac{2\gamma}{\Re}(\bfu_{||,\mathrm{BC}} - \bfu_{||}) + s \bfn, \curl \bfchi \Big\rangle_S + \langle p_{\mathrm{BC}} \bfn, \curl \bfchi \rangle_N \, \Big],  \label{eq:semidiscrete_3d_bcs_b}  \\
        (\partial_t \bfu, \bfv)  &=  \Big[ \, (\bfu \times \bfomega, \bfv) + \left(\frac{1}{2}|\bfu|^2 + p, \div \bfv\right) - \frac{2}{\Re}(\varepsilon\,\bfu, \varepsilon\,\bfv) \, \Big]  \notag \\
            &\qquad\quad + \Big[ - \left\langle \frac{1}{2}|\bfu|^2, \bfv \cdot \bfn \right\rangle + \langle \bfsigma, \bfv \rangle_D + \Big\langle \frac{2\gamma}{\Re}(\bfu_{||,\mathrm{BC}} - \bfu_{||}) + s \bfn, \bfv \Big\rangle_S \! - \langle p_{\mathrm{BC}} \bfn, \bfv \rangle_N \, \Big],  \label{eq:semidiscrete_3d_bcs_c}  \\
        0  &=  (\div \bfu, q),  \\
        \langle \bfu_{\mathrm{BC}}, \bfrho \rangle_D  &=  \langle \bfu, \bfrho \rangle_D,  \label{eq:semidiscrete_3d_bcs_e}  \\
        \langle u_{\perp,\mathrm{BC}}, r \rangle_S  &=  \langle \bfu \cdot \bfn, r \rangle_S,  \label{eq:semidiscrete_3d_bcs_f}
    \end{align}
    \end{subequations}
    for all $(\beta, \bfchi, \bfv, q, \bfrho, r) \in \bbA \times \bbW \times \bbU \times \bbP \times \bbU_D \times \bbU_S$.
\end{definition}
As for the convex polytopal free-slip case \eqref{eq:semidiscrete_2d}, the 2D discretisation is largely similar:
(i)~the vorticity $\omega$ is a scalar, with (ii)~no Lagrange multiplier $\alpha$.
Recalling that the 2D analogue of $\curl$ acting on the vorticity is $\omega \mapsto - \nabla^\perp\omega$, the boundary contributions in the vorticity equation appear with the opposite sign.

\begin{definition}[2D semidiscretisation (general boundary conditions)]
    Find $(\omega, \bfu, p, \bfsigma, s) \in \bbW \times \bbU \times \bbP \times \bbU_D \times \bbU_S$ such that
    \begin{subequations}\label{eq:semidiscrete_2d_bcs}
    \begin{align}
        (\nabla \omega, \nabla \chi)  &=  - \, 2 (\varepsilon \, \bfu, \varepsilon \, \nabla^\perp \chi)  \notag  \\
            &\qquad + \Re \, \Big[ \, \langle p \bfn, \nabla^\perp \chi \rangle + \langle \bfsigma, \nabla^\perp \chi \rangle_D + \Big\langle \frac{2\gamma}{\Re}(\bfu_{||,\mathrm{BC}} - \bfu_{||}) + s \bfn, \nabla^\perp \chi \Big\rangle_S \! - \langle p_{\mathrm{BC}} \bfn, \nabla^\perp \chi \rangle_N \, \Big],  \label{eq:semidiscrete_2d_bcs_a}  \\
        (\partial_t \bfu, \bfv)  &=  \Big[ - \, (\omega \bfu^\perp, \bfv) + \left(\frac{1}{2}|\bfu|^2 + p, \div \bfv\right) - \frac{2}{\Re}(\varepsilon\,\bfu, \varepsilon\,\bfv) \, \Big]  \notag \\
            &\qquad\qquad + \Big[ - \left\langle \frac{1}{2}|\bfu|^2, \bfv \cdot \bfn \right\rangle + \langle \bfsigma, \bfv \rangle_D + \Big\langle \frac{2\gamma}{\Re}(\bfu_{||,\mathrm{BC}} - \bfu_{||}) + s \bfn, \bfv \Big\rangle_S \! - \langle p_{\mathrm{BC}} \bfn, \bfv \rangle_N \, \Big],  \label{eq:semidiscrete_2d_bcs_b}  \\
        0  &=  (\div \bfu, q),  \\
        \langle \bfu_{\mathrm{BC}}, \bfrho \rangle_D  &=  \langle \bfu, \bfrho \rangle_D,  \label{eq:semidiscrete_2d_bcs_d}  \\
        \langle u_{\perp,\mathrm{BC}}, r \rangle_S  &=  \langle \bfu \cdot \bfn, r \rangle_S,  \label{eq:semidiscrete_2d_bcs_e}
    \end{align}
    \end{subequations}
    for all $(\chi, \bfv, q, \bfrho, r) \in \bbW \times \bbU \times \bbP \times \bbU_D \times \bbU_S$.
\end{definition}

Note that nothing obliges us to treat the whole boundary the same way.
As in the introductory setting of \Cref{sec:scheme}, where the boundary is a flat, free slip wall, the conditions $\bfu\cdot\bfn = 0$ and $\bfomega\times\bfn = \bfzero$ may still be imposed strongly, as in \eqrefs{eq:semidiscrete_3d,eq:semidiscrete_2d};
only the remaining portion need carry the multipliers of \eqrefs{eq:semidiscrete_3d_bcs,eq:semidiscrete_2d_bcs}.
We do exactly this for the flow past an obstacle in \Cref{sec:vortex_street}, where the channel walls are free slip but the obstacle is not.

\subsection{Energy \& enstrophy stability}

Before considering the discrete preservation of energy and enstrophy, we note that the partial slip boundary conditions necessitate a new definition for the enstrophy,
\begin{equation}\label{eq:slip_enstrophy}
    \calE(\bfu)
        \coloneqq  \|\varepsilon\,\bfu\|^2 + \gamma \|\bfu_{||} - \bfu_{||,\mathrm{BC}}\|_S^2,
\end{equation}
where $\|\cdot\|_S$ denotes the $L^2$ norm on the slip surface $\Gamma_S$; the second term accounts for the drag forces.

Under these general boundary conditions, the continuous evolution laws \eqrefs{eq:energy,eq:enstrophy} acquire corresponding boundary contributions from the energy and enstrophy fluxes into and out of the domain.
On the continuous level, testing the momentum equation \eqref{eq:navier-stokes_1} against $\bfu$ yields the energy identity:
\begin{subequations}\label{eq:energy_bcs}
\begin{align}
    \partial_t \calK
        &=  - \, \frac{2}{\Re}\|\varepsilon\,\bfu\|^2 + \left[ - \oint_{\partial\Omega} \frac{1}{2}|\bfu|^2\bfu\cdot\bfn + \oint_{\partial\Omega} \Big(\frac{2}{\Re}\varepsilon\,\bfu \cdot \bfn - p \bfn\Big)\cdot\bfu \right]  \\
        &=  - \, \frac{2}{\Re}\calE(\bfu) + \left[ - \oint_{\partial\Omega} \frac{1}{2}|\bfu|^2\bfu\cdot\bfn + \int_{\Gamma_D} \Big(\frac{2}{\Re}\varepsilon\,\bfu \cdot \bfn - p \bfn\Big)\cdot\bfu_{\mathrm{BC}}\right.  \notag  \\
        &\qquad\qquad\qquad\qquad \left. + \int_{\Gamma_S} \Big[ \frac{2\gamma}{\Re}(\bfu_{||,\mathrm{BC}} - \bfu_{||})\cdot\bfu_{||,\mathrm{BC}} + \Big(\bfn \cdot \frac{2}{\Re}\varepsilon\,\bfu \cdot \bfn - p\Big) u_{\perp,\mathrm{BC}} \Big] - \int_{\Gamma_N} p_{\mathrm{BC}}\bfu\cdot\bfn \right]\!.
\end{align}
\end{subequations}
Assuming the boundary data $\bfu_{\mathrm{BC}}$, $\bfu_{||,\mathrm{BC}}$, $u_{\perp,\mathrm{BC}}$ on $\bfu$ are all constant in time, testing against $\Delta\bfu$ yields the enstrophy identity:
\begin{align}\label{eq:enstrophy_bcs}
    \partial_t \calE
        &=  \int_\Omega \bfu\cdot(\bfomega\times\curl\bfomega) - \frac{1}{\Re}\|\Delta\bfu\|^2 + \Re \int_{\Gamma_N} (p - p_{\mathrm{BC}})\,\dot{\bfu}\cdot\bfn + \oint_{\partial\Omega} \Big( \frac{1}{2}|\bfu|^2 + p \Big) \, \Delta\bfu \cdot \bfn,
\end{align}
where again this is the modified enstrophy \eqref{eq:slip_enstrophy}.

\begin{theorem}[Energy \& enstrophy stability (general boundary conditions)]\label{th:stability_bcs}
    Assuming the boundary data $\bfu_{\mathrm{BC}}$, $\bfu_{||,\mathrm{BC}}$, $u_{\perp,\mathrm{BC}}$ on $\bfu$ are all constant in time, the Navier--Stokes semidiscretisations in 3D \eqref{eq:semidiscrete_3d_bcs} and 2D \eqref{eq:semidiscrete_2d_bcs} with boundary conditions preserve discrete forms of the energy \eqref{eq:energy_bcs} and enstrophy \eqref{eq:enstrophy_bcs} laws:
    \begin{subequations}\label{eq:discrete_stability_bcs}
    \begin{align}
        \partial_t \calK
            &=  - \, \frac{2}{\Re}\calE(\bfu) + \left[ - \oint_{\partial\Omega} \frac{1}{2}|\bfu|^2\bfu\cdot\bfn + \int_{\Gamma_D} \bfsigma\cdot\bfu_{\mathrm{BC}}\right.  \notag  \\
            &\qquad\qquad\qquad\qquad\qquad\qquad\qquad \left.+ \int_{\Gamma_S} \Big[ \frac{2\gamma}{\Re}(\bfu_{||,\mathrm{BC}} - \bfu_{||})\cdot\bfu_{||,\mathrm{BC}} + s u_{\perp,\mathrm{BC}} \Big] - \int_{\Gamma_N} p_{\mathrm{BC}}\bfu\cdot\bfn \right]\!,  \label{eq:energy_discrete_bcs}  \\
        \partial_t \calE
            &=  \int_\Omega \bfu\cdot(\bfomega\times\curl\bfomega) - \frac{1}{\Re}\|\curl\bfomega\|^2 + \Re \int_{\Gamma_N} (p - p_{\mathrm{BC}})\,\dot{\bfu}\cdot\bfn - \oint_{\partial\Omega} \Big( \frac{1}{2}|\bfu|^2 + p \Big) \, \curl\bfomega \cdot \bfn.  \label{eq:enstrophy_discrete_bcs}
    \end{align}
    \end{subequations}
    Again in 2D the advective term $\int_\Omega \bfu \cdot (\bfomega \times \curl\bfomega)$ in the discrete enstrophy law \eqref{eq:enstrophy_discrete_bcs} vanishes.
\end{theorem}

\begin{proof}
    After considering $\bfrho = \bfsigma$ in \eqref{eq:semidiscrete_3d_bcs_e} and $r = s$ in \eqref{eq:semidiscrete_3d_bcs_f}, the proof here is largely identical to that of \Cref{th:stability}, considering the same test functions, using the same exactness properties, with many of the same cancellations.
\end{proof}

\begin{remark}[Enstrophy flux on the boundary]\label{rem:enstrophy_flux}
    The flux $\oint_{\partial\Omega} (\tfrac{1}{2}|\bfu|^2 + p)\,\curl\bfomega\cdot\bfn$ appearing in \eqrefs{eq:enstrophy_bcs,eq:enstrophy_discrete_bcs} in general has unknown sign, implying that under general boundary conditions\footnote{
        Integration by parts over the boundary in the continuous case gives
        \begin{equation}
            \oint_{\partial\Omega} \Big(\frac{1}{2}|\bfu|^2 + p\Big)\,\curl\bfomega\cdot\bfn
                = - \oint_{\partial\Omega} \nabla \Big(\frac{1}{2}|\bfu|^2 + p\Big) \cdot (\bfomega\times\bfn).
        \end{equation}
        In the polytopal free slip case of \Cref{sec:scheme}, $\bfomega\times\bfn = \bfzero$ \eqref{eq:slip_2}, justifying why the term is absent from \Cref{th:stability}, and why it appears here.
    } the enstrophy is not necessarily dissipated, even in 2D and even with $\Gamma_N = \emptyset$.
    We stress that this is not an artefact of the discretisation: the same term appears in the continuous identity \eqref{eq:enstrophy_bcs}, and our proposed scheme \eqrefs{eq:semidiscrete_3d_bcs,eq:semidiscrete_2d_bcs} reproduces it faithfully.
    The discretisation still preserves the \emph{evolution} of the enstrophy, rather than a monotone dissipation law, which we find nevertheless to have stabilising effects on our numerical results (see \Cref{sec:simulations} below).
\end{remark}
\begin{remark}[Local structures]\label{rem:local}
    The machinery of this section localises in a natural way.
    By restricting to subdomains within $\Omega$, local versions of the energy \eqref{eq:energy} and enstrophy \eqref{eq:enstrophy_both} structures may be observed, dictating how energy and enstrophy are transported throughout the domain via their respective fluxes.
    These local structures are often more important for the qualitative behaviour of solutions than their global counterparts, as they prohibit, for example, vortices from instantly transferring to a different position in the domain, a behaviour that would violate the local enstrophy structure but not the global one.
    Local, cellwise versions of the stability results \eqref{eq:discrete_stability} in \Cref{th:stability}, \eqref{eq:enstrophy_stretching} in \Cref{th:stretching}, and \eqref{eq:discrete_stability_bcs} in \Cref{th:stability_bcs} can be proven for the proposed semidiscretisations by ``breaking'' the test spaces, enforcing the required continuity on cell interfaces through Lagrange multipliers as would be done in hybridisation or domain decomposition \citep{Berchenko-Kogan_Stern_2021a,Berchenko-Kogan_Stern_2021b}.
    Practically, these Lagrange multipliers play an identical role to those introduced above in \eqrefs{eq:semidiscrete_3d_bcs,eq:semidiscrete_2d_bcs}, approximating traction forces on the facets, and giving local cellwise energy and enstrophy stability similar to \eqref{eq:discrete_stability_bcs} in \Cref{th:stability_bcs}.
    These preserved discrete local structures aid in accurately modelling the motion of vortices, e.g.~in the vortex shedding demonstrations in \Cref{sec:vortex_street} below.
\end{remark}

\subsection{Periodic boundary conditions \& non-trivial topologies}\label{sec:periodic_2}


As established in \Cref{sec:periodic_1}, periodic boundaries imply non-trivial topologies, and non-trivial topologies give rise to non-trivial harmonic forms.
This causes the Stokes complex \eqrefs{eq:stokes_3d_bc,eq:stokes_2d_bc} to lose exactness.

For the 3D scheme \eqref{eq:semidiscrete_3d}, recall that we found the harmonic 0-, 2- and 3-forms to be relatively benign.
However, the presence of harmonic 1-forms caused the subsystem \eqrefs{eq:semidiscrete_3d_a,eq:semidiscrete_3d_b} defining the discrete vorticity $\bfomega$ to be under-determined, impacting the evolution of the velocity $\bfu$ through the presence of $\bfomega$ in the advective term $(\bfu\times\bfomega, \bfv)$ in \eqref{eq:semidiscrete_3d_c}.
This was resolved through a Lagrange multiplier $\bflambda \in \frakH^1_0$, fixing the harmonic component of $\bfomega$ to be zero by enforcing $(\bfomega, \bfmu) = 0$ for all $\bfmu \in \frakH^1_0$.
For the 2D scheme \eqref{eq:semidiscrete_2d}, harmonic 1- and 2-forms were similarly found to be relatively benign:
we were able to restore well-posedness for the discrete vorticity $\omega$ in the presence of harmonic 0-forms by enforcing $\int_\Omega \omega = 0$.

In large part, everything discussed in \Cref{sec:periodic_1} carries over to the general boundary condition discretisation.
In particular, the majority of the work required is in resolving a nullspace in the definition of $\bfomega$.
However, since our domains of consideration need no longer necessarily be convex, non-trivial topologies may arise even without periodic boundary conditions, when corresponding topological defects (e.g.~holes, tunnels, voids) are present.
Notably also, the relevant harmonic spaces change:
boundary data are no longer imposed strongly, with the general boundary condition scheme \eqrefs{eq:semidiscrete_3d_bcs,eq:semidiscrete_2d_bcs} using the \emph{unconstrained} Stokes complex \eqrefs{eq:stokes_3d,eq:stokes_2d}.
The relevant harmonic spaces are $\frakH^k$ (as opposed to $\frakH^k_0$), reversing the Betti numbers $b_k$ \eqref{eq:betti}.
A topological defect that troubled one variable in \eqrefs{eq:semidiscrete_3d,eq:semidiscrete_2d} before may now trouble another in \eqrefs{eq:semidiscrete_3d_bcs,eq:semidiscrete_2d_bcs}, while one that was harmless may no longer be so.

\paragraph{The Lagrange multiplier $\alpha$ (0-forms in 3D).}
With no boundary conditions enforced strongly, the harmonic nullspace $\frakH^0$ of constant functions will always leave $\alpha$ ill-defined.
This can be resolved as above by either (i)~adding the term $(\alpha, \beta)$ to \eqref{eq:semidiscrete_3d_bcs_a}, or (ii)~reporting the constant nullspace to the solver.

\paragraph{The vorticity $\bfomega$ (1-forms in 3D, 0-forms in 2D).}
In 3D, a new linearly independent harmonic 1-form in $\frakH^1$ is introduced for each tunnel through the domain (i.e.~$b_1$ quantifies the number of tunnels).
Similar to the zero-boundary-condition harmonic 1-forms $\frakH^1_0$ considered in \Cref{sec:periodic_1}, these cause an issue for the well-definedness of $\bfomega$ in \eqref{eq:semidiscrete_3d_bcs} which requires resolution.
Consider first a general curl-free function $\bfmu \in \bfH(\curl)$.
Applying integration by parts over $\Omega$,
\begin{equation}\label{eq:harmonic_boundary}
    (\curl\bfu, \bfmu)
        =  \langle \bfu, \bfmu \times \bfn \rangle + (\bfu, \curl\bfmu)
        =  \langle \bfu, \bfmu \times \bfn \rangle.
\end{equation}
For $\bfmu \in \frakH^1_0$, we see $(\curl\bfu, \bfmu) = 0$, as $\bfmu \times \bfn = \bfzero$ on the boundary $\partial\Omega$.
In \Cref{sec:periodic_1}, we observed that the subsystem \eqrefs{eq:semidiscrete_3d_a,eq:semidiscrete_3d_b} defining the discrete vorticity $\bfomega$ in 3D had a nullspace corresponding to the harmonic 1-forms $\frakH^1_0 \subset \bbW_0$ with zero boundary conditions.
Consistent with $(\curl\bfu, \bfmu) = 0$, well-posedness was restored by enforcing the harmonic component of $\bfomega$ to be zero via a Lagrange multiplier, specifying that $(\bfomega, \bfmu) = 0$ for all $\bfmu \in \frakH^1_0$.

Under more general boundary conditions \eqref{eq:semidiscrete_3d_bcs}, the subsystem \eqrefs{eq:semidiscrete_3d_bcs_a,eq:semidiscrete_3d_bcs_b} defining the discrete vorticity $\bfomega$ has a similar nullspace:
the harmonic 1-forms $\frakH^1 \subset \bbW$ \emph{without} zero boundary conditions.
Inspired by \eqref{eq:harmonic_boundary}, the solution is similar:
fix $(\bfomega, \bfmu) = \langle \bfu, \bfmu \times \bfn \rangle$ for all $\bfmu \in \frakH^1$.
Practically, this requires (i)~introducing a new variable $\bflambda \in \frakH^1$ into \eqref{eq:semidiscrete_3d_bcs}, modifying \eqref{eq:semidiscrete_3d_bcs_b} through a Lagrange multiplier term,
\begin{subequations}\label{eq:vorticity_constraint_bcs}
\begin{equation}
    (\curl \bfomega, \curl \bfchi)  =  \Big[ \, 2 (\varepsilon \, \bfu, \varepsilon \curl \bfchi) - (\nabla \alpha, \bfchi) \, \Big] + \cdots + (\bflambda, \bfchi),
\end{equation}
and (ii)~imposing the constraint
\begin{equation}\label{eq:circulation_3d}
    (\bfomega, \bfmu) =  \langle \bfu, \bfmu \times \bfn \rangle
\end{equation}
\end{subequations}
for all $\bfmu \in \frakH^1$.
As in \Cref{sec:periodic_1}, $\frakH^1$ is of a limited finite dimension ($\dim\frakH^1 = b_1$), implying that a basis for $\frakH^1$ may be computed \emph{a priori}.
Under Dirichlet boundary conditions, where the tangential component of $\bfu$ is prescribed on $\partial\Omega$, the right-hand side of \eqref{eq:circulation_3d} reduces to imposed boundary data;
where only the normal component is prescribed, as on the partial and free slip boundaries, it does not.

\begin{remark}[Harmonic 1-forms and tunnels]\label{rem:tunnels}
    As stated above, each tunnel through the domain can be associated with a discrete harmonic 1-form in $\frakH^1$, a curl-free and discretely divergence-free function flowing in the pipe around that tunnel.
    The constraint \eqref{eq:circulation_3d} imposes (in an appropriate discrete sense) a constraint from Stokes' theorem:
    the vorticity flux through each pipe must equal the total flow of velocity around that pipe's circumference (\Cref{fig:torus}).
\end{remark}

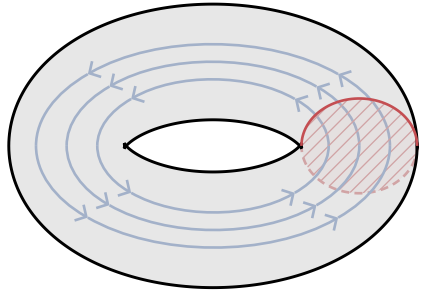
\begin{figure}[pos=!ht]
    \centering
    \begin{tikzpicture}[scale=0.9, >=stealth,
            edge/.style  = {line width=1.1pt, black},
            glass/.style = {fill=black!15, fill opacity=0.62},
            flow/.style  = {seabornblue, line width=1pt, -{Straight Barb[length=1.1mm, width=2.4mm]}},
            circ/.style  = {seabornred, line width=1pt},
            circh/.style = {seabornred, line width=1pt, densely dashed},
            flowp/.style = {seabornblue, line width=1pt},
            cut/.style   = {seabornred, line width=0.5pt}]

        \def\torus#1{plot[variable=\t, domain=0:360, samples=145, smooth cycle]
            ({(2.15 + #1*0.85*cos(atan(1.428148*sin(\t)))) * cos(\t)},
            {(2.15 + #1*0.85*cos(atan(1.428148*sin(\t)))) * sin(\t) * 0.573576
              + #1*0.85*sin(atan(1.428148*sin(\t))) * 0.819152})}

        \def\pipe#1#2{plot[variable=\p, domain=#1:#2, samples=60]
            ({2.15 + 0.85*cos(\p)}, {0.85*sin(\p) * 0.819152})}

        %
        \draw[flow]  (1.7,0)       arc[start angle=  0, end angle= 45, x radius=1.7, y radius=0.9751];
        \draw[flow]  (1.2021, 0.6895)  arc[start angle= 45, end angle=135, x radius=1.7, y radius=0.9751];
        \draw[flowp] (-1.2021, 0.6895) arc[start angle=135, end angle=180, x radius=1.7, y radius=0.9751];
        \draw[flow]  (2.15,0)       arc[start angle=  0, end angle= 45, x radius=2.15, y radius=1.2333];
        \draw[flow]  (1.5203, 0.8721)  arc[start angle= 45, end angle=135, x radius=2.15, y radius=1.2333];
        \draw[flowp] (-1.5203, 0.8721) arc[start angle=135, end angle=180, x radius=2.15, y radius=1.2333];
        \draw[flow]  (2.6,0)       arc[start angle=  0, end angle= 45, x radius=2.6, y radius=1.4913];
        \draw[flow]  (1.8385, 1.0545)  arc[start angle= 45, end angle=135, x radius=2.6, y radius=1.4913];
        \draw[flowp] (-1.8385, 1.0545) arc[start angle=135, end angle=180, x radius=2.6, y radius=1.4913];
        \draw[circh] \pipe{360}{180};

        \begin{scope}
            \clip (2.15,0) ellipse[x radius=0.85, y radius=0.696];
            \fill[seabornred, fill opacity=0.10] (2.15,0) ellipse[x radius=0.85, y radius=0.696];
            \foreach \c in {-3.7,-3.5,...,-0.6} { \draw[cut] (1.0,{1.0+\c}) -- (3.3,{3.3+\c}); }
        \end{scope}

        \draw[flow]  (-1.7,0)      arc[start angle=180, end angle=225, x radius=1.7, y radius=0.9751];
        \draw[flow]  (-1.2021,-0.6895) arc[start angle=225, end angle=315, x radius=1.7, y radius=0.9751];
        \draw[flowp] (1.2021,-0.6895)  arc[start angle=315, end angle=360, x radius=1.7, y radius=0.9751];
        \draw[flow]  (-2.15,0)      arc[start angle=180, end angle=225, x radius=2.15, y radius=1.2333];
        \draw[flow]  (-1.5203,-0.8721) arc[start angle=225, end angle=315, x radius=2.15, y radius=1.2333];
        \draw[flowp] (1.5203,-0.8721)  arc[start angle=315, end angle=360, x radius=2.15, y radius=1.2333];
        \draw[flow]  (-2.6,0)      arc[start angle=180, end angle=225, x radius=2.6, y radius=1.4913];
        \draw[flow]  (-1.8385,-1.0545) arc[start angle=225, end angle=315, x radius=2.6, y radius=1.4913];
        \draw[flowp] (1.8385,-1.0545)  arc[start angle=315, end angle=360, x radius=2.6, y radius=1.4913];

        \fill[glass, even odd rule] \torus{1} \torus{-1};
        \draw[edge] \torus{1};
        \draw[edge] \torus{-1};

        \draw[circ] \pipe{180}{0};
    \end{tikzpicture}
    
    \vspace{2mm}
    
    \caption{
        A solid torus, carrying a non-trivial harmonic 1-form (blue).
        The constraint \eqref{eq:circulation_3d} equates the flux of $\bfomega$ through a cross-section of the pipe with the circulation of $\bfu$ around the boundary of that cross-section (red).
    }
    \label{fig:torus}
\end{figure}

In 2D, $\omega$ is similarly ill-defined up to constant functions in $\frakH^0 \subset \bbW$, present on all domains $\Omega$.
The solution is to fix the circulation (i.e.~the total vorticity)
\begin{equation}\label{eq:vorticity_mean}
    \int_\Omega \omega  =  \oint_{\partial\Omega} \bfu \cdot \bft,
\end{equation}
again through a Lagrange multiplier.
We apply this fix to the Dirichlet Kelvin--Helmholtz test of \Cref{sec:kh}.

\paragraph{The velocity $\bfu$ (2-forms in 3D, 1-forms in 2D).}
As in \Cref{sec:periodic_1}, harmonic 2-forms $\frakH^2$ in 3D and 1-forms $\frakH^1$ in 2D obstruct the existence of the streamfunction $\bfpsi$ used in the proof of \Cref{th:stability_bcs}, but in general do not affect the well-posedness of the discretisation \eqrefs{eq:semidiscrete_3d_bcs,eq:semidiscrete_2d_bcs}.
These harmonic forms can be associated with domain cavities in 3D, and general holes in 2D.

\paragraph{The pressure $p$ (3-forms in 3D, 2-forms in 2D).}
Outside of domains with zero boundary (i.e.~fully periodic domains), $b_n = \dim \frakH^n = 0$.
Nonetheless, in the absence of traction boundary conditions ($\Gamma_N = \emptyset$), the pressure is still ill-determined up to a constant\footnote{
    Specifically, a nullspace can be identified by incrementing $p \mapsto p + p_\rmH$, alongside $\bfsigma \mapsto \bfsigma - p_\rmH\bfn$ and $s \mapsto s - p_\rmH$, for a fixed constant $p_\rmH$.
    As ever, well-posedness requires the total imposed flux $\oint_{\partial\Omega}\bfu\cdot\bfn$ to be zero, consistent with incompressibility $\div\bfu = 0$.
}, as is true for any enclosed flow.
The same solutions as in \Cref{sec:periodic_1} present themselves:
either (i)~normalise $\int_\Omega p = 0$ through a Lagrange multiplier, or (ii)~report it to the solver.

\subsection{Instabilities on curved boundaries \& variational vorticity boundary conditions}\label{sec:strong_vorticity}

While the above variational approach for enforcing general boundary conditions \eqrefs{eq:semidiscrete_3d_bcs,eq:semidiscrete_2d_bcs} preserves the evolution of enstrophy \eqref{eq:enstrophy_discrete_bcs}, one issue presents itself around certain discrete approximations of curved boundaries.
To see this, consider the equivalence between the slip boundary condition formulations \eqref{eq:slip_1} and \eqref{eq:slip_2} presented in \Cref{sec:scheme}.
Following \citet[Sec.~2]{Costabel_Dauge_1999}, we calculate
\begin{subequations}\label{eq:compatibility}
\begin{align}
    \curl\bfu \times \bfn
    &= (\nabla\bfu^\top - \nabla\bfu)\cdot\bfn  \\
    &= (\nabla\bfu^\top + \nabla\bfu)\cdot\bfn
        + 2 \nabla\bfn\cdot\bfu
        - 2 (\nabla\bfu\cdot\bfn + \nabla\bfn\cdot\bfu)  \\
    &= 2\varepsilon\,\bfu \cdot \bfn
        + 2\sfS \cdot \bfu
        - 2\nabla[\bfu \cdot \bfn],
\end{align}
\end{subequations}
where $\sfS \coloneqq \nabla \bfn$ is the shape operator \citep{lee2018}, a symmetric matrix characterising the curvature of the boundary (see \Cref{rem:shape_operator} below).
Over a partial slip portion $\Gamma_S$ of the boundary $\partial\Omega$, the boundary condition \eqref{eq:partial} then implies that for any tangential $\bft$ on $\partial\Omega$,
\begin{equation}\label{eq:partial_compatibility}
    \bft \cdot (\curl\bfu \times \bfn)
        =  \bft \cdot (2\gamma(\bfu_{||,\mathrm{BC}} - \bfu_{||}) + 2 \sfS \cdot \bfu - 2 \nabla u_{\perp,\mathrm{BC}}).
\end{equation}
For free-slip ($\gamma = 0$), flat ($\sfS = 0$), zero-flux ($u_{\perp,\mathrm{BC}} = 0$) boundaries as considered in \Cref{sec:scheme}, these right-hand side terms each vanish, giving the equivalence between \eqref{eq:slip_1} and \eqref{eq:slip_2}.
At corners in the domain, however, the shape operator $\sfS$ becomes singular.
For \emph{convex} corners, Navier--Stokes solutions $\bfu$ with zero-flux, slip boundary conditions are parallel to the edge, implying the curvature term $2 \sfS \cdot \bfu$ vanishes at the corner, preserving the equivalence between \eqref{eq:slip_1} and \eqref{eq:slip_2}.
For \emph{concave} (re-entrant) corners, the velocity gradient $\nabla\bfu$ (and with it the vorticity $\curl\bfu$) at such corners blows up and the velocity $\bfu$ becomes ill-defined, rendering $2 \sfS \cdot \bfu$ ill-defined and breaking the equivalence.

In practice, to avoid the use of curved elements it is often reasonable to approximate curved obstacles and walls via piecewise linear approximations (see \Cref{fig:curved_approx}).
This is particularly true for the schemes presented here, as it is difficult to construct discrete Stokes complexes \eqrefs{eq:stokes_3d,eq:stokes_2d} that remain well-defined on curved elements.
However, there is reasonable basis in \eqref{eq:partial_compatibility} for expecting the curvature singularities caused by such approximations to lead to unphysical vorticities $\bfomega$.
Indeed, we observe this numerically, as the quality of the vorticity approximation under slip boundary conditions degrades around such linear approximations to curved boundaries, leading to instabilities in the discrete solution.

\begin{figure}[pos=!ht]
    \centering
    \begin{subfigure}{0.48\textwidth}
        \centering
        \begin{tikzpicture}[scale=1.5, baseline=(current bounding box.center), >=stealth,
                edge/.style  = {line width=1.1pt, black},
                lite/.style  = {line width=1.1pt, black!40},
                hatch/.style = {black!55, line width=0.5pt},
                nrm/.style   = {seabornblue, line width=0.9pt, -{Straight Barb[length=1.1mm, width=2.4mm]}}]
            \fill[black!15] (0,0) -- (180:1.9) arc[start angle=180, end angle=90, radius=1.9] -- cycle;
            \foreach \a in {97,102,...,177} {\draw[hatch] (\a:1.9) -- ++({\a+225}:0.18);}
            \foreach \a in {95,105,...,175} {\draw[nrm] (\a:1.9) -- ++({\a}:0.55);}
            \draw[edge] (180:1.9) arc[start angle=180, end angle=90, radius=1.9];
        \end{tikzpicture}
        \caption{Curved boundary}\label{fig:curved_exact}
    \end{subfigure}
    \begin{subfigure}{0.48\textwidth}
        \centering
        \begin{tikzpicture}[scale=1.5, baseline=(current bounding box.center), >=stealth,
                edge/.style   = {line width=1.1pt, black},
                lite/.style   = {line width=1.1pt, black!40},
                ghost/.style  = {line width=0.6pt, black!35, densely dashed},
                hatch/.style  = {black!55, line width=0.5pt},
                nrm/.style    = {seabornblue, line width=0.9pt, -{Straight Barb[length=1.1mm, width=2.4mm]}},
                corner/.style = {seabornred}]
            \fill[black!15] (0,0) -- (90:1.9) -- (120:1.9) -- (150:1.9) -- (180:1.9) -- cycle;
            \draw[ghost] (180:1.9) arc[start angle=180, end angle=90, radius=1.9];
            \foreach \k in {0,1,2} {
                \pgfmathsetmacro{\th}{105+30*\k}
                \foreach \s in {-0.31,-0.13,0.05,0.23,0.41} {
                    \pgfmathsetmacro{\px}{1.83526*cos(\th) + \s*cos(\th+90)}
                    \pgfmathsetmacro{\py}{1.83526*sin(\th) + \s*sin(\th+90)}
                    \draw[hatch] (\px,\py) -- ++({\th+225}:0.18);
                }
                \foreach \s in {-0.30,0,0.30} {
                    \pgfmathsetmacro{\qx}{1.83526*cos(\th) + \s*cos(\th+90)}
                    \pgfmathsetmacro{\qy}{1.83526*sin(\th) + \s*sin(\th+90)}
                    \draw[nrm] (\qx,\qy) -- ++({\th}:0.55);
                }
            }
            \draw[edge] (90:1.9) -- (120:1.9) -- (150:1.9) -- (180:1.9);
            \foreach \a in {120,150} {\fill[corner] (\a:1.9) circle (0.05);}
        \end{tikzpicture}
        \caption{Piecewise linear approximation}\label{fig:curved_approx_linear}
    \end{subfigure}
    \caption{
        A curved obstacle and its piecewise linear approximation, the fluid occupying the exterior, with outward normals $\bfn$ marked (blue).
        On the curved boundary $\bfn$ rotates continuously and the shape operator $\sfS = \nabla\bfn$ is bounded.
        On the approximation $\bfn$ is constant along each facet, so $\sfS$ vanishes there and concentrates instead at the vertices (red), each of which is a re-entrant corner of the fluid domain.
    }\label{fig:curved_approx}
\end{figure}
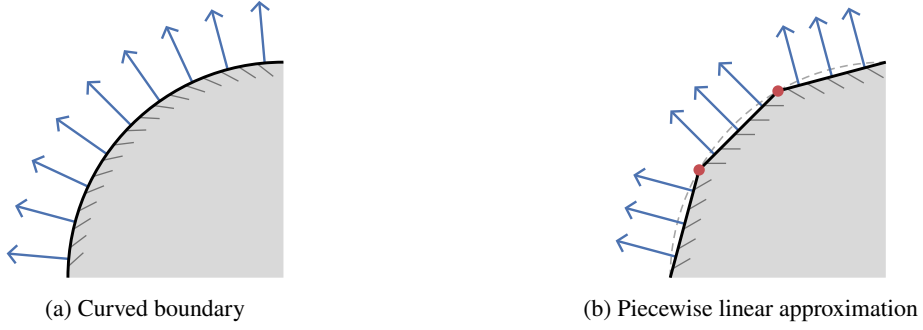

A natural solution is to enforce a boundary condition on the vorticity $\bfomega$ strongly.
The general compatibility condition \eqref{eq:partial_compatibility} represents a linear relation between $\bfomega$ and the velocity $\bfu$ that can be enforced variationally.
Letting $\bbW_S$ denote the tangential trace of $\bbW$ on $\Gamma_S$, we (i)~introduce a new variable $\bfzeta \in \bbW_S$ into \eqref{eq:semidiscrete_3d_bcs}, modifying the vorticity equation \eqref{eq:semidiscrete_3d_bcs_b} through the introduction of a Lagrange multiplier term
\begin{subequations}\label{eq:strong_vorticity_bc}
\begin{equation}\label{eq:strong_vorticity_bc_a}
    (\curl \bfomega, \curl \bfchi)  =  \Big[ \, 2 (\varepsilon \, \bfu, \varepsilon \curl \bfchi) - (\nabla \alpha, \bfchi) \, \Big] + \cdots + \langle \bfzeta, \bfchi \rangle_S,
\end{equation}
and (ii)~enforce the relation
\begin{equation}\label{eq:strong_vorticity_bc_b}
    \langle \bfomega, \bfeta \rangle_S
        =  2\gamma \langle \bfu_{||,\mathrm{BC}} - \bfu_{||}, \bfeta \times \bfn \rangle_S + 2 \langle \sfS \cdot \bfu, \bfeta \times \bfn \rangle_S - 2 \langle \nabla u_{\perp,\mathrm{BC}}, \bfeta \times \bfn \rangle_S
\end{equation}
\end{subequations}
to hold for all $\bfeta \in \bbW_S$, approximating \eqref{eq:partial_compatibility}.
The modification to \eqref{eq:semidiscrete_2d_bcs} in 2D is similar.
Since this represents a strong boundary condition for $\bfomega$ on $\Gamma_S$, we similarly impose the strong boundary condition $\alpha = 0$ on $\Gamma_S$.
This restriction on the complex will impact the relevant harmonic forms that might cause the vorticity subsystem \eqrefs{eq:semidiscrete_3d_bcs_a,eq:semidiscrete_3d_bcs_b} to be ill-defined, and the constraint \eqref{eq:vorticity_constraint_bcs} must be modified accordingly.
To ensure that \eqref{eq:strong_vorticity_bc} enforces physically reasonable boundary conditions on $\bfomega$, the curvature $\sfS$ used in \eqref{eq:strong_vorticity_bc_b} should be that of the desired \emph{continuous} boundary, and not that of the \emph{discrete} boundary mesh approximating it (which will typically be different as above). 

\begin{remark}[Re-entrant corners]
    The construction in \eqref{eq:strong_vorticity_bc} is ill-posed at \emph{concave} (re-entrant) corners.
    We note, however, that this is largely consistent with the continuous theory \citep{Grisvard_2011}:
    such domains imply singularities in the vorticity $\bfomega$, as the curvature term $2 \sfS \cdot \bfu$ in \eqref{eq:partial_compatibility} becomes ill-defined.
    It is unclear what discrete boundary conditions would be appropriate for a variable which, in the continuous setting, is singular on the boundary.
    \emph{Convex} corners in the continuous domain, on the other hand, require no special treatment and can simply be ignored in \eqref{eq:strong_vorticity_bc_b}.
\end{remark}

This modification weakens the enstrophy stability result \eqref{eq:enstrophy_discrete_bcs} of \Cref{th:stability_bcs} through the introduction of boundary terms linearly dependent on the Lagrange multiplier $\bfzeta$.
However, from \eqref{eq:strong_vorticity_bc_a} we see $\bfzeta$ represents a discrete approximation to $\bfzero$, fixing errors and instabilities on $\Gamma_S$:
if the vorticity approximation on $\Gamma_S$ from \eqrefs{eq:semidiscrete_3d_bcs,eq:semidiscrete_2d_bcs} remains reasonable, $\bfzeta$ should remain relatively small.
Indeed in the flat, free-slip, zero-flux case, we see the variationally imposed boundary conditions reduce exactly to those considered in \Cref{sec:scheme}, and $\bfzeta = \bfzero$ exactly.


\begin{remark}[Shape operator]\label{rem:shape_operator}
    Also referred to as the \emph{Weingarten map}, the shape operator $\sfS$ can be written in terms of the principal directions $(\bfnu_i)_{i=1}^{n-1}$ and principal curvatures $(\kappa_i)_{i=1}^{n-1}$ on $\partial\Omega$ as
    \begin{equation}
        \sfS = \sum_{i=1}^{n-1}\kappa_i\bfnu_i \otimes \bfnu_i.
    \end{equation}
    For instance, around a ball-shaped cavity of radius $r$, $\sfS \cdot \bfu$ is $- 1/r$ times the projection of $\bfu$ onto the surface of the ball.
    For further works utilising the compatibility condition \eqref{eq:compatibility} to handle boundary conditions for the Stokes and Navier--Stokes equations, see \citet{Mitrea_Monniaux_2009}, \citet{Costa_et_al_2023} and \citet{Boon_et_al_2026}.
\end{remark}

\section{Numerical simulations}\label{sec:simulations}

We conclude with various numerical experiments investigating the enstrophy stability properties of our scheme \eqrefs{eq:semidiscrete_3d,eq:semidiscrete_2d}, testing also the validity of the penalty method \eqrefs{eq:semidiscrete_3d_penalty,eq:semidiscrete_2d_penalty}, reparametrisation \eqrefs{eq:charlie_3d,eq:charlie_2d}, and MEEVC-inspired stabilisation \eqref{eq:semidiscrete_3d_delta} proposed in \Cref{sec:implementation}.
All simulations are performed in \texttt{Firedrake} \citep{Ham_et_al_2023};
the code used to produce them is available at \citet{Andrews_2026_GitHub}.

In \Cref{sec:kh}, we simulate a simple 2D shear flow.
In \Cref{sec:hill_vortex}, we consider a 3D spherical vortex to study the effects of vortex stretching and its discrete preservation.
In \Cref{sec:vortex_street}, we simulate, both in 2D and in 3D, the turbulent wake behind cylindrical and spherical obstacles.
\Cref{tab:experiments} summarises these experiments, alongside the boundary conditions considered, discrete complex used, and other implementation requirements and techniques applied in each case.

\begin{table}[pos=!ht]
    \centering
    \small
    \begin{tabular}{@{}lllll@{}}
        \toprule
        Test & Dim. & Boundary conditions & Discrete complex & Implementation notes \\
        \midrule
        \Cref{sec:kh}
            & 2
            & Periodic
            & Scott--Vogelius \eqref{eq:sv_2d}
            & $H^2$-free \eqref{eq:charlie_2d} \\
            (Shear flow)
            &
            & Free slip (strong) \eqref{eq:semidiscrete_2d}
            & 
            & Harmonic constraints \\
            &
            & Dirichlet (weak) \eqref{eq:semidiscrete_2d_bcs}
            &
            & $\qquad\qquad\qquad$(\Cref{sec:periodic_2}) \\
        \addlinespace
        \Cref{sec:hill_vortex}
            & 3
            & Periodic
            & Scott--Vogelius \eqref{eq:sv_3d}
            & $\bfH(\grad\curl)$-free \eqref{eq:charlie_3d} \\
            (Hill vortex)
            &
            &
            & 
            & Harmonic constraints \\
            &
            &
            & 
            & $\;\qquad$(\Cref{sec:periodic_1,sec:periodic_2}) \\
        \addlinespace
        \Cref{sec:vortex_street_2d}
            & 2
            & Periodic
            & Polynomial de Rham \eqref{eq:fedr_poly}
            & Penalty method \eqref{eq:semidiscrete_2d_penalty} \\
            (Cylindrical obstacle)
            &
            & Free slip (strong)
            & \\
            &
            & Partial slip (variational) \eqref{eq:strong_vorticity_bc}
            &
            & \\
        \addlinespace
        \Cref{sec:vortex_street_3d}
            & 3
            & Periodic
            & Hybrid de Rham \eqref{eq:fedr_hybrid}
            & Penalty method \eqref{eq:semidiscrete_3d_penalty} \\
            (Spherical obstacle)
            &
            & Free slip (strong)
            &
            & Reduced $\bar{\sigma}$ (\Cref{sec:penalty_sigma}) \\
            &
            & Partial slip (variational)
            &
            & $\delta$ stabilisation \eqref{eq:semidiscrete_3d_delta} \\
            &
            &
            &
            & Adaptive timestepping \\
        \bottomrule
    \end{tabular}
    \caption{
        Summary of the numerical experiments considered throughout \Cref{sec:simulations}.
    }
    \label{tab:experiments}
\end{table}

\begin{remark}[Validity of 3D results]\label{rem:3d_results}
    Establishing whether the results seen from our proposed integrator in the 3D tests below (\Cref{sec:hill_vortex,sec:vortex_street_3d}) represent accurate reproductions of the vortex stretching dynamics \eqref{eq:enstrophy_stretching_continuous} (rather than an artefact of, e.g., the spatial resolution) lies beyond what these experiments can settle.
    Doing so would require initial data and forcing whose behaviour is known independently (such as the recent blow-up scenario of \citet{OpenAI_2026b}) and a higher resolution than the coarse spatial discretisations considered.
    This we leave to future work.
\end{remark}

\subsection{Shear flow \& Kelvin--Helmholtz instability}\label{sec:kh}

We consider the Kelvin--Helmholtz instability on a unit square domain $\Omega = (0, 1)^2$ with periodic boundary conditions in $x$ only.
Up to divergence-free projection, the initial conditions model a smoothed shear layer with a small asymmetric perturbation,
\begin{equation}
    \bfu(0) = \begin{pmatrix}
        \tanh(10y - 5)  \\
        10^{-5}y\,(1-y)\,(\sin(2 \pi x) + \sin(4 \pi x) - \cos(4 \pi x))
    \end{pmatrix}.
\end{equation}
On the top ($y = 1$) and bottom ($y = 0$) walls we consider both free slip and Dirichlet boundary conditions in turn.
For the free slip discretisation, the boundary data are imposed strongly as in \Cref{sec:scheme} \eqref{eq:semidiscrete_2d};
for the Dirichlet discretisation, the boundary data are imposed variationally as in \Cref{sec:bcs} \eqref{eq:semidiscrete_2d_bcs}.
In this latter case, we note that the constant harmonic 0-form leaves the vorticity $\omega$ ill-defined, and requires fixing through a Lagrange multiplier \eqref{eq:vorticity_mean} enforcing
\begin{equation}
    \int_\Omega \omega = \oint_{\partial\Omega} \bfu \cdot \bft = - 2\tanh 5.
\end{equation}

Over a uniform triangular mesh of diameter $0.02$, we consider the Scott--Vogelius complex over an Alfeld-split mesh \eqref{eq:sv_2d}.
We use the $H^2$-free reparametrisation \eqref{eq:charlie_2d} presented in \Cref{sec:charlie} to access our scheme without requiring the implementation of $\bbH\bbC\bbT_3$\footnote{
    The Hsieh--Clough--Tocher element is well supported in {\tt Firedrake} \citep{Brubeck_Kirby_2026}.
    We use the $H^2$-free reparametrisation \eqref{eq:charlie_2d} primarily to validate its efficacy.
};
this employs the lower regularity polynomial de Rham complex \eqref{eq:fedr_poly} with $k = 2$ over the same Alfeld-split mesh, observing that $\bbU = [\bbC\bbG_2^\rmA]^2 \subset \bbB\bbD\bbM_2^\rmA$ as required, to evaluate $\omega$.
The reparametrisation \eqref{eq:charlie_2d} thus introduces the auxiliary variables $(\bfeta, R) \in \bbU_0 \times \bbP = [\bbC\bbG^\rmA_2]^2 \times \bbD\bbG^\rmA_1$, solving instead for $\omega \in \hat{\bbW}_0 = \bbC\bbG^\rmA_3$ such that $\grad^\perp \omega = \bfeta \in \bbU_0$, implicitly guaranteeing $\omega \in \bbW_0 = \bbH\bbC\bbT_3$.

We use a relatively large Reynolds number $\Re = 10^6$, and compare with a classical unstabilised discretisation where simply $\omega = - \div \bfu^\perp$\footnote{
    The MEEVC scheme is not considered for this test, as it is most naturally associated with the de Rham complex and not the Stokes complex used in the reparametrisation.
    The same is true of the vortex test case in \Cref{sec:hill_vortex}.
}.
We discretise both schemes in time using the implicit midpoint rule, over a timestep $\Delta t = 0.1$, running until a final time $t = 10$.

With variables for the velocity $\bfu$ and pressure $p$, the unstabilised scheme uses around 105k DoFs.
Due to the $H^2$-free reparametrisation \eqref{eq:charlie_2d}, our proposed scheme uses five fields in total, leading to a total of around 278k DoFs, approximately 164\% more.
Note, however, that the scheme would have required in total only around 120k DoFs, approximately 14\% more than the unstabilised scheme, had we used our proposed scheme without the reparametrisation (see \Cref{rem:charlie_overhead}).

\paragraph{Results.}

\Cref{fig:kh_2d_data} shows the evolution of enstrophy $\calE(\bfu)$ (\Cref{fig:kh_2d_enstrophy}) and energy $\calK(\bfu)$ (\Cref{fig:kh_2d_energy}) in the simulations.
\Cref{fig:kh_2d_enstrophy} uses a broken $y$-axis, the upper axis ranging over $[4, 500]$ and the lower over $[6.51, 6.69]$.
We see from the lower axis that the enstrophy in our discretisation \eqref{eq:charlie_2d} decreases monotonically across the duration in both settings, gradually at first, then more quickly after the onset of the instability around $t = 5$, to a final value around 2.2\% below its initial value under free slip and 1.7\% below under Dirichlet conditions.
The upper axis shows that the enstrophy in the unstabilised scheme \eqref{eq:unstabilised_semidiscrete} attains a minimum value at around $t = 2.6$ for free slip and $t = 1.9$ for Dirichlet, before climbing rapidly to final values around 3680\% and 5320\% above that of the initial state, respectively.
In \Cref{fig:kh_2d_energy}, we see the manifestation of this behaviour in the evolution of the energy, through the preserved energy dissipation law $\partial_t \calK = - \tfrac{2}{\Re}\calE$ \eqref{eq:energy_discrete}:
as the enstrophy balloons in the classical discretisation, the energy begins to decay more rapidly;
as the enstrophy remains bounded in the enstrophy-stable discretisation, the energy maintains a gradual rate of decay.

\begin{figure}[pos=!ht]
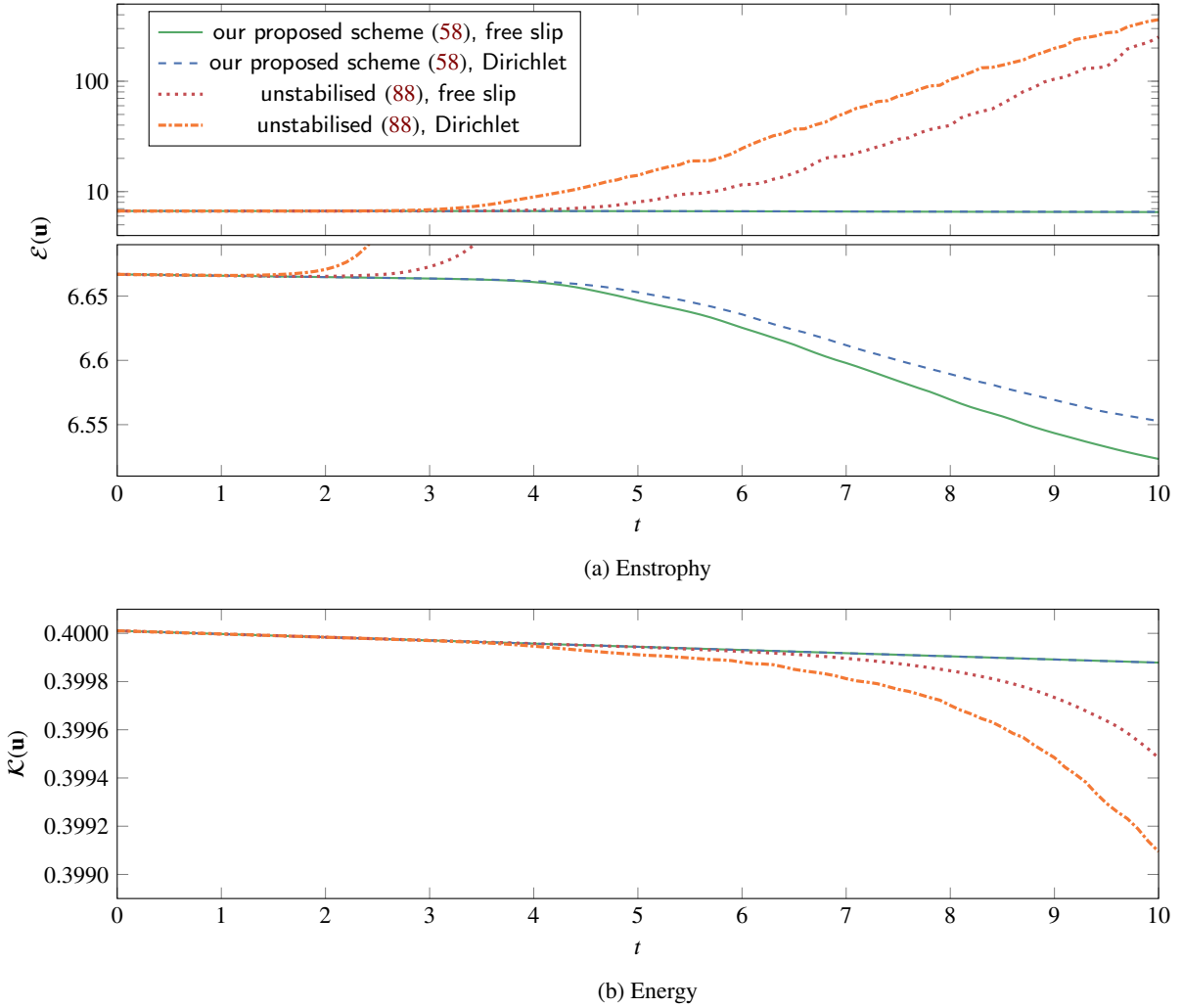

    \centering
    \pgfplotsset{
        common axis style/.style={
            xmin = 0, xmax = 10, xtick distance = 1,
            width = 0.9\textwidth, height = 0.2\textwidth,
            axis on top,
            scale only axis,
            extra y ticks = {0},
            extra y tick labels = {},
            extra y tick style = {grid = major,
                grid style = {black, dotted, line width = 0.5pt}},
        }
    }
    \begin{subfigure}{0.95\textwidth}
        \centering
        \begin{tikzpicture}[trim axis left]
            \begin{axis}[
                common axis style,
                xticklabels = {},
                ymode = log,
                ymin = 4, ymax = 500, ytick = {1, 10, 100, 1000},
                log ticks with fixed point,
                name = top_axis,
                legend pos = north west,
            ]
                \input{plots/kelvin_helmholtz/enstrophy/asf_slip.tex}
                \addlegendentry{our proposed scheme \eqref{eq:charlie_2d}, free slip}
                \input{plots/kelvin_helmholtz/enstrophy/asf_dirichlet.tex}
                \addlegendentry{our proposed scheme \eqref{eq:charlie_2d}, Dirichlet}
                \input{plots/kelvin_helmholtz/enstrophy/standard_slip.tex}
                \addlegendentry{unstabilised \eqref{eq:unstabilised_semidiscrete}, free slip}
                \input{plots/kelvin_helmholtz/enstrophy/standard_dirichlet.tex}
                \addlegendentry{unstabilised \eqref{eq:unstabilised_semidiscrete}, Dirichlet}
            \end{axis}

            \begin{axis}[
                common axis style,
                ymin = 6.51, ymax = 6.69, ytick distance = 0.05, restrict y to domain* = 6.4:6.8,
                at = {(top_axis.below south west)},
                anchor = north west,
                yshift = 1mm,
                xlabel = {$t$},
                ylabel = {$\calE(\bfu)$},
                ylabel style = {xshift = 1.65cm},
            ]
                \input{plots/kelvin_helmholtz/enstrophy/asf_slip.tex}
                \input{plots/kelvin_helmholtz/enstrophy/asf_dirichlet.tex}
                \input{plots/kelvin_helmholtz/enstrophy/standard_slip.tex}
                \input{plots/kelvin_helmholtz/enstrophy/standard_dirichlet.tex}
            \end{axis}
        \end{tikzpicture}
        \caption{Enstrophy}
        \label{fig:kh_2d_enstrophy}
    \end{subfigure}\\
    \vspace{3mm}
    \begin{subfigure}{0.95\textwidth}
        \centering
        \begin{tikzpicture}[trim axis left]
        \begin{axis}[
            common axis style,
            height = 0.25\textwidth,
            xlabel = {$t$},
            ymin = 0.3989, ymax = 0.4001, ytick distance = 0.0002,
            ylabel = {$\calK(\bfu)$},
            yticklabel style={/pgf/number format/fixed, /pgf/number format/precision=4, /pgf/number format/fixed zerofill},
        ]
            \input{plots/kelvin_helmholtz/energy/asf_slip.tex}
            \input{plots/kelvin_helmholtz/energy/asf_dirichlet.tex}
            \input{plots/kelvin_helmholtz/energy/standard_slip.tex}
            \input{plots/kelvin_helmholtz/energy/standard_dirichlet.tex}
        \end{axis}
        \end{tikzpicture}
        \caption{Energy}
        \label{fig:kh_2d_energy}
    \end{subfigure}\\

    \caption{Evolution of enstrophy and energy in both simulations of the Kelvin--Helmholtz instability.}
    \label{fig:kh_2d_data}
\end{figure}

\Cref{fig:kh_2d_plots} illustrates velocity profiles for both discretisations under the free slip conditions at the final time $t = 10$.
The numerical results from our proposed discretisation \eqref{eq:charlie_2d} demonstrate the expected Kelvin--Helmholtz vortices.
Large instabilities visible in the unstabilised discretisation \eqref{eq:unstabilised_semidiscrete} are associated with the large, uncontrolled enstrophy.

\begin{figure}[pos=!ht]
    \centering
    \begin{subfigure}{0.48\textwidth}
        \centering
        \includegraphics[width=0.8\textwidth]{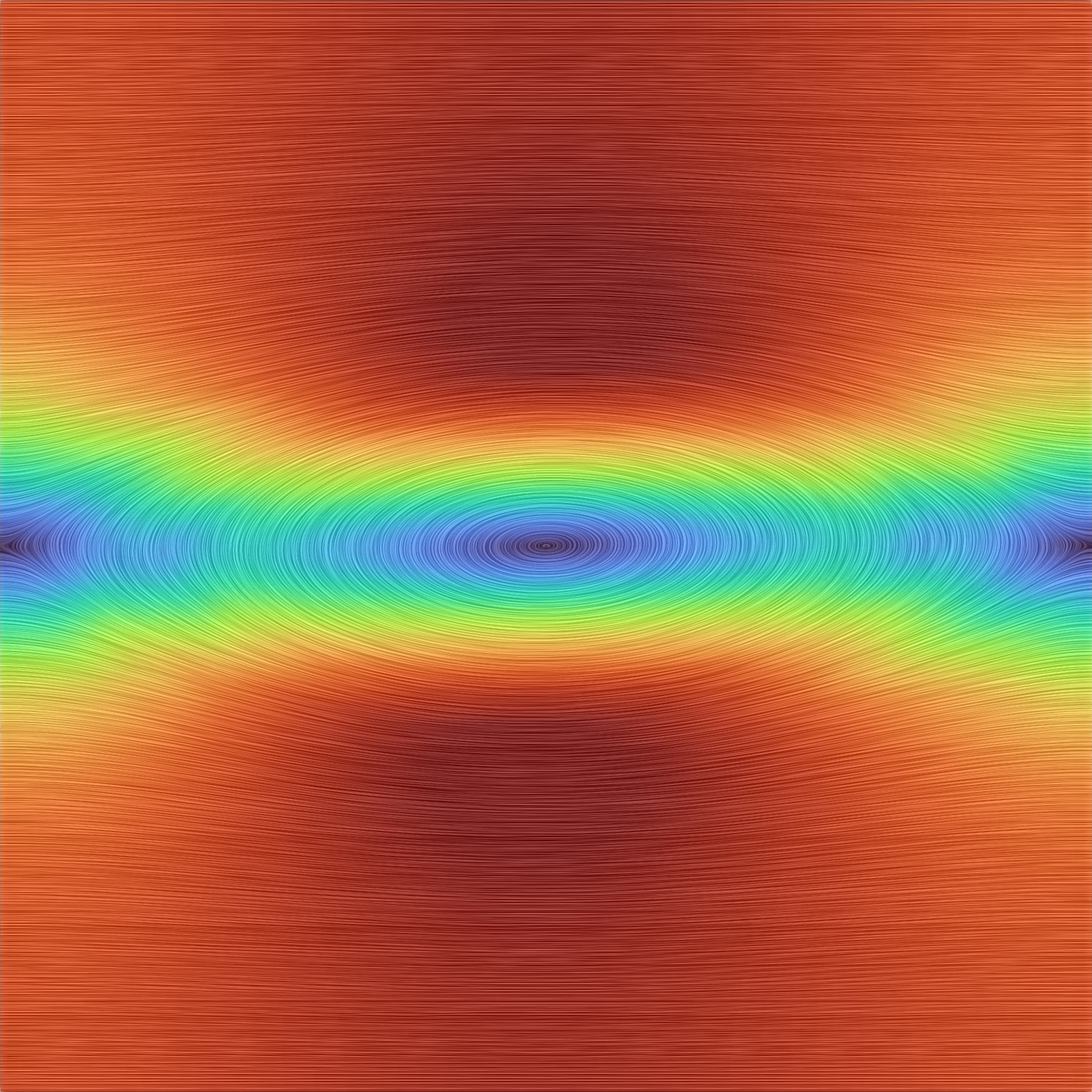}
        \caption{Our proposed scheme \eqref{eq:charlie_2d}}
    \end{subfigure}
    \begin{subfigure}{0.48\textwidth}
        \centering
        \includegraphics[width=0.8\textwidth]{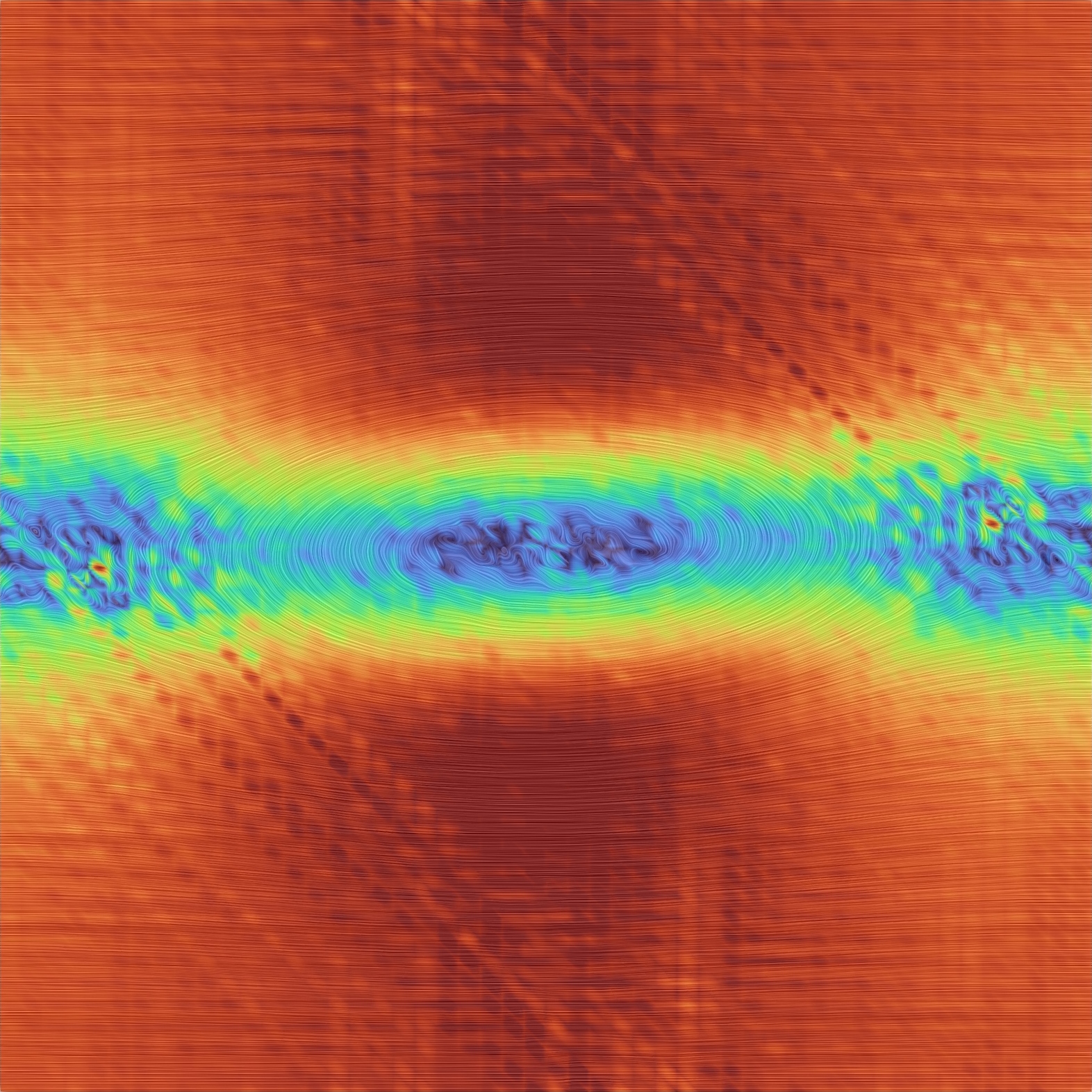}
        \caption{Unstabilised \eqref{eq:unstabilised_semidiscrete}}
    \end{subfigure}
    \caption{Final velocity profiles at time $t=10$ for both simulations of the Kelvin--Helmholtz instability, coloured from blue ($|\bfu| = 0$) to red ($|\bfu| = 1.2$) and textured to enhance streamlines.}
    \label{fig:kh_2d_plots}
\end{figure}

\subsection{Hill spherical vortex}\label{sec:hill_vortex}

We consider now a 3D flow with tightly wound vortex field lines, where we expect the advective contribution $\int_\Omega \bfu \cdot (\curl \bfu \times \curl^2 \bfu)$ to the enstrophy evolution \eqref{eq:enstrophy} to play a more significant role.
We take as our test problem the Hill spherical vortex \citep{Hill_1894,Moffatt_1969}, an exact steady-state solution to the incompressible Euler equations describing a vortex ring confined to a sphere, of radius $0.25$.
This may be defined in spherical coordinates $(r, \theta, \varphi)$ from the Stokes stream function
\begin{subequations}\label{eq:hill}
\begin{equation}
    \psi  =  \begin{cases}
        \left(\frac{J_{3/2}(4\eta r)}{(4r)^{3/2}} - J_{3/2}(\eta)\right)(r\sin\theta)^2,  &  r \le 0.25  \\
        0,  &  r \ge 0.25
    \end{cases},
\end{equation}
where $J_\alpha$ denotes the Bessel function of the first kind of order $\alpha$, and $\eta$ the first positive root of $J_{5/2}$, around $5.76$.
Up to projection, and normalisation to initial energy $\calK(\bfu(0)) = 1$, the initial conditions are given in terms of $\psi$ as
\begin{equation}
    \bfu(0)
        = \frac{1}{r^2 \sin\theta} \frac{\partial\psi}{\partial\theta} \bfe_r
        - \frac{1}{r \sin\theta} \frac{\partial\psi}{\partial r} \bfe_\theta
        + \frac{4 \eta}{r \sin\theta} \psi \bfe_\varphi.
\end{equation}
\end{subequations}
This classical solution features strong vortex field line curvature, and therefore provides an appropriate test for the behaviour of our scheme in a closed system where enstrophy is not simply dissipated.

We choose as our domain $\Omega = (-0.5, 0.5)^3$ with periodic boundary conditions.
As proposed in \Cref{sec:periodic_2}, we handle the uniform discrete harmonic 1-forms in $\frakH^1$ via Lagrange multipliers, and the harmonic 0- and 3-forms in $\frakH^0$ and $\frakH^3$ by providing them to the nonlinear solver.
We use a timestep of $0.001$, running to a final time $t = 0.5$.

The spatial discretisation is done via the $\bfH(\grad\curl)$-free reparametrisation \eqref{eq:charlie_3d}.
For the velocity $\bfu \in \bbU_0$ and the pressure $p \in \bbP$, we employ the Alfeld-split Scott--Vogelius pair \eqref{eq:sv_3d} at $k = 3$, recalling from \Cref{sec:charlie} that, for this reparametrisation, the preceding $\bfH^1$- and $\bfH(\grad\curl)$-conforming spaces in the discrete Stokes complex are not directly relevant to us.
Instead we evaluate $\alpha$ and the discrete vorticity $\bfomega$ through the lower regularity polynomial de Rham complex \eqref{eq:fedr_poly} at $k = 3$ over the same Alfeld-split mesh, observing that $\bbU = [\bbC\bbG_3^\rmA]^3 \subset \bbB\bbD\bbM_3^\rmA$ as required.
The reparametrisation introduces the auxiliary variables $(\bfeta, R) \in \bbU \times \bbP = [\bbC\bbG^\rmA_3]^3 \times \bbD\bbG^\rmA_2$, as well as $(\alpha, \bfomega) \in \hat{\bbA} \times \hat{\bbW} = \bbC\bbG^\rmA_5 \times \bbN\bbE\bbD^{(2),\rmA}_4$ such that (i)~$\curl\bfomega = \bfeta$ exactly, and (ii)~$\bfomega$ is $\hat{\bbA}$-discretely divergence-free.

Due to the high dimensionality, we consider a coarse uniform tetrahedral mesh of cell diameter $0.2$ (see \Cref{rem:hill_resolution}), albeit with a relatively large Reynolds number $\Re = 10^4$.
We compare with two alternative schemes:
(i)~the helicity-preserving scheme proposed by \citet{Rebholz_2007} \eqref{eq:rebholz_semidiscrete}, and (ii)~a classical unstabilised discretisation with $\bfomega = \curl\bfu$.

\begin{definition}[Helicity-stable semidiscretisation]
    Find $(\bfomega, \alpha, \bfu, P) \in \bbU \times \bbP \times \bbU \times \bbP$ such that
    \begin{subequations}\label{eq:rebholz_semidiscrete}
    \begin{align}
        (\bfomega, \bfchi)  &=  (\curl\bfu, \bfchi) + (\alpha, \div \bfchi),  \\
        0  &=  (\div\bfomega, \beta),  \\
        (\partial_t \bfu, \bfv)  &=  (\bfu \times \bfomega, \bfv) + (P, \div \bfv) - \frac{2}{\Re}(\varepsilon\,\bfu, \varepsilon\,\bfv),  \\
        0  &=  (\div \bfu, q),
    \end{align}
    \end{subequations}
    for all $(\bfchi, \beta, \bfv, q) \in \bbU \times \bbP \times \bbU \times \bbP$.
\end{definition}

In contrast to our proposed scheme \eqref{eq:semidiscrete_3d}, this semidiscretisation preserves the evolution of helicity $\calH(\bfu) \coloneqq \tfrac{1}{2}(\bfu, \curl\bfu)$, an important topological quantity in the Euler and Navier--Stokes equations \citep{Arnold_Khesin_2008}.
This helicity-preserving scheme \eqref{eq:rebholz_semidiscrete} has been observed to have improved stability under Hill vortex initial conditions \eqref{eq:hill}, when compared with a classical discretisation \citep[Sec.~2]{Andrews_Farrell_2025b} using Taylor--Hood elements.


With variables for the velocity $\bfu$ and pressure $p$ only, the unstabilised scheme requires around 74k DoFs.
With the additional auxiliary vorticity $\bfomega$ and Lagrange multiplier $\alpha$ in the same spaces, the helicity-stable scheme of \citet{Rebholz_2007} requires twice as many, around 148k.
Due to the $\bfH(\grad\curl)$-free reparametrisation \eqref{eq:charlie_3d}, our proposed scheme uses six fields in total, leading to a total of around 367k, approximately 396\% more than the unstabilised scheme.
Accessing the $\bfH(\grad\curl)$-conforming space without its implementation comes at a heavy cost (see \Cref{rem:charlie_overhead}).

\begin{remark}[Resolution for the Hill vortex discretisation]\label{rem:hill_resolution}
    By any reasonable measure, the mesh considered here is overly coarse for the given problem:
    the vortex has diameter $0.5$ and spans only two and a half cells.
    However, as stated above, with the $\bfH(\grad\curl)$-free reparametrisation \eqref{eq:charlie_3d}, our proposed semidiscretisation on the given discrete spaces already uses around 367k DoFs on the present mesh.
    The main obstacle to a higher resolution simulation is the accessibility of appropriate 3D discrete Stokes complexes\footnote{
        For example, 64k of the 367k degrees of freedom (i.e.~around 17.6\%) are spent resolving the variable $\hat{\alpha} \in \bbA = \bbC\bbG_5^\rmA$, which we know to be zero (see the proof of \Cref{th:stability});
        were $\hat{\alpha} \in \bbA = \bbC\bbG_1$ for example (as is the case in the complex of \citet{Hu_Zhang_Zhang_2022}) this would reduce to just 125 DoFs.
        The discrete variable $\bfomega \in \bbW_0$ simultaneously has a higher polynomial degree than $\bfu \in \bbU_0$, which is the opposite of what approximation theory asks for;
        the additional degrees of freedom do not improve the accuracy of $\bfomega$, which is limited by that of $\curl\bfu$.
    }.
    It is therefore unreasonable to expect \emph{any} of the discretisations considered below to reproduce the \emph{true} dynamics\footnote{
        At this resolution, it is natural to expect that the recovered vorticity $\bfomega$ in both our proposed scheme \eqref{eq:charlie_3d} and the helicity-stable scheme of \citet{Rebholz_2007} \eqref{eq:rebholz_semidiscrete} is not an accurate discrete representation of $\curl\bfu$.
    }, and we do not claim that they do so.
    We present this experiment as (i)~a verification that the full conforming 3D discretisation \eqref{eq:semidiscrete_3d} functions as defined, (ii)~a verification that the reparametrisation \eqref{eq:charlie_3d} is valid, and (iii)~a comparison between an L-stable (implicit Euler) and symmetric (implicit midpoint) integrator.
    The results should not necessarily be read as evidence that one scheme is more accurate than another.
\end{remark}

In the 2D shear flow problem above (\Cref{sec:kh}), any B-stable integrator applied to our 2D semidiscretisation \eqref{eq:semidiscrete_2d} would retain the $\bfH^1$ velocity bound from enstrophy dissipation \eqref{eq:enstrophy_discrete}.
This allowed us to use an implicit midpoint time discretisation, a symplectic method whose lack of L-stability would typically render it inappropriate for stiff parabolic systems such as the Navier--Stokes equations.
In the 3D setting, where vortex stretching (expected to make a large contribution with this initial data) can cause enstrophy to rise, this argument does not follow:
the preserved identity \eqref{eq:enstrophy_stretching} carries a production term, and bounds nothing.
We therefore consider both an implicit midpoint and implicit Euler discretisation, the latter being L-stable.

\paragraph{Results (implicit Euler).}

The implicit Euler simulations (\Cref{fig:hill_euler_data}) separate the three spatial discretisations cleanly:
the enstrophy of our proposed scheme \eqref{eq:charlie_3d} lies strictly below that of both alternatives throughout the simulation (\Cref{fig:hill_euler_enstrophy}).
Its peak sits $197\%$ above the initial value, against $488\%$ for the unstabilised scheme \eqref{eq:unstabilised_semidiscrete} and $836\%$ for the helicity-stable scheme \eqref{eq:rebholz_semidiscrete}.
The combination of the L-stable time discretisation and the enstrophy-stable spatial discretisation appears to offer favourable control over the enstrophy $\calE$, i.e.~control in $\bfH^1$ of the velocity $\bfu$.

\begin{figure}[pos=!ht]
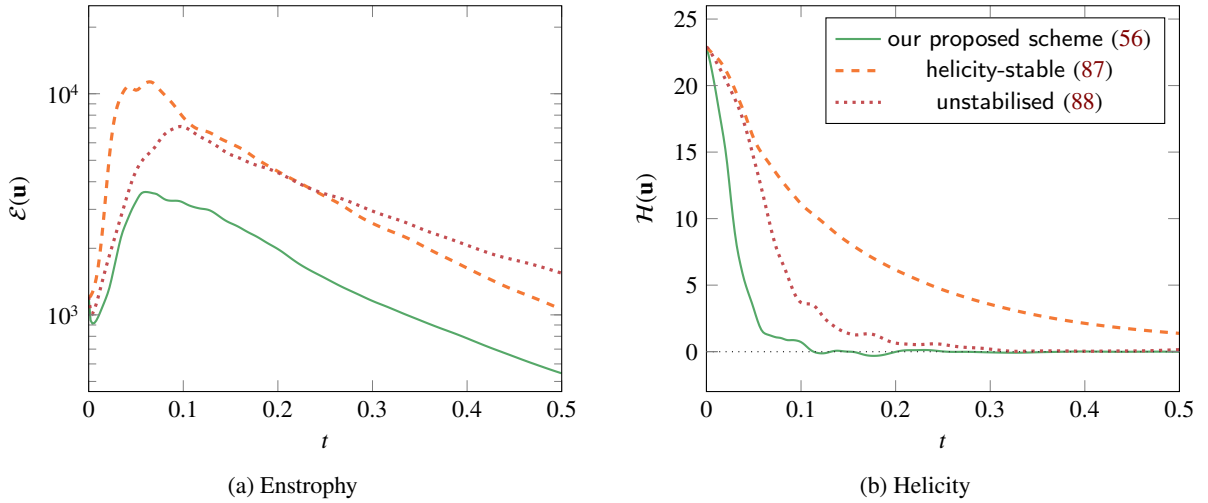

    \centering
    \pgfplotsset{
        common axis style/.style={
            extra y ticks = {0},
            extra y tick labels = {},
            extra y tick style = {grid = major,
                grid style = {black!80, dotted, line width = 0.5pt}},
            xmin = 0, xmax = 0.5, xtick distance = 0.1,
            axis on top,
            scale only axis,
            xlabel = {$t$},
        }
    }
    \begin{subfigure}{0.5\textwidth}
        \centering
        \begin{tikzpicture}
        \begin{axis}[
            common axis style,
            ymode = log,
            width = 0.76\linewidth, height = 0.62\linewidth,
            ymin = 450, ymax = 25000, ytick = {100, 1000, 10000, 100000},
            ylabel = {$\calE(\bfu)$},
        ]
            \input{plots/hill_vortex/euler/enstrophy/asf_re_1e4.tex}
            \input{plots/hill_vortex/euler/enstrophy/rebholz_re_1e4.tex}
            \input{plots/hill_vortex/euler/enstrophy/standard_re_1e4.tex}
        \end{axis}
        \end{tikzpicture}
        \caption{Enstrophy}
        \label{fig:hill_euler_enstrophy}
    \end{subfigure}%
    \begin{subfigure}{0.5\textwidth}
        \centering
        \begin{tikzpicture}
        \begin{axis}[
            common axis style,
            width = 0.76\linewidth, height = 0.62\linewidth,
            ymin = -3, ymax = 26, ytick distance = 5,
            ylabel = {$\calH(\bfu)$},
            legend pos = north east,
            legend style = {font = \small},
        ]
            \input{plots/hill_vortex/euler/helicity/asf_re_1e4.tex}
            \addlegendentry{our proposed scheme \eqref{eq:charlie_3d}}
            \input{plots/hill_vortex/euler/helicity/rebholz_re_1e4.tex}
            \addlegendentry{helicity-stable \eqref{eq:rebholz_semidiscrete}}
            \input{plots/hill_vortex/euler/helicity/standard_re_1e4.tex}
            \addlegendentry{unstabilised \eqref{eq:unstabilised_semidiscrete}}
        \end{axis}
        \end{tikzpicture}
        \caption{Helicity}
        \label{fig:hill_euler_helicity}
    \end{subfigure}\\

    \caption{
        Evolution of enstrophy and helicity in the Hill spherical vortex simulations using implicit Euler.
    }
    \label{fig:hill_euler_data}
\end{figure}

The energy cascade present in the Navier--Stokes equations transfers the distribution of energy to the high-frequency modes.
It is well known that L-stable methods applied to linear problems introduce an artificial dissipation acting preferentially on these high-frequency modes \citep[see e.g.][Sec.~IV.3]{Hairer_Wanner_1996};
it is possible that the enstrophy-stable spatial discretisation helps preserve part of this argument in the nonlinear setting.
Formalising this argument is the subject of ongoing work.

\paragraph{Results (implicit midpoint).}

This distinction is less present in the implicit midpoint simulations (\Cref{fig:hill_midpoint_data}).
Without the damping of an L-stable method, the enstrophy $\calE$ is controlled in none of the three discretisations.
The enstrophy of our proposed scheme \eqref{eq:charlie_3d} rises to a peak approximately 1271\% above its initial value, against 897\% for the unstabilised scheme \eqref{eq:unstabilised_semidiscrete} and 1357\% for the helicity-stable scheme \eqref{eq:rebholz_semidiscrete}, and all three remain well above their initial values at $t = 0.5$.

\begin{figure}[pos=!ht]
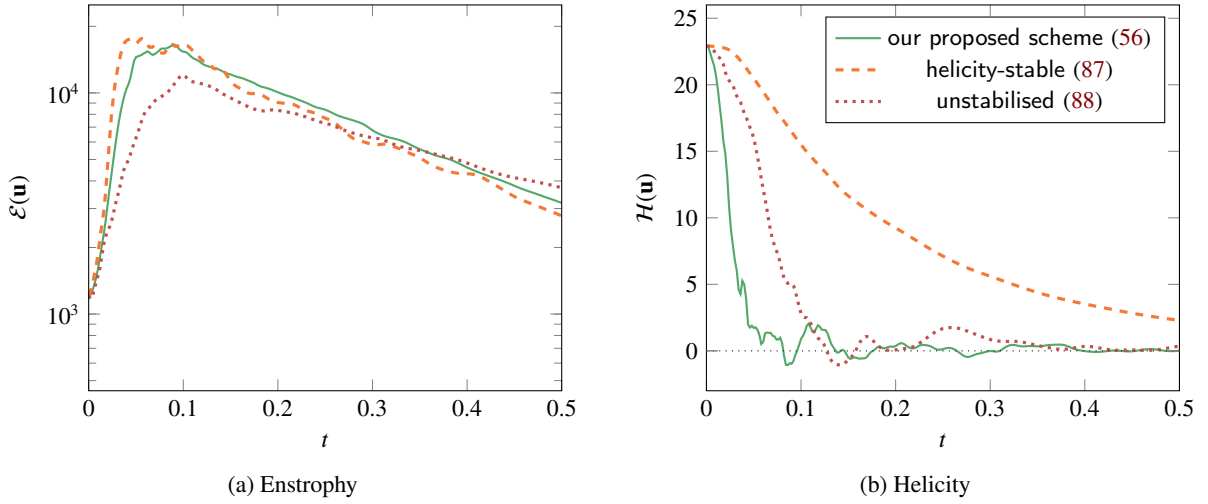

    \centering
    \pgfplotsset{
        common axis style/.style={
            xmin = 0, xmax = 0.5, xtick distance = 0.1,
            axis on top,
            scale only axis,
            xlabel = {$t$},
            extra y ticks = {0},
            extra y tick labels = {},
            extra y tick style = {grid = major,
                grid style = {black!80, dotted, line width = 0.5pt}},
        }
    }
    \begin{subfigure}{0.5\textwidth}
        \centering
        \begin{tikzpicture}
        \begin{axis}[
            common axis style,
            ymode = log,
            width = 0.76\linewidth, height = 0.62\linewidth,
            ymin = 450, ymax = 25000, ytick = {100, 1000, 10000, 100000},
            ylabel = {$\calE(\bfu)$},
        ]
            \input{plots/hill_vortex/midpoint/enstrophy/asf_re_1e4.tex}
            \input{plots/hill_vortex/midpoint/enstrophy/rebholz_re_1e4.tex}
            \input{plots/hill_vortex/midpoint/enstrophy/standard_re_1e4.tex}
        \end{axis}
        \end{tikzpicture}
        \caption{Enstrophy}
        \label{fig:hill_midpoint_enstrophy}
    \end{subfigure}%
    \begin{subfigure}{0.5\textwidth}
        \centering
        \begin{tikzpicture}
        \begin{axis}[
            common axis style,
            width = 0.76\linewidth, height = 0.62\linewidth,
            ymin = -3, ymax = 26, ytick distance = 5,
            ylabel = {$\calH(\bfu)$},
            legend pos = north east,
            legend style = {font = \small},
        ]
            \input{plots/hill_vortex/midpoint/helicity/asf_re_1e4.tex}
            \addlegendentry{our proposed scheme \eqref{eq:charlie_3d}}
            \input{plots/hill_vortex/midpoint/helicity/rebholz_re_1e4.tex}
            \addlegendentry{helicity-stable \eqref{eq:rebholz_semidiscrete}}
            \input{plots/hill_vortex/midpoint/helicity/standard_re_1e4.tex}
            \addlegendentry{unstabilised \eqref{eq:unstabilised_semidiscrete}}
        \end{axis}
        \end{tikzpicture}
        \caption{Helicity}
        \label{fig:hill_midpoint_helicity}
    \end{subfigure}\\

    \caption{
        Evolution of enstrophy and helicity in the Hill spherical vortex simulations using implicit midpoint.
    }
    \label{fig:hill_midpoint_data}
\end{figure}

To conclude, we note that, because both are consistent time discretisations, implicit Euler and implicit midpoint should theoretically converge to the same semidiscrete solution as $\Delta t \to 0$.
We conjecture that, for a sufficiently small timestep, the implicit midpoint results would resemble those of the implicit Euler time discretisation.

\subsection{Flow past a cylinder/sphere}\label{sec:vortex_street}

To conclude, we consider the turbulent wake behind a ball $B$ in both 2D and 3D.
We centre the ball at the origin and set its diameter to $1$, defining the domain to be $\Omega = (-1, 5) \times (-1, 1)^{n-1} \setminus B$, periodic in the $x$ direction only.
In either case we use an unstructured simplicial mesh.
Despite the curved obstacle, we do not use curved elements.
Due to the discrete approximation of the curved boundary, we cannot use a conforming complex, as the no-flux condition (to be specified) on the obstacle would inherently enforce $\bfu = \bfzero$ (i.e.~no slip) at the shallow re-entrant corners around $\partial B$.
We therefore use the penalty method outlined in \Cref{sec:penalty}, with penalty parameter $\sigma = 50$.

Using the same spaces for the velocity, pressure and vorticity, we compare with (i)~the MEEVC scheme \eqref{eq:meevc_semidiscrete}, and (ii)~a typical non-conforming discretisation without any stabilisation using an interior penalty method as in \Cref{sec:penalty} \citep{Douglas_Dupont_1976,DiPietro_Ern_2011} \eqref{eq:unstabilised_semidiscrete}, specified here in the simple free-slip setting of \Cref{sec:scheme}.

\begin{definition}[Classical unstabilised semidiscretisation (penalty method)]
    Find $(\bfu, p) \in \bbU_0 \times \bbP$ such that
    \begin{subequations}\label{eq:unstabilised_semidiscrete}
    \begin{align}
        (\partial_t \bfu, \bfv)
            &=  \!\left[\int_{\calT^h} (\bfu \times \curl \bfu) \cdot \bfv + \int_{\calF^h} (\leftavg\bfu\rightavg \times \leftjump\bfu \times \bfn\rightjump) \cdot \leftavg\bfv\rightavg\right]\! + \left(\frac{1}{2}|\bfu|^2 + p, \div \bfv\right) - \frac{2}{\Re}\calD^h\,[\bfu, \bfv],  \\
        0  &=  (\div \bfu, q),
    \end{align}
    \end{subequations}
    for all $(\bfv, q) \in \bbU_0 \times \bbP$.
\end{definition}

The walls ($y$ or $z = \pm 1$) and obstacle ($\bfx \in \partial B$) both use flux-free ($\bfu \cdot \bfn = 0$) slip boundary conditions.
The former uses free slip, while the latter uses partial slip with a friction parameter $\gamma_{\text{obstacle}}$.
The flow is driven by a steady body force $F \bfe_1$.
On both the walls and the obstacle we impose the boundary conditions on the velocity strongly.
In the two mixed discretisations \eqrefs{eq:semidiscrete_2d_penalty,eq:meevc_semidiscrete}, we impose the boundary conditions on the vorticity strongly on the walls, and variationally on the obstacle following \Cref{sec:strong_vorticity} \eqref{eq:partial_compatibility}\footnote{
    For further details on enforcing general boundary conditions in the MEEVC scheme, we refer the reader to \citet{deDiego_Palha_Gerritsma_2019,Zhang_et_al_2024}.
}.

Simulations are run to a final time $t = 50$.

\subsubsection{2D}\label{sec:vortex_street_2d}

In 2D, we take a Reynolds number $\Re = 10^4$, obstacle friction $\gamma_{\text{obstacle}} = 10^2$, and forcing $F = 5 \cdot 10^{-2}$, beginning from rest $\bfu(0) = \bfzero$.
We discretise in time again with an implicit midpoint method, over a timestep $\Delta t = 5 \cdot 10^{-2}$.
We use a mesh size of around $0.2$ in the far field, and around $0.05$ in the near field (within distance 0.5 of the obstacle; see \Cref{fig:vs_2d_mesh}).
For our discrete spaces, we consider the finite element de Rham complex of polynomials \eqref{eq:fedr_poly} at $k = 2$.
With variables for the velocity $\bfu$ and pressure $p$, the unstabilised scheme uses around 36k DoFs.
With the additional auxiliary vorticity $\omega$, both the MEEVC scheme and our own require around 52k, an increase of approximately 43\%.

\begin{figure}[pos=!ht]
    \centering
    \includegraphics[width=0.6\textwidth]{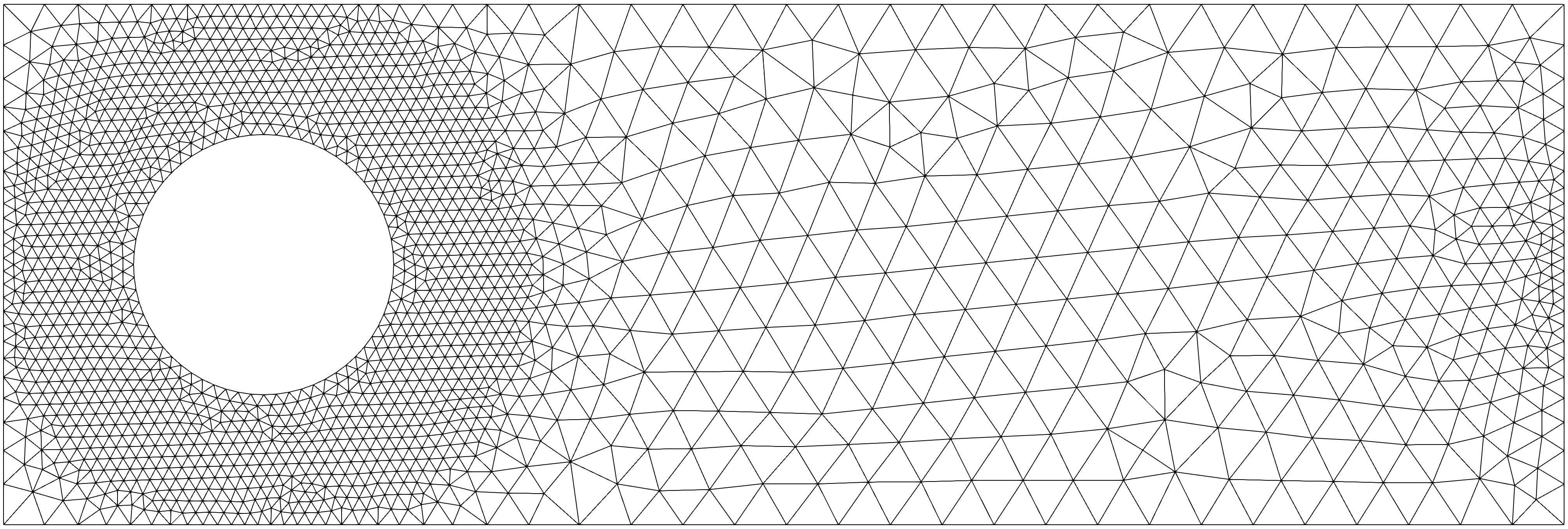}
    \caption{Triangular mesh for the 2D turbulent wake test problem.}
    \label{fig:vs_2d_mesh}
\end{figure}

\paragraph{Results.}

\Cref{fig:vs_2d_velocity} illustrates velocity profiles in each of the three simulations;
von K\'arm\'an vortex streets are visible in each case.
With the (broken) enstrophy $\calE^h(\bfu)$ not preserved in the MEEVC (\Cref{fig:vs_2d_velocity_meevc}) or unstabilised (\Cref{fig:vs_2d_velocity_unstabilised}) simulations, discontinuities in $\bfu$ are visible between cells, most prominently in the wake;
the numerical results from our scheme (\Cref{fig:vs_2d_velocity_asf}) appear visibly smoother by comparison.

\begin{figure}[pos=!ht]
    \centering
    \begin{subfigure}{0.98\textwidth}
        \centering
        \makebox[0.48\textwidth]{$t = 25$}
        \hspace{2mm}
        \makebox[0.48\textwidth]{$t = 50$}
        \\
        \vspace{2mm}
        \includegraphics[width=0.48\textwidth]{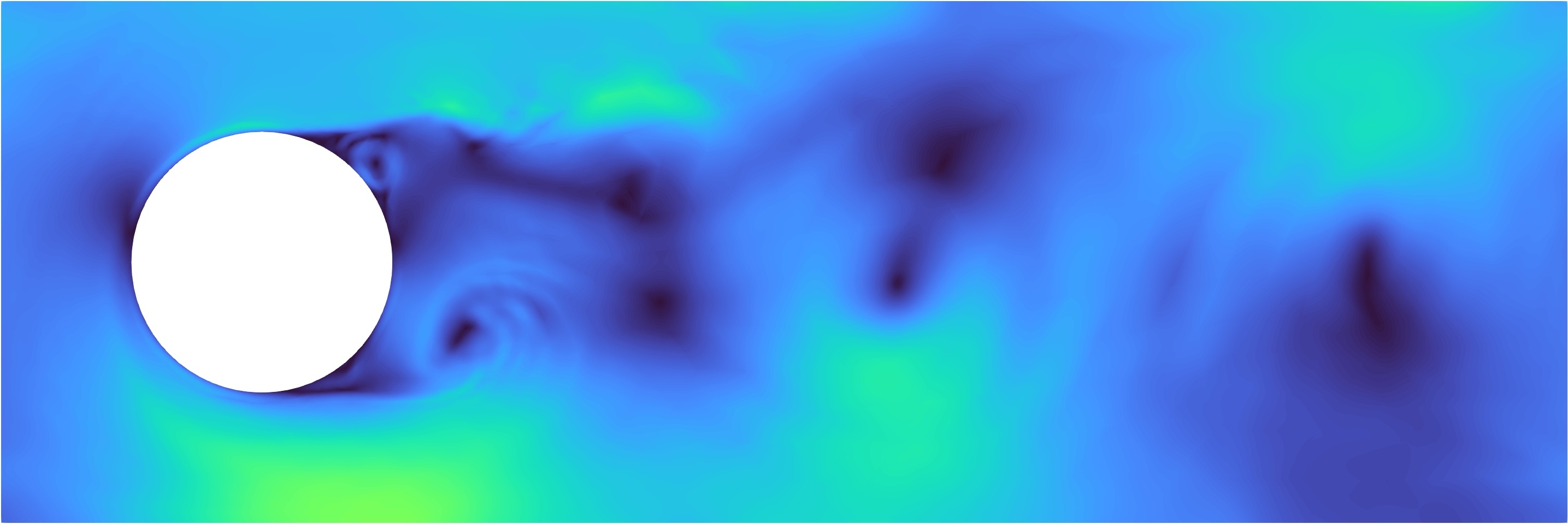}
        \hspace{2mm}
        \includegraphics[width=0.48\textwidth]{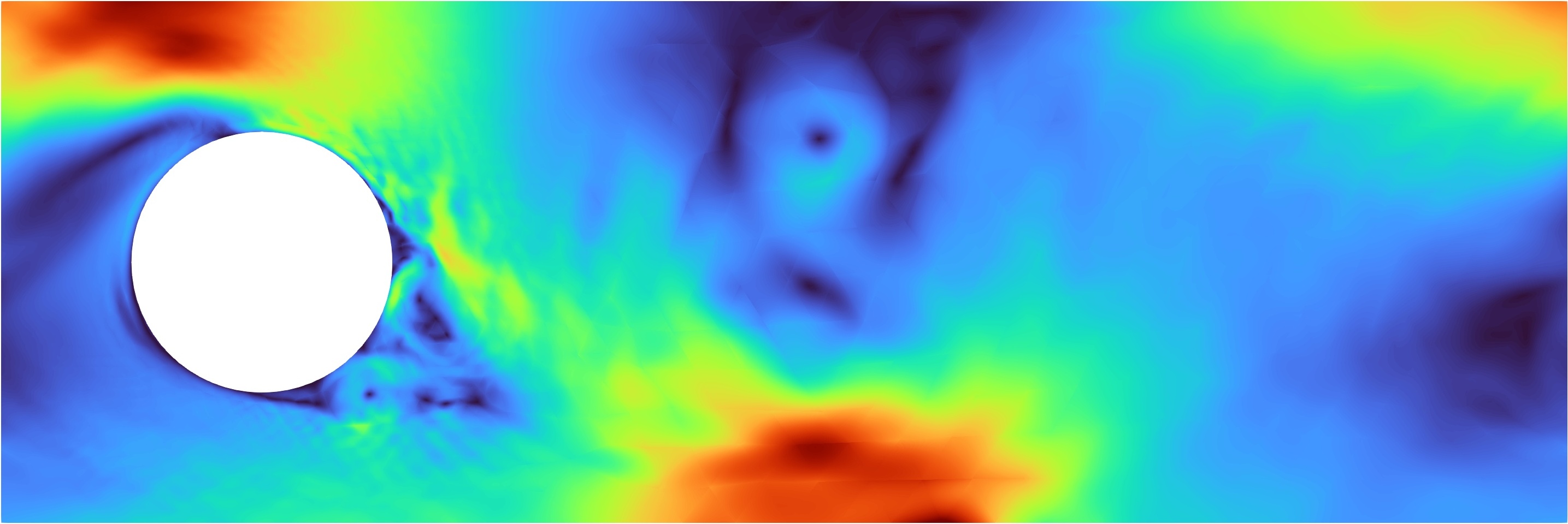}
        \caption{Our proposed scheme \eqref{eq:semidiscrete_2d_penalty}}
        \label{fig:vs_2d_velocity_asf}
    \end{subfigure}
    \\
    \vspace{2mm}
    \begin{subfigure}{0.98\textwidth}
        \centering
        \includegraphics[width=0.48\textwidth]{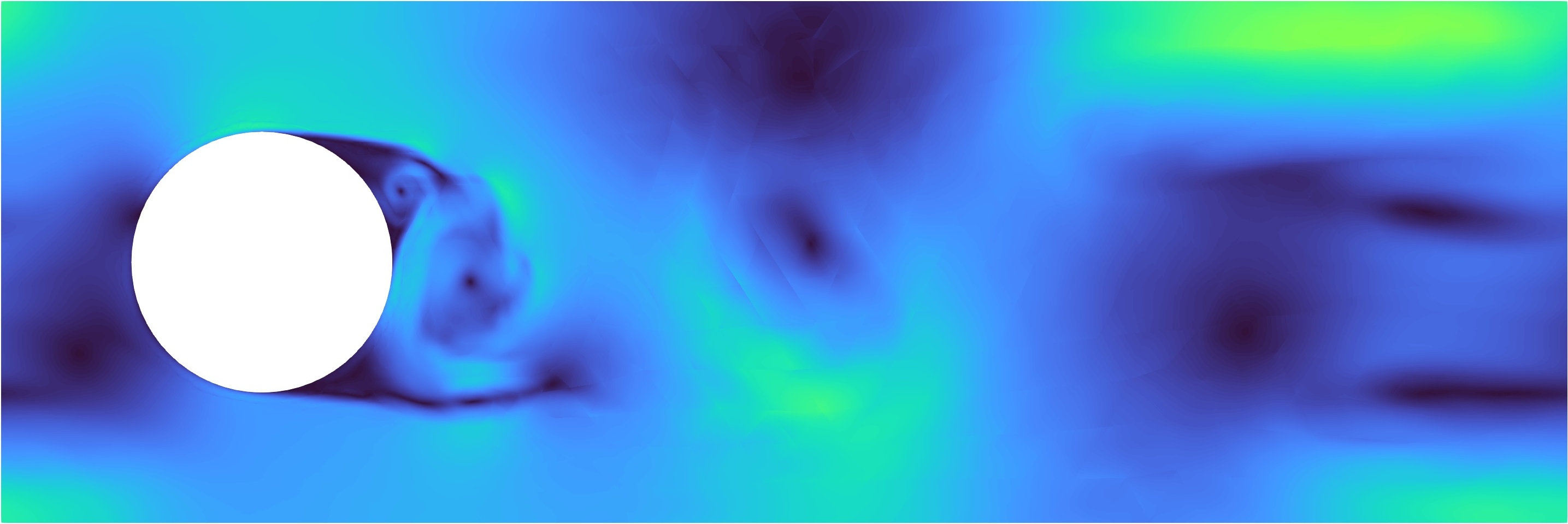}
        \hspace{2mm}
        \includegraphics[width=0.48\textwidth]{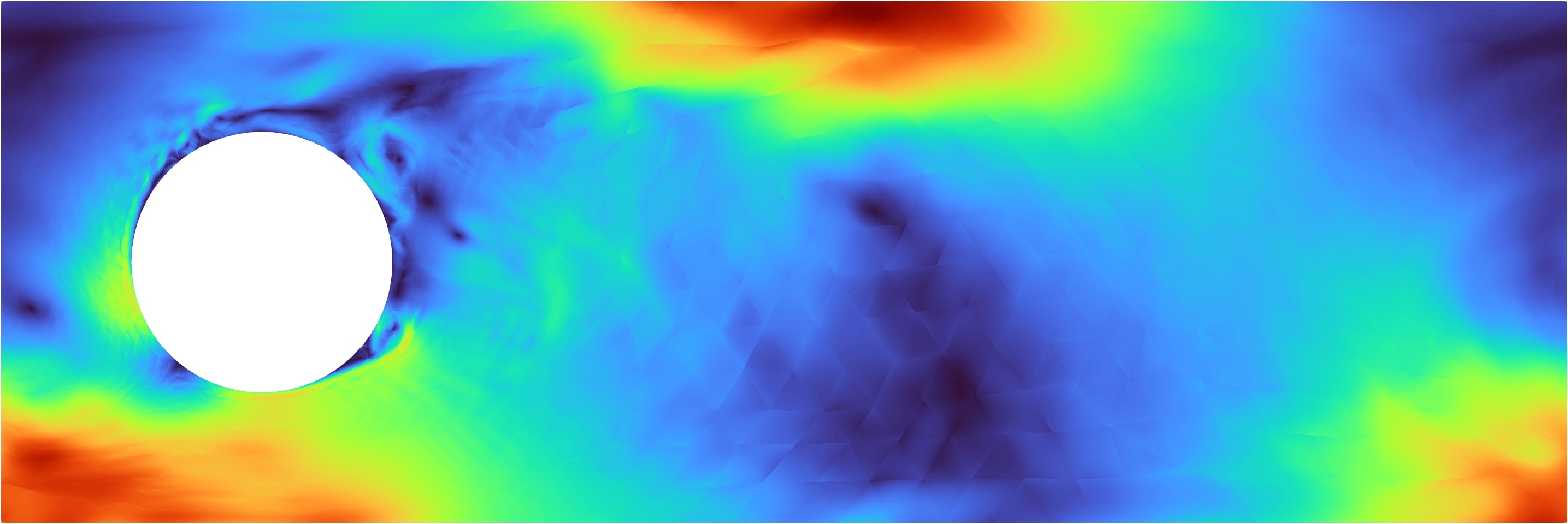}
        \caption{MEEVC \eqref{eq:meevc_semidiscrete}}
        \label{fig:vs_2d_velocity_meevc}
    \end{subfigure}
    \\
    \vspace{2mm}
    \begin{subfigure}{0.98\textwidth}
        \centering
        \includegraphics[width=0.48\textwidth]{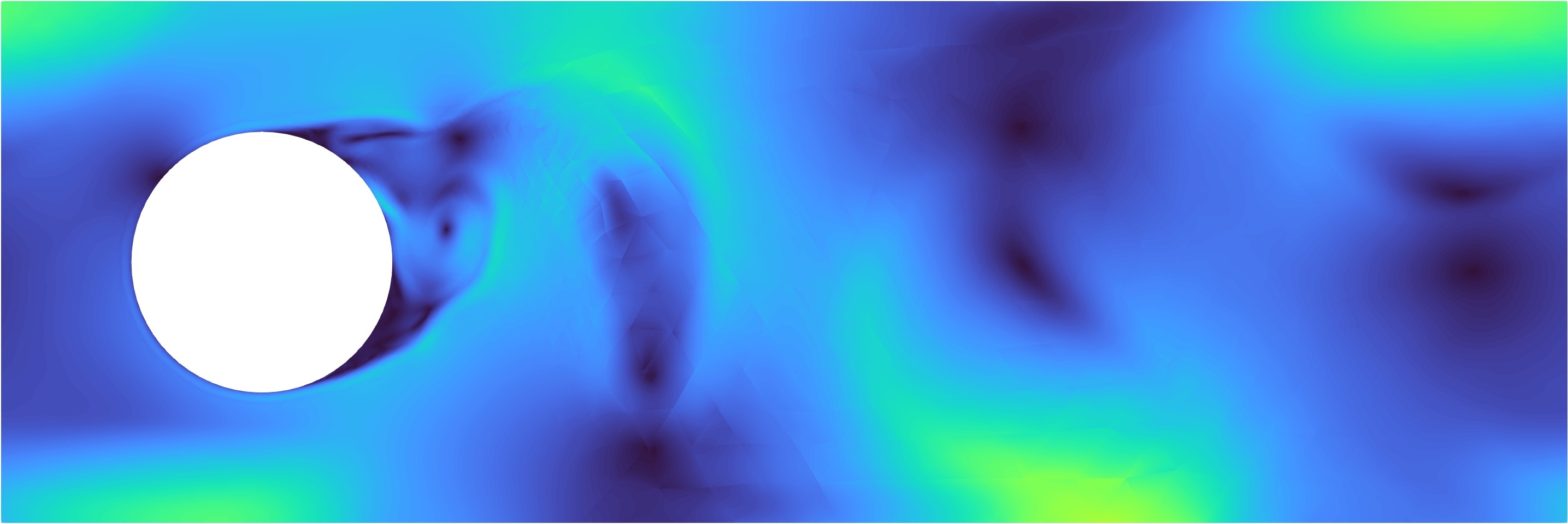}
        \hspace{2mm}
        \includegraphics[width=0.48\textwidth]{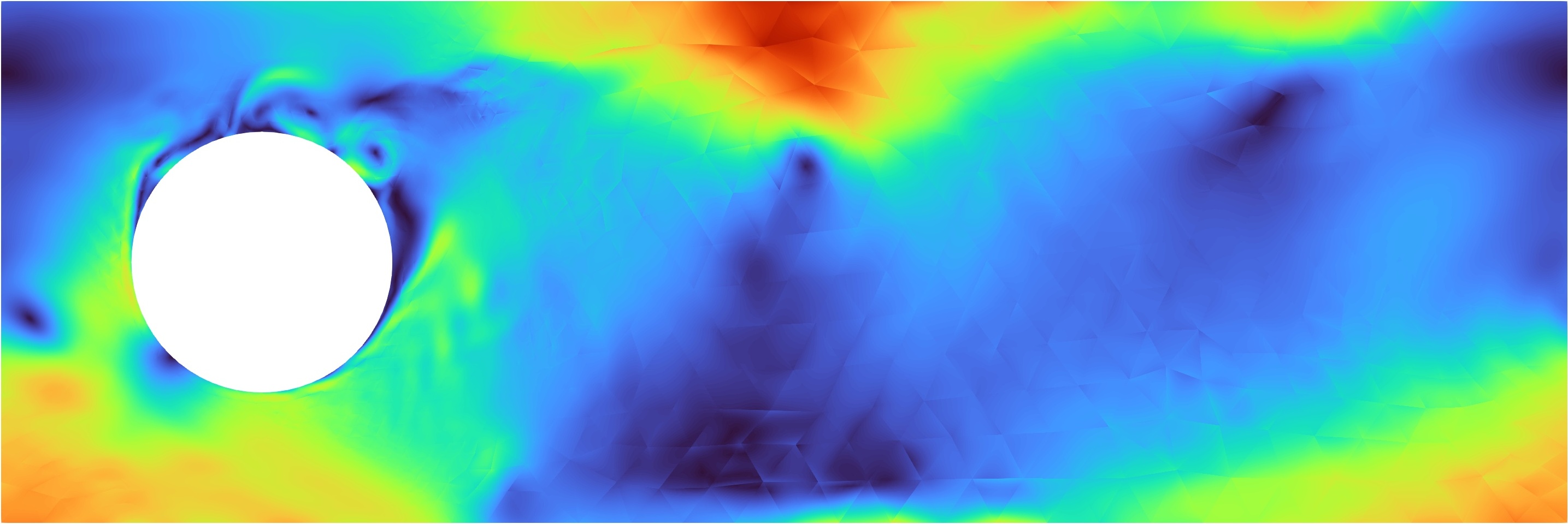}
        \caption{Unstabilised \eqref{eq:unstabilised_semidiscrete}}
        \label{fig:vs_2d_velocity_unstabilised}
    \end{subfigure}
    \caption{
        Velocity profiles for each of the three simulations of the 2D turbulent wake at times $t \in \{25, 50\}$, coloured from blue ($|\bfu| = 0$) to red ($|\bfu| = 5$).
    }
    \label{fig:vs_2d_velocity}
\end{figure}

\Cref{fig:vs_2d_data} shows the evolution of various quantities of interest across the three simulations, the first of which is the broken enstrophy $\calE^h(\bfu) \coloneqq \calD^h[\bfu, \bfu] + \gamma_{\text{obstacle}} \|\bfu\|_{\text{obstacle}}^2$, including contributions from the boundary friction (\Cref{fig:vs_2d_enstrophy})\footnote{
    Note that the enstrophy should not decrease monotonically in this problem because the flow is continually driven by the body force.
}.
Starting from rest, all three enter a turbulent regime by around $t = 25$.
We see from \Cref{fig:vs_2d_enstrophy} that the enstrophy $\calE^h(\bfu)$ in our discretisation \eqref{eq:semidiscrete_2d_penalty} remains largely the smallest of the three throughout, averaging around $350$ over the turbulent regime $t \in [25, 50]$;
compare with around $477$ for the MEEVC discretisation\footnote{
    It could be argued that a more appropriate enstrophy metric for the MEEVC scheme would be an auxiliary enstrophy evaluated on $\omega$, as it preserves the evolution of this auxiliary quantity.
    The penalty parameter $\sigma$, used in the definition of the broken enstrophy, is in fact absent from the specification of the MEEVC discretisation, as it is conforming on the lower regularity de Rham complex.
    The rise in the broken enstrophy $\calE^h(\bfu)$ is, however, meaningful in its quantification of the large discontinuities in the velocity field found with the MEEVC scheme, visible in \Cref{fig:vs_2d_velocity_meevc}.
} and $486$ without stabilisation, which are $36\%$ and $39\%$ higher, respectively.

\begin{figure}[pos=!ht]
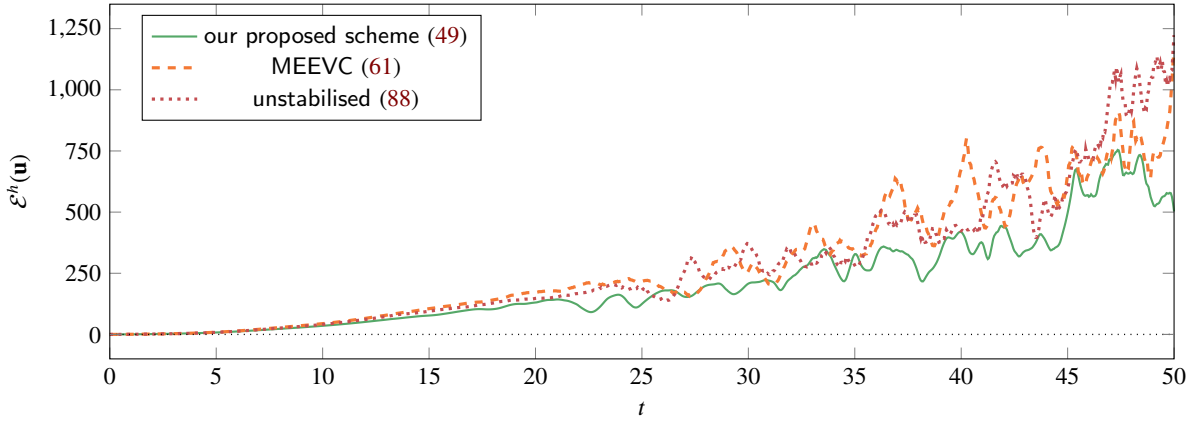
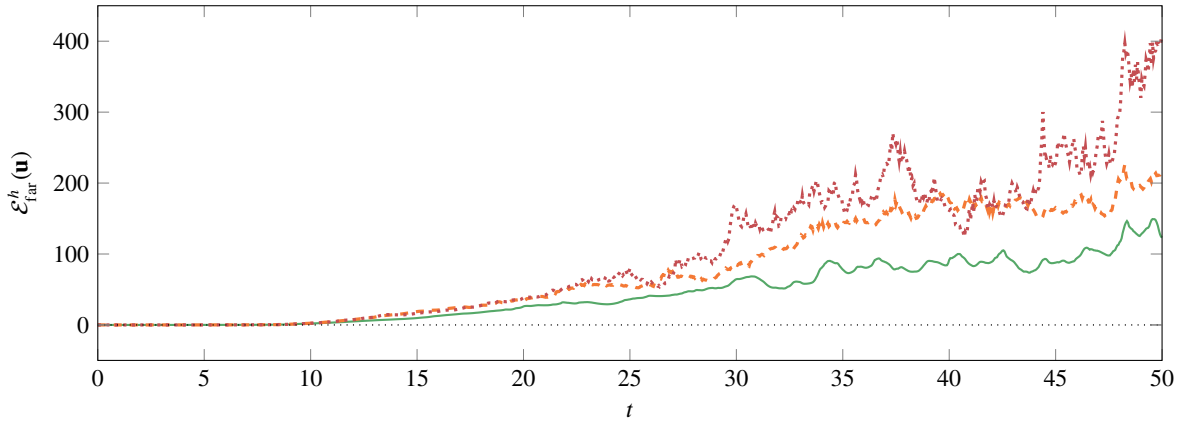
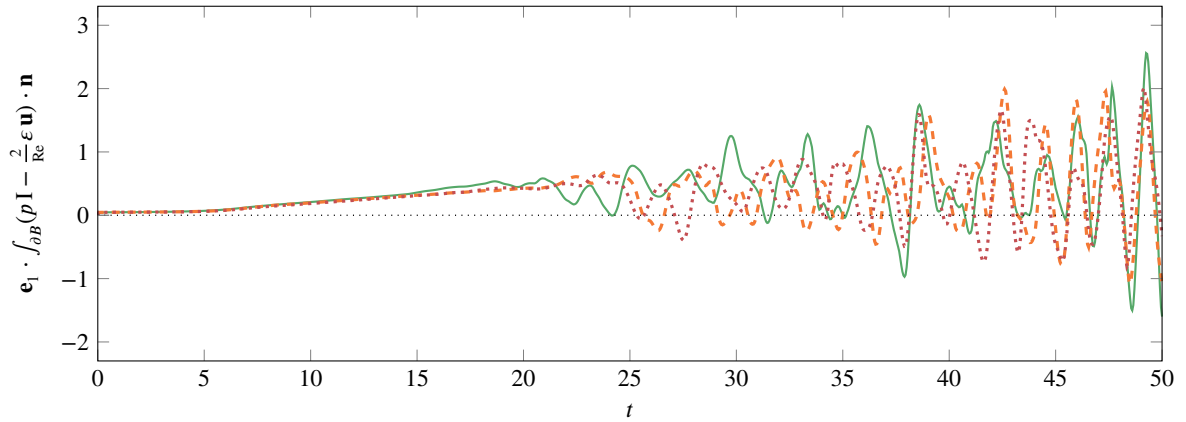

    \centering
    \pgfplotsset{
        common axis style/.style={
            extra y ticks = {0},
            extra y tick labels = {},
            extra y tick style = {grid = major,
                grid style = {black, dotted, line width = 0.5pt}},
            xmin = 0, xmax = 50, xtick distance = 5,
            width = 0.9\textwidth, height = 0.3\textwidth,
            axis on top,
            scale only axis,
            xlabel = {$t$},
        }
    }
    \begin{subfigure}{0.95\textwidth}
        \centering
        \begin{tikzpicture}[trim axis left]
        \begin{axis}[
            common axis style,
            ymin = -100, ymax = 1350, ytick distance = 250, ylabel = {$\calE^h(\bfu)$},
            legend pos = north west,
        ]
            \input{plots/vortex_street/2d/enstrophy/asf.tex}
            \input{plots/vortex_street/2d/enstrophy/meevc.tex}
            \input{plots/vortex_street/2d/enstrophy/standard.tex}
            \legend{our proposed scheme \eqref{eq:semidiscrete_2d_penalty}, MEEVC \eqref{eq:meevc_semidiscrete}, unstabilised \eqref{eq:unstabilised_semidiscrete}}
        \end{axis}
        \end{tikzpicture}
        \caption{(Broken) enstrophy}
        \label{fig:vs_2d_enstrophy}
    \end{subfigure}\\
    \vspace{3mm}
    \begin{subfigure}{0.95\textwidth}
        \centering
        \begin{tikzpicture}[trim axis left]
        \begin{axis}[
            common axis style,
            ymin = -50, ymax = 450, ytick distance = 100, ylabel = {$\calE^h_{\mathrm{far}}(\bfu)$},
        ]
            \input{plots/vortex_street/2d/far/asf.tex}
            \input{plots/vortex_street/2d/far/meevc.tex}
            \input{plots/vortex_street/2d/far/standard.tex}
        \end{axis}
        \end{tikzpicture}
        \caption{Far-field (broken) enstrophy}
        \label{fig:vs_2d_far}
    \end{subfigure}\\
    \vspace{3mm}
    \begin{subfigure}{0.95\textwidth}
        \centering
        \begin{tikzpicture}[trim axis left]
        \begin{axis}[
            common axis style,
            ymin = -2.3, ymax = 3.3, ytick distance = 1, ylabel = {$\bfe_1 \cdot \int_{\partial B}(p \, \rmI - \tfrac{2}{\rm Re}\varepsilon\,\bfu) \cdot \bfn$},
        ]
            \input{plots/vortex_street/2d/drag/asf.tex}
            \input{plots/vortex_street/2d/drag/meevc.tex}
            \input{plots/vortex_street/2d/drag/standard.tex}
        \end{axis}
        \end{tikzpicture}
        \caption{Drag}
        \label{fig:vs_2d_drag}
    \end{subfigure}\\

    \caption{
        Evolution of the broken enstrophy, its far-field restriction \eqref{eq:far_enstrophy}, and the drag, in each of the three simulations of the 2D turbulent wake.
    }
    \label{fig:vs_2d_data}
\end{figure}

The stabilisation controls the enstrophy, but only somewhat;
the contrast is certainly less stark than the Kelvin--Helmholtz setting of \Cref{sec:kh}.
However, the dominant fraction of $\calE^h(\bfu)$ is not in the bulk of the flow:
the contribution of the partial slip friction alone accounts for around $30\%$ of $\calE^h(\bfu)$ for our scheme, around $29\%$ without stabilisation, and around $21\%$ for MEEVC\footnote{
    Specifically, these values are the mean fractions $\gamma_{\text{obstacle}} \|\bfu\|_{\text{obstacle}}^2 / \calE^h(\bfu)$ over the interval $t \in [25,50]$.
}.
A further large fraction of the enstrophy is concentrated in the immediate vicinity of the obstacle.
To separate this from the behaviour of the discretisations in the bulk, we also report the enstrophy restricted to the far field,
\begin{equation}\label{eq:far_enstrophy}
    \calE^h_{\mathrm{far}}(\bfu)  \coloneqq  \calD^h\,[\bfu, \bfu] \big|_{\Omega_{\mathrm{far}}},
    \qquad
    \Omega_{\mathrm{far}}  \coloneqq  \{\bfx \in \Omega : |\bfx| \ge 1\},
\end{equation}
in which the cells (and interior facets) within a distance $0.5$ of the obstacle are excluded, along with the friction term on $\partial B$.
Restricting attention to the far field accordingly sharpens the comparison considerably (\Cref{fig:vs_2d_far}): over the same window $\calE^h_{\mathrm{far}}(\bfu)$ averages around $79$ across $t \in [25,50]$ for our scheme, against $138$ for MEEVC and $183$ without stabilisation, which are now $74\%$ and $131\%$ higher, respectively.
From $t = 5.45$, the far-field enstrophy in our proposed discretisation remains at all times below that of both the unstabilised and MEEVC schemes.
This aligns with the velocity profiles in \Cref{fig:vs_2d_velocity}:
away from the obstacle, where the curved boundary is no longer the limiting factor, the schemes are most clearly separated.

\Cref{fig:vs_2d_drag} reports the drag $\bfe_1 \cdot \int_{\partial B}(p \, \rmI - \tfrac{2}{\rm Re}\varepsilon\,\bfu) \cdot \bfn$ on the obstacle, which oscillates strongly once vortex shedding sets in.
Averaged across $t \in [25, 50]$ the three schemes give comparable values, around $0.51$ for our proposed scheme, $0.38$ for MEEVC and $0.41$ without stabilisation.

\subsubsection{3D}\label{sec:vortex_street_3d}

In 3D, we take a smaller Reynolds number $\Re = 5 \cdot 10^2$, lower obstacle friction $\gamma_{\text{obstacle}} = 2 \cdot 10^1$, and lower forcing $F = 2 \cdot 10^{-3}$.
We use a coarser mesh of size around $0.5$ in the far field and around $0.2$ near the obstacle.
For our discrete spaces, we consider the hybrid de Rham complex \eqref{eq:fedr_hybrid} at $k = 2$.

Rather than starting from rest, we take as the initial condition $\bfu(0)$ the steady \emph{Stokes} solution (i.e.~the stationary state in the absence of the advection term) at the same problem parameters.
We note that this gives an initial velocity flux in the $x$ direction of around $3.96$, close to the cross-sectional area of $4$ for the containing box.

For the time discretisation, we consider both a two-step backward differentiation formula (BDF2) and, as for the Hill vortex (\Cref{sec:hill_vortex}), the implicit midpoint method.
We use a Newton-based adaptive procedure for the timestep\footnote{
    The target of the adaptive timestepping is 3 Newton iterates per step.
    When the nonlinear solve requires 2 or fewer iterates, the timestep at the following step is increased by a factor of $1.4$ (below the $1 + \sqrt{2}$ zero-stability bound of \citet{Grigorieff_1983});
    for 4 or more it is decreased by a factor of $0.7$.
    After 6 Newton iterates, the nonlinear solve is restarted with half the timestep.
    The initial guess comes from a linear extrapolation of the previous two steps.
}, with an initial step of $\Delta t = 1 \cdot 10^{-1}$ and a maximum step of $\Delta t = 2 \cdot 10^{-1}$.

\begin{remark}[BDF2 and G-stability]\label{rem:bdf2}
    BDF2 is A-stable, L-stable and (crucially here, being a multistep method) G-stable \citep[][Sec.~V.6]{Dahlquist_1976,Hairer_Wanner_1996}.
    For a constant timestep\footnote{
        Note that G-stability is a property of BDF2 with a \emph{constant} timestep.
        For variable timesteps, the BDF2 coefficients depend on the step ratio $\Delta t_n / \Delta t_{n-1}$, and the G-stability argument no longer applies.
        With the adaptive timestepping used here, the modified energy and enstrophy stability of the BDF2 simulations \eqref{eq:bdf2_modified} is therefore not guaranteed.
        This does not affect the implicit midpoint method however, whose stability holds for any timestep.
    }, after time discretisation the energy and enstrophy stability results \eqrefs{eq:discrete_stability,eq:discrete_stability_bcs} hold not on the energy $\calK(\bfu^n)$ and enstrophy $\calE(\bfu^n)$ at a single time level $\bfu^n = \bfu(t^n)$, but on the respective multistep G-approximants
    \begin{equation}\label{eq:bdf2_modified}
        \frac{1}{2}\calK(\bfu^{n}) + \frac{1}{2}\calK(2\bfu^{n} - \bfu^{n-1}),
        \qquad
        \frac{1}{2}\calE(\bfu^{n}) + \frac{1}{2}\calE(2\bfu^{n} - \bfu^{n-1}),
    \end{equation}
    with $\bfu^{n-1} = \bfu(t^{n-1})$.
\end{remark}
The domain here is topologically non-trivial, exhibiting a harmonic 1-form that must be accounted for in the discretisation (see \Cref{sec:periodic_2}).
Up to a lifting of the tangential boundary conditions, the vorticity $\bfomega$ lies in $\bbW_0 \subset \bfH_0(\curl)$, for which the relevant harmonic space is $\frakH^1_0 \subset \bbW_0$ of dimension $b_2$;
the cavity induced by the ball contributes $b_2 = 1$\footnote{
    Note that in the 2D problem above (\Cref{sec:vortex_street_2d}) there are \emph{two} (independent) non-trivial harmonic 1-forms in $\frakH^1_0$:
    one from the cavity induced by the ball, and one from the $x$-periodicity.
    Recall, however, that in the 2D case $\frakH^1_0 \subset \bbU_0$, the space occupied by the velocity $\bfu \in \bbU_0$;
    it is the non-trivial harmonic 0-forms in 2D that present issues for the vorticity equation \eqref{eq:semidiscrete_2d_a}, of which this problem has none.
}.

Rather than removing this harmonic component with a Lagrange multiplier $\bflambda \in \frakH^1_0$ \eqref{eq:harmonic_constraint}, we use the MEEVC-inspired stabilisation \eqref{eq:semidiscrete_3d_delta} of \Cref{sec:meevc_stabilisation} with $\delta = 10^{-5}$;
this imposes both the discrete divergence-free and harmonic-free conditions on $\bfomega$ implicitly, and dispenses with both $\alpha \in \bbA_0$ and the multiplier $\bflambda \in \frakH^1_0$.
This leaves our scheme with the same variables, and hence the same number of degrees of freedom, as MEEVC.
With variables for the velocity $\bfu$ and pressure $p$, the unstabilised scheme has around 81k DoFs.
With the additional auxiliary vorticity $\bfomega$, both the MEEVC scheme and our own require around 151k, an increase of approximately 87\%.
Without the $\delta$-stabilisation terms \eqref{eq:semidiscrete_3d_delta}, the multiplier $\alpha$ (and the constraint fixing the harmonic component of $\bfomega$) would have added a further 18k.

Using the 3D penalty method \eqref{eq:semidiscrete_3d_penalty} directly, we observe that the vorticity diverges very rapidly, before the wake has the opportunity to develop.
This is in line with the unstable mode discussed in \Cref{sec:penalty_sigma}.
Following the argument therein, we reduce the penalty parameter in the vorticity reconstruction to some $\bar{\sigma} \in (0, \sigma)$, retaining $\sigma$ in the momentum equation.
\Cref{fig:vs_3d_sigma} considers the effects of different choices for $\bar{\sigma}$ up to a final time of only $t = 10$;
specifically, we consider $\bar{\sigma}/\sigma \in \{0, 0.03, 0.1, 0.3, 0.5, 1\}$.
\Cref{fig:vs_3d_sigma_vorticity} plots the vorticity norm $\|\bfomega\|^2$ for each value of $\bar{\sigma}$:
the divergence at $\bar{\sigma} = \sigma$ is clear, with the remaining choices behaving more stably.
In the continuous setting, integration by parts along with the boundary condition relation \eqref{eq:compatibility} of \citet{Costabel_Dauge_1999} gives
\begin{equation}
    \|\curl\bfu\|^2
        =  2 \|\varepsilon\,\bfu\|^2 - 4 \|\bfu\|_{\partial B}^2.
\end{equation}
\Cref{fig:vs_3d_sigma_ratio} thus non-dimensionalises this metric by evaluating the ratio of the norm $\|\bfomega\|^2$ to $2\calD^h[\bfu, \bfu] - 4 \|\bfu\|_{\partial B}^2$:
the two remain most uniformly in balance within the considered set at $\bar{\sigma} = 0.3\sigma$, motivating us to make this choice for the full simulation.
Whether this scaling remains appropriate in different problem settings remains an open question.

\begin{figure}[pos=!ht]
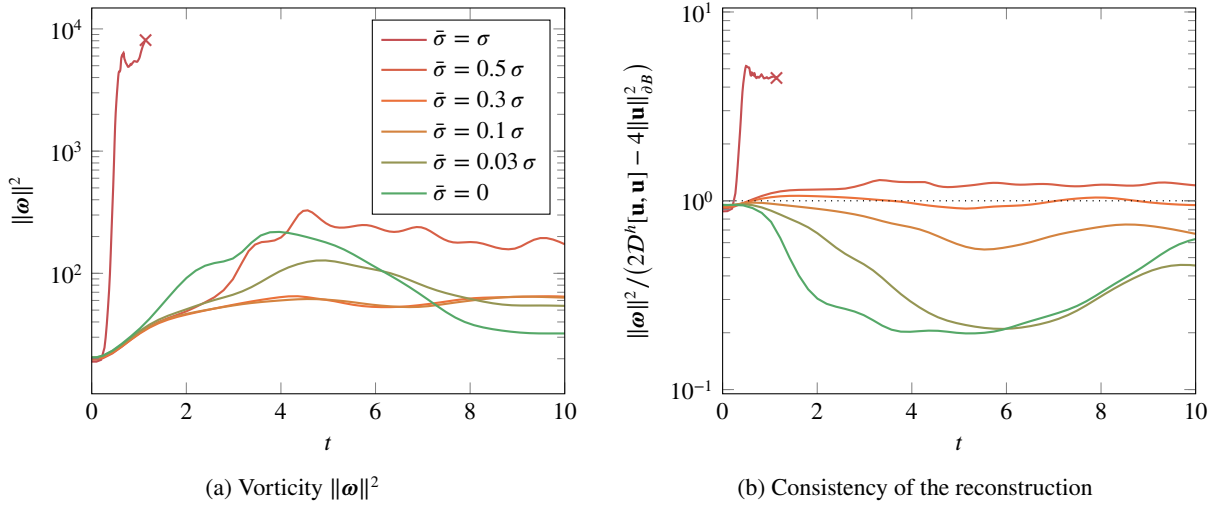

    \centering
    \pgfplotsset{
        common axis style/.style={
            xmin = 0, xmax = 10, xtick distance = 2,
            width = 0.76\linewidth, height = 0.62\linewidth,
            axis on top,
            scale only axis,
            xlabel = {$t$},
            extra y ticks = {0},
            extra y tick labels = {},
            extra y tick style = {grid = major,
                grid style = {black, dotted, line width = 0.5pt}},
        }
    }
    \begin{subfigure}{0.5\textwidth}
        \centering
        \begin{tikzpicture}
        \begin{axis}[
            common axis style,
            ymode = log,
            extra y ticks = {},
            ylabel = {$\|\bfomega\|^2$},
            legend pos = north east,
            legend cell align = left,
            every axis plot post/.append style = {thick},
        ]
            \input{plots/vortex_street/3d/vorticity/sig100.tex}
            \input{plots/vortex_street/3d/vorticity/sig050.tex}
            \input{plots/vortex_street/3d/vorticity/sig030.tex}
            \input{plots/vortex_street/3d/vorticity/sig010.tex}
            \input{plots/vortex_street/3d/vorticity/sig003.tex}
            \input{plots/vortex_street/3d/vorticity/sig000.tex}
        \end{axis}
        \end{tikzpicture}
        \caption{Vorticity $\|\bfomega\|^2$}
        \label{fig:vs_3d_sigma_vorticity}
    \end{subfigure}%
    \begin{subfigure}{0.5\textwidth}
        \centering
        \begin{tikzpicture}
        \begin{axis}[
            common axis style,
            ymode = log,
            extra y ticks = {1},
            ymin = 0.095, ymax = 10.5, ytick = {0.01, 0.1, 1, 10, 100},
            ylabel = {$\|\bfomega\|^2 / \big(2\calD^h[\bfu, \bfu] - 4\|\bfu\|_{\partial B}^2\big)$},
            every axis plot post/.append style = {thick},
        ]
            \input{plots/vortex_street/3d/ratio/sig100.tex}
            \input{plots/vortex_street/3d/ratio/sig050.tex}
            \input{plots/vortex_street/3d/ratio/sig030.tex}
            \input{plots/vortex_street/3d/ratio/sig010.tex}
            \input{plots/vortex_street/3d/ratio/sig003.tex}
            \input{plots/vortex_street/3d/ratio/sig000.tex}
        \end{axis}
        \end{tikzpicture}
        \caption{Consistency of the reconstruction}
        \label{fig:vs_3d_sigma_ratio}
    \end{subfigure}

    \caption{
        Effect of the penalty parameter $\bar{\sigma}$ used in the vorticity reconstruction \eqrefs{eq:semidiscrete_3d_penalty_a,eq:semidiscrete_3d_penalty_b} with all other parameters held fixed.
    }
    \label{fig:vs_3d_sigma}
\end{figure}

To quantify the error introduced by the $\delta$ stabilisation \eqref{eq:semidiscrete_3d_delta}, we repeat the above simulation at $\bar{\sigma} = 0.3\sigma$ with $\delta = 0$, i.e.~restoring the multiplier $\alpha$ and the Lagrange multiplier $\bflambda$ fixing the harmonic component of $\bfomega$ \eqref{eq:harmonic_constraint}, and holding all other parameters fixed.
The two agree to four significant figures in every quantity we record:
the velocity flux, the broken enstrophy and its far-field restriction, the energy, and $\|\bfomega\|^2$.
While the stabilisation makes little difference in accuracy (at least to the precision these diagnostics resolve), the nonlinear solves are completed in approximately $0.4$ times the wall-clock time on the same machine\footnote{
    The saving is larger than the $10\%$ reduction in degrees of freedom would suggest, since the real block carrying the multiplier $\bflambda$ cannot be held in a monolithic sparse matrix, and forces a matrix-free Schur complement in place of a single factorisation.
}.

\paragraph{Results (BDF2).}

\Cref{fig:vs_3d_bdf2_data} shows the evolution of the far-field (broken) enstrophy $\calE^h_{\mathrm{far}}$ (\Cref{fig:vs_3d_bdf2_far}) and energy $\calK$ (\Cref{fig:vs_3d_bdf2_energy}) in our proposed scheme \eqref{eq:semidiscrete_3d_penalty} with $\bar{\sigma} = 0.3\sigma$, the unstabilised scheme \eqref{eq:unstabilised_semidiscrete}, and the MEEVC scheme \eqref{eq:meevc_semidiscrete}, using BDF2.
Both our proposed discretisation and the unstabilised scheme reach the final time $t = 50$, while our nonlinear solver fails for the MEEVC discretisation at $t = 5.20$;
at this point its far-field enstrophy $\calE^h_{\mathrm{far}}$ has reached $783$, more than forty times that of either other scheme.
Our scheme ends at around $\calE^h_{\mathrm{far}} = 8.1$, against $13.9$ for the unstabilised discretisation.
The far-field enstrophy from our scheme is no longer uniformly below that of the unstabilised scheme as was the case in the 2D setting (\Cref{sec:vortex_street_2d}), although it lies below for around $73\%$ of the simulation, and throughout from $t \approx 28$ onwards.
\Cref{fig:vs_3d_bdf2_energy} shows roughly comparable rates of decay for the energy in our proposed discretisation when compared with the unstabilised one.

\begin{figure}[pos=!ht]
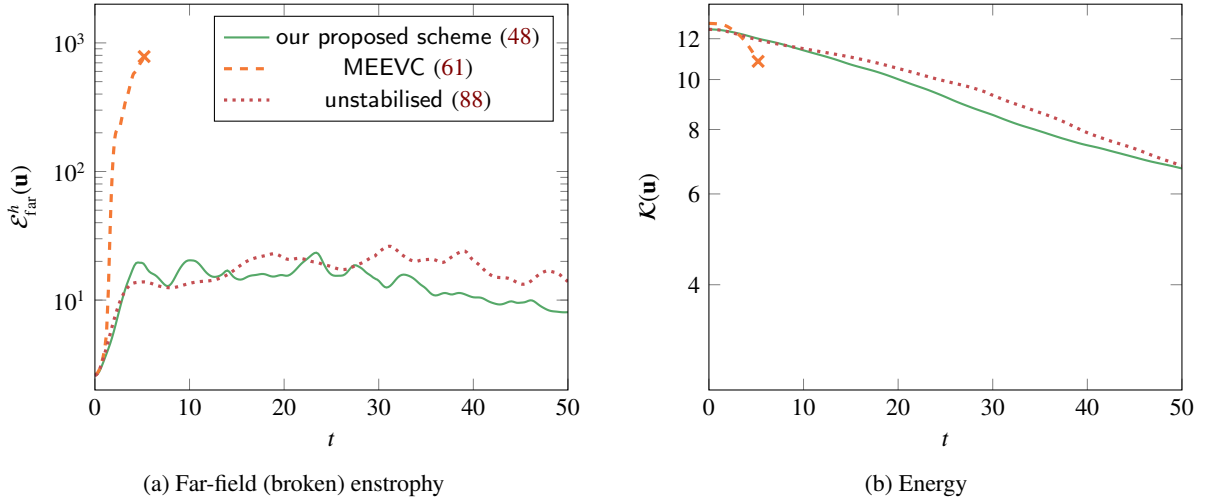

    \centering
    \pgfplotsset{
        common axis style/.style={
            xmin = 0, xmax = 50, xtick distance = 10,
            width = 0.76\linewidth, height = 0.62\linewidth,
            axis on top,
            scale only axis,
            xlabel = {$t$},
        }
    }
    \begin{subfigure}{0.5\textwidth}
        \centering
        \begin{tikzpicture}
        \begin{axis}[
            common axis style,
            ymode = log,
            ymin = 2, ymax = 2000,
            ytick = {10, 100, 1000},
            ylabel = {$\calE^h_{\mathrm{far}}(\bfu)$},
            legend pos = north east,
        ]
            \input{plots/vortex_street/3d/bdf2/far/asf.tex}
            \input{plots/vortex_street/3d/bdf2/far/meevc.tex}
            \input{plots/vortex_street/3d/bdf2/far/standard.tex}
            \legend{our proposed scheme \eqref{eq:semidiscrete_3d_penalty}, MEEVC \eqref{eq:meevc_semidiscrete}, unstabilised \eqref{eq:unstabilised_semidiscrete}}
        \end{axis}
        \end{tikzpicture}
        \caption{Far-field (broken) enstrophy}
        \label{fig:vs_3d_bdf2_far}
    \end{subfigure}%
    \begin{subfigure}{0.5\textwidth}
        \centering
        \begin{tikzpicture}
        \begin{axis}[
            common axis style,
            ymode = log,
            ymin = 2.5, ymax = 14,
            ytick = {4, 6, 8, 10, 12},
            yticklabels = {$4$, $6$, $8$, $10$, $12$},
            ylabel = {$\calK(\bfu)$},
        ]
            \input{plots/vortex_street/3d/bdf2/energy/asf.tex}
            \input{plots/vortex_street/3d/bdf2/energy/meevc.tex}
            \input{plots/vortex_street/3d/bdf2/energy/standard.tex}
        \end{axis}
        \end{tikzpicture}
        \caption{Energy}
        \label{fig:vs_3d_bdf2_energy}
    \end{subfigure}\\

    \caption{
        Evolution of the far-field (broken) enstrophy \eqref{eq:far_enstrophy} and energy in the 3D turbulent wake simulations using BDF2.
    }
    \label{fig:vs_3d_bdf2_data}
\end{figure}

\paragraph{Results (implicit midpoint).}

\Cref{fig:vs_3d_midpoint_data} shows the same quantities using the implicit midpoint method.
While the MEEVC discretisation does reach the final time $t = 50$, its far-field enstrophy $\calE^h_{\mathrm{far}}$ nevertheless undergoes the same early growth, peaking at around $928$ at $t \approx 5.8$, and remains above that of both other schemes throughout the rest of the simulation.
Its energy $\calK$ meanwhile decays to $2.9$ by $t = 50$, less than half that of either other scheme.
Our scheme ends at around $\calE^h_{\mathrm{far}} = 8.9$, against $12.9$ for the unstabilised discretisation, lying below it for around $72\%$ of the simulation and throughout from $t \approx 28$ onwards.

\begin{figure}[pos=!ht]
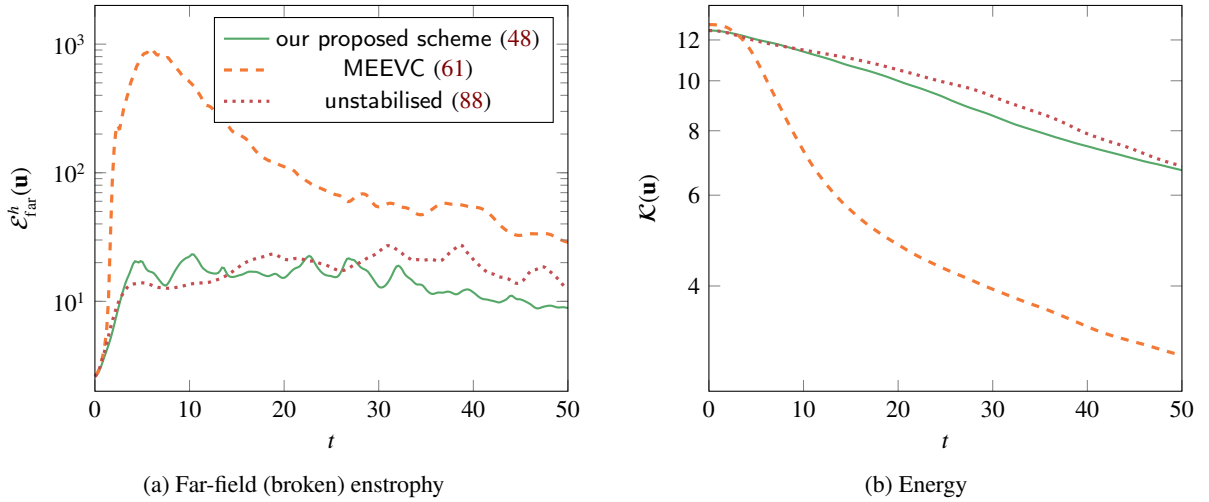

    \centering
    \pgfplotsset{
        common axis style/.style={
            xmin = 0, xmax = 50, xtick distance = 10,
            width = 0.76\linewidth, height = 0.62\linewidth,
            axis on top,
            scale only axis,
            xlabel = {$t$},
        }
    }
    \begin{subfigure}{0.5\textwidth}
        \centering
        \begin{tikzpicture}
        \begin{axis}[
            common axis style,
            ymode = log,
            ymin = 2, ymax = 2000,
            ytick = {10, 100, 1000},
            ylabel = {$\calE^h_{\mathrm{far}}(\bfu)$},
            legend pos = north east,
        ]
            \input{plots/vortex_street/3d/midpoint/far/asf.tex}
            \input{plots/vortex_street/3d/midpoint/far/meevc.tex}
            \input{plots/vortex_street/3d/midpoint/far/standard.tex}
            \legend{our proposed scheme \eqref{eq:semidiscrete_3d_penalty}, MEEVC \eqref{eq:meevc_semidiscrete}, unstabilised \eqref{eq:unstabilised_semidiscrete}}
        \end{axis}
        \end{tikzpicture}
        \caption{Far-field (broken) enstrophy}
        \label{fig:vs_3d_midpoint_far}
    \end{subfigure}%
    \begin{subfigure}{0.5\textwidth}
        \centering
        \begin{tikzpicture}
        \begin{axis}[
            common axis style,
            ymode = log,
            ymin = 2.5, ymax = 14,
            ytick = {4, 6, 8, 10, 12},
            yticklabels = {$4$, $6$, $8$, $10$, $12$},
            ylabel = {$\calK(\bfu)$},
        ]
            \input{plots/vortex_street/3d/midpoint/energy/asf.tex}
            \input{plots/vortex_street/3d/midpoint/energy/meevc.tex}
            \input{plots/vortex_street/3d/midpoint/energy/standard.tex}
        \end{axis}
        \end{tikzpicture}
        \caption{Energy}
        \label{fig:vs_3d_midpoint_energy}
    \end{subfigure}\\

    \caption{
        Evolution of the far-field (broken) enstrophy \eqref{eq:far_enstrophy} and energy in the 3D turbulent wake simulations using implicit midpoint.
    }
    \label{fig:vs_3d_midpoint_data}
\end{figure}

As in the BDF2 case (\Cref{fig:vs_3d_bdf2_energy}), \Cref{fig:vs_3d_midpoint_energy} shows that the energy decays at a similar rate for our proposed discretisation and the unstabilised scheme.
MEEVC has the fastest decay rate of the three schemes.

The velocity profiles and streamlines in \Cref{fig:vs_3d_velocity} compare the state of our proposed scheme \eqref{eq:semidiscrete_3d_penalty}, the unstabilised one \eqref{eq:unstabilised_semidiscrete}, and MEEVC \eqref{eq:meevc_semidiscrete} respectively, at the final time $t = 50$.
The MEEVC discretisation (\Cref{fig:vs_3d_velocity_meevc}) is visibly of lower energy than the other two, in line with \Cref{fig:vs_3d_midpoint_energy}, and exhibits larger jumps between cells.

Our proposed discretisation (\Cref{fig:vs_3d_velocity_asf}) appears marginally more regular in the wake than the unstabilised one (\Cref{fig:vs_3d_velocity_unstabilised}), consistent with its lower far-field enstrophy over the latter part of the simulation (\Cref{fig:vs_3d_midpoint_far}).
One plausible explanation for this is vortex compression.
The vorticity shed from the obstacle lies largely in the plane transverse to the flow direction $\bfe_1$.
As the wake recovers downstream of the obstacle, the streamwise velocity increases ($\partial_x u_1 > 0$) and the flow contracts in the transverse plane ($\partial_y u_2 + \partial_z u_3 < 0$).
We expect vortex tubes lying in this plane therefore generally to be compressed rather than stretched, such that the stretching term $\bfomega \cdot \varepsilon\,\bfu \cdot \bfomega$ of \eqref{eq:enstrophy_stretching_continuous} tends to be negative, and enstrophy is removed from the wake, stabilising the flow.
A discretisation that reproduces this mechanism faithfully would be expected to regularise the wake accordingly.

\begin{remark}[Interpretation of the 3D wake results]\label{rem:vs_3d_interpretation}
    Due to the reduced penalty parameter $\bar{\sigma}$ in the vorticity reconstruction (\Cref{sec:penalty_sigma}), the discretisation used here is not fully enstrophy-stable in the sense of \Cref{th:stability}.
    As with the Hill spherical vortex (\Cref{sec:hill_vortex}), these simulations are therefore primarily a demonstration that the proposed semidiscretisation \eqref{eq:semidiscrete_3d_penalty} functions in a non-trivial 3D setting, and the observations above cannot be attributed with confidence to enstrophy stability.
    A true test would require both (i)~a higher spatial resolution, and (ii)~a weakly conforming Stokes complex, capable of handling the re-entrant corners of the approximated obstacle without resorting to the penalty method (see \Cref{rem:conformity}).
\end{remark}


\begin{figure}[pos=!ht]
    \centering
    \begin{subfigure}{\textwidth}
        \centering
        \includegraphics[width=0.75\textwidth]{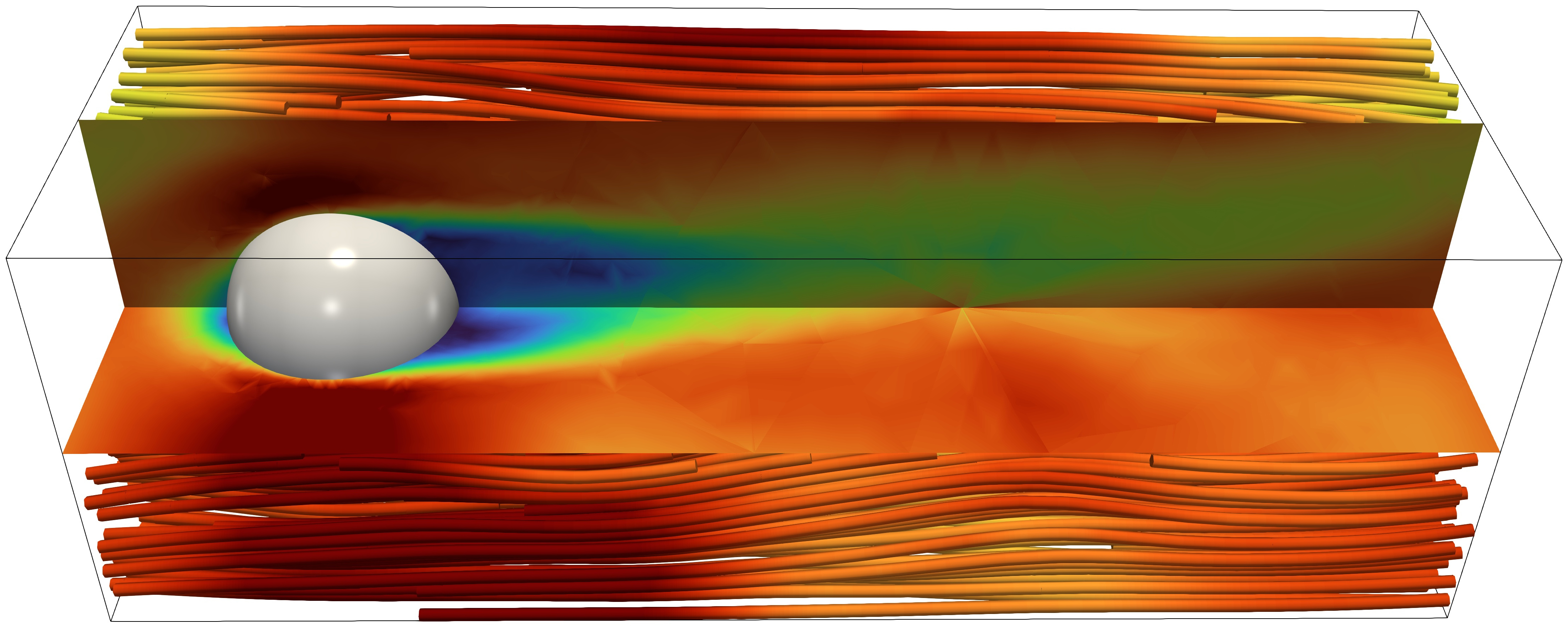}
        \caption{Our proposed scheme \eqref{eq:semidiscrete_3d_penalty}}
        \label{fig:vs_3d_velocity_asf}
    \end{subfigure}
    \\
    \vspace{2mm}
    \begin{subfigure}{\textwidth}
        \centering
        \includegraphics[width=0.75\textwidth]{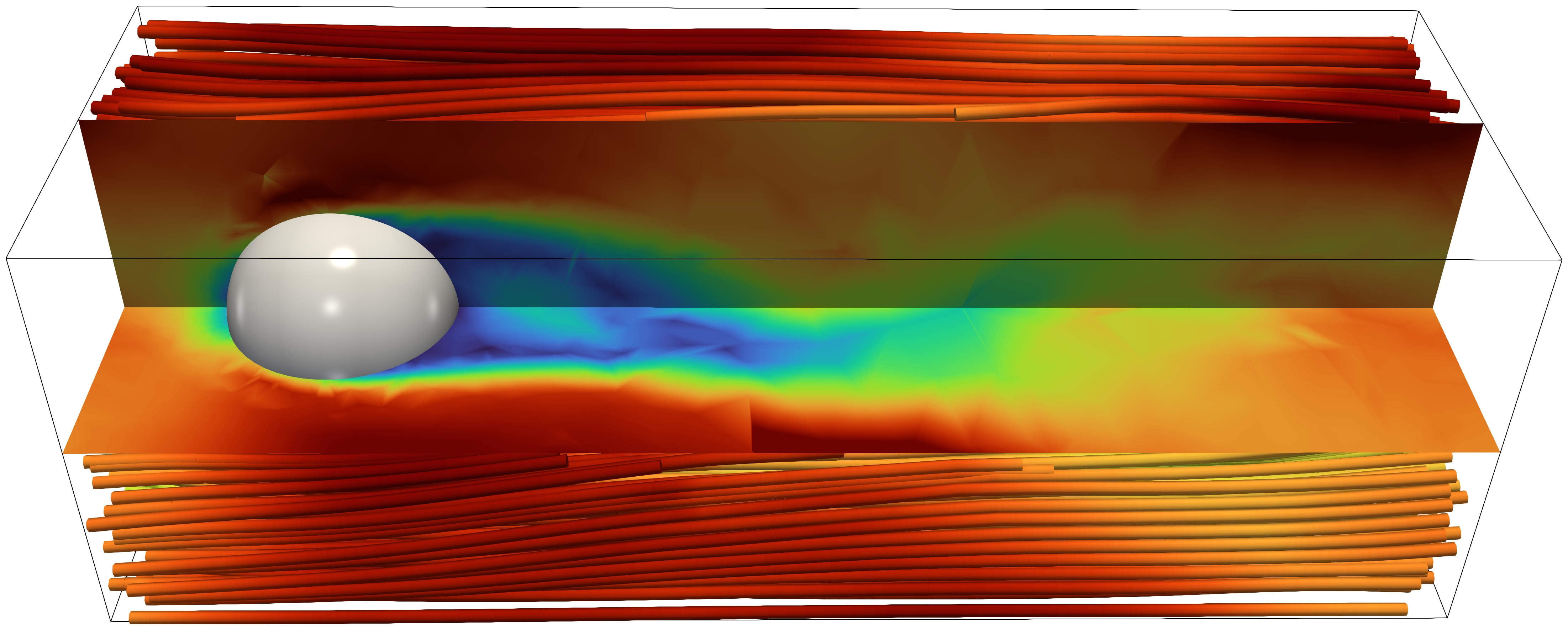}
        \caption{Unstabilised \eqref{eq:unstabilised_semidiscrete}}
        \label{fig:vs_3d_velocity_unstabilised}
    \end{subfigure}
    \\
    \vspace{2mm}
    \begin{subfigure}{\textwidth}
        \centering
        \includegraphics[width=0.75\textwidth]{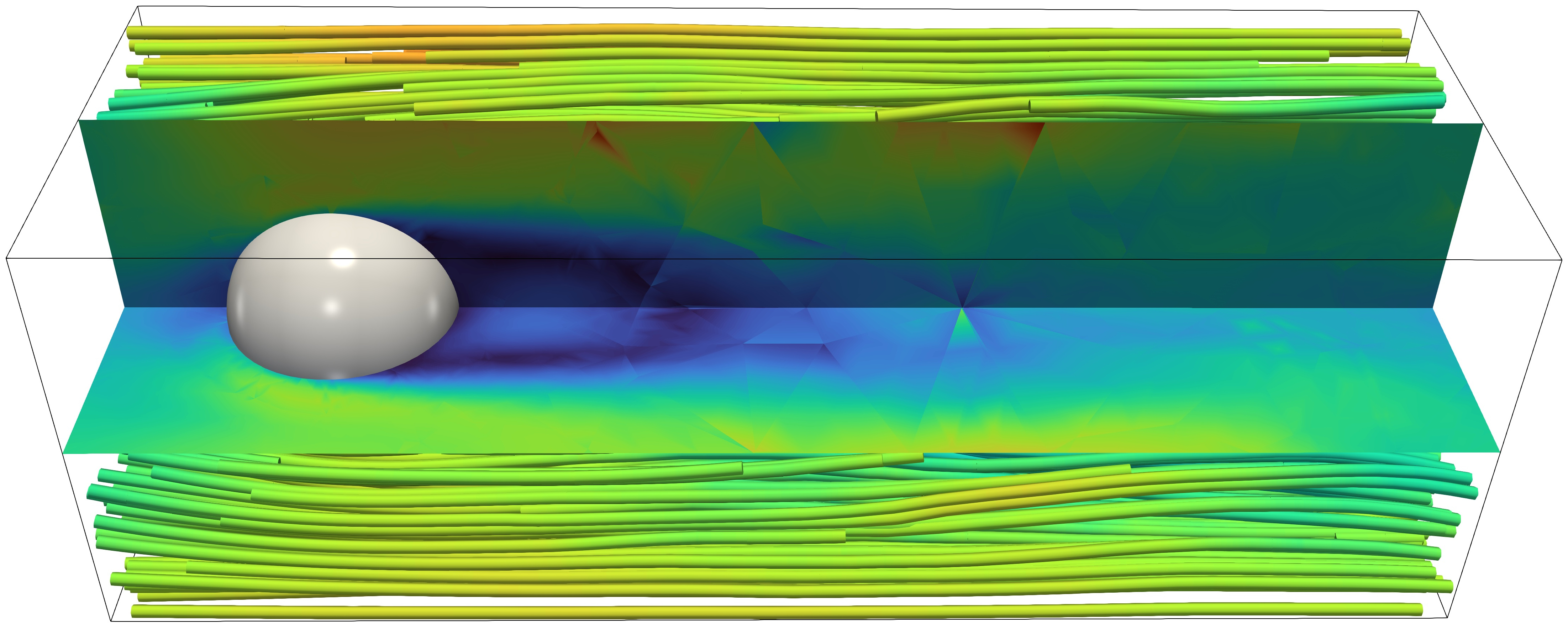}
        \caption{MEEVC \eqref{eq:meevc_semidiscrete}}
        \label{fig:vs_3d_velocity_meevc}
    \end{subfigure}
    \caption{
        Velocity profiles on cross sections of the flow for each of the three simulations of the 3D turbulent wake using implicit midpoint at the final time $t = 50$, coloured from blue ($|\bfu| = 0$) to red ($|\bfu| = 1$), with streamlines in the background.
    }
    \label{fig:vs_3d_velocity}
\end{figure}

\section{Conclusions and future work}\label{sec:conclusions}

We have proposed and analysed a mixed finite element discretisation for the incompressible Navier--Stokes equations \eqref{eq:navier-stokes_1} that preserves the evolution of both energy $\calK$ \eqref{eq:energy} and enstrophy $\calE$ \eqref{eq:enstrophy_both} at the discrete level \eqref{eq:discrete_stability}.
In two dimensions, the boundedness of enstrophy yields a Reynolds-robust bound on $\|\nabla\bfu\|$, providing a natural stabilisation mechanism for under-resolved flows.
Conforming implementations rest on discrete Stokes complexes with $\bfH(\grad\curl)$- or $H^2$-conforming vorticity spaces \eqrefs{eq:stokes_3d,eq:stokes_2d} which are available only in limited polynomial degrees and mesh configurations.
We have therefore introduced two reformulations: a symmetric interior penalty formulation over the standard discrete de Rham complex (\Cref{sec:penalty}), and an $\bfH(\grad\curl)$- or $H^2$-free reparametrisation via auxiliary variables using only standard $\grad$-, $\curl$- and $\div$-conforming spaces (\Cref{sec:charlie}).
A broad range of boundary conditions can be incorporated (\Cref{sec:bcs}).
Numerical experiments (\Cref{sec:simulations}) demonstrate a stabilising effect of enstrophy preservation, in particular in 2D.
Several directions for future work present themselves.

\paragraph{Computational efficiency and preconditioning.}

The auxiliary-variable formulation introduces additional degrees of freedom requiring a nonlinear solve at each timestep.
IMEX splitting to treat the advective term of our semidiscretisation explicitly \eqrefs{eq:semidiscrete_3d,eq:semidiscrete_2d} \citep{Ascher_Ruuth_Wetton_1995,Ern_Guermond_2023}, alongside robust preconditioner strategies \citep{Benzi_Golub_Liesen_2005,Elman_Silvester_Wathen_2014,Farrell_Mitchell_Wechsung_2019}, would reduce the computational cost considerably, bringing the scheme closer to large-scale practical use.
It remains an open question how such split time discretisations would affect the enstrophy stability and corresponding $\Re$-robustness properties.

\paragraph{More viable discrete Stokes complexes (low-degree and curved elements).}

The principal practical limitation of the present approach is the lack of efficient, low-degree discrete Stokes complexes (see \Cref{rem:charlie_overhead}).
Particularly in 3D, existing constructions often rely on supersmoothness, which results in spaces of high polynomial degree and generates large linear systems \citep{Neilan_2015,Fu_Guzman_Neilan_2020}.
Furthermore, curved boundaries (such as the obstacle considered in \Cref{sec:vortex_street}) require special handling when modelled on non-curved meshes (see \Cref{sec:strong_vorticity}).
It remains generally unclear, however, how to transform the shape functions in a discrete Stokes complex onto curved elements without losing the complex property.
The further development and implementation of trimmed Stokes complexes in the spirit of \citet{Guzman_Neilan_2018} and \citet{Hu_Zhang_Zhang_2022}, alongside discrete Stokes complexes that remain valid on curved meshes, would substantially improve the practical reach of the proposed schemes.

\paragraph{Manifolds and non-smooth domains.}

Extension to Riemannian manifolds raises further open questions.
The several possible notions of the vector Laplacian no longer coincide, and the appropriate definition of enstrophy is not immediate.
Discrete Stokes complexes on manifolds without global $C^1$ regularity, such as Regge manifolds \citep{Christiansen_2011,Gawlik_McKee_2026}, remain an ongoing area of research;
non-conforming Stokes elements, e.g.~the Mardal--Tai--Winther element \citep{Mardal_Tai_Winther_2002,Tai_Winther_2006}, represent an important step in this direction.

\paragraph{Spectral and statistical properties.}

The Navier--Stokes equations carry invariant structure beyond energy and enstrophy.
In 2D, the full hierarchy of Casimir invariants (i.e.~all moments of the vorticity) is conserved in the Euler limit $\Re = \infty$.
In both 2D and 3D, the energy spectrum and its cascades (including the Kolmogorov $-5/3$ scaling law) are central to the characterisation of turbulence \citep{Kraichnan_1967,Kolmogorov_1991,Frisch_Kolmogorov_1995}.
Since enstrophy $\calE$ \eqref{eq:enstrophy} is the second moment of the energy spectrum, enstrophy-stable discretisations as proposed here preserve the evolution of this quantity;
it is an open and numerically interesting question whether such schemes may reproduce the full spectral and statistical structure of turbulent flows more faithfully than classical discretisations.

\paragraph{Analysis.}

The existence and convergence of discrete solutions are not addressed here.
In light of recent works using vortex stretching to give finite-time blow-ups for the 3D Navier--Stokes equations \citep{OpenAI_2026b}, a discretisation that faithfully reproduces the vortex stretching term $\int_\Omega \bfomega \cdot \varepsilon\,\bfu \cdot \bfomega$ \eqref{eq:enstrophy_stretching_continuous} in the generation of enstrophy might more reliably reproduce these dynamics.
A similarly interesting open question is whether discretely preserving the enstrophy evolution equation implies analogues of its continuous consequences (local-in-time existence of strong solutions and global existence under small-data conditions).
As above, those consequences typically rest on the vortex stretching form of the enstrophy evolution \citep{Robinson_2020}, which our discretisation largely reproduces (\Cref{th:stretching}).

\printcredits

\bibliographystyle{cas-model2-names}

\bibliography{references}

\end{document}